%% file: main.tex
\documentclass{article} 

\input{preamble}
\ifsiamart
\else\fi

\begin{document}

\maketitle
\begingroup
\renewcommand\thefootnote{\textdagger}
\footnotetext{All authors contributed equally.}
\endgroup

\begin{abstract}
    We study the derivative-free global optimization algorithm Consensus-Based Optimization (CBO),
    establishing uniform-in-time propagation of chaos as well as an almost uniform-in-time stability result for the microscopic particle system. Moreover, we prove almost sure exponential convergence of the microscopic CBO system around a point close to the global minimizer.
    The proof of these results is based on a novel stability estimate for the weighted mean and on
    a quantitative concentration inequality for the microscopic particle system around the empirical mean.
    Our propagation of chaos result recovers the classical Monte Carlo rate,
    with a prefactor that depends explicitly on the parameters of the problem.
    Notably, in the case of CBO with anisotropic noise,
    this prefactor is independent of the problem dimension.
\end{abstract}




\section{Introduction}\label{sec:introduction}

\subsection{Overview}
\label{sub:overview}

As a powerful alternative to gradient-based optimization algorithms,
a number of metaheuristics have been developed to solve highly challenging global optimization problems.
Many algorithms in this family evolve a set of particles that interact and are driven by two forces:
One is a deterministic drift, and the other is a stochastic noise that allows particles to explore the landscape of the objective function and to escape local minima.
Under the umbrella of metaheuristics lie well-known methods such as Simulated Annealing (SA) \cite{MR702485}, Particle Swarm Optimization (PSO) \cite{488968}, or Ant Colony Optimization \cite{585892}.
In this article, we focus on another algorithm known as Consensus-Based Optimization (CBO),
which was proposed relatively recently in~\cite{CBO}.
Like many other metaheuristics,
this method is a gradient-free algorithm which drives particles through a combination of a drift force that exploits available information on the objective function,
and a multiplicative random noise that promotes exploration.
Since CBO does not require gradient evaluations of the objective function,
the method is particularly convenient when the function to minimize is only available as a black box oracle,
or when it is difficult to calculate derivatives of this function efficiently or accurately.

Although global optimization metaheuristics are challenging to analyze rigorously in general,
some pioneering works have been accomplished for simulated annealing \cite{MR942621, MR933455, MR995752, }
and for CBO \cite{carrillo2018analytical,fornasier2024consensus, bonandin2025strongglobalconvergenceconsensusbasedoptimization}.
Unlike PSO, which is more widely used than CBO in the optimization community given it was already introduced 30 years ago,
CBO is an interacting particle system for which mean-field descriptions are accessible.
In the limit where the number of particles tends to infinity,
the spatial configuration of the particles may be described by a probability density evolving according to a deterministic but nonlocal Fokker--Planck equation.
Through such a mean-field approximation,
the works~\cite{carrillo2018analytical,fornasier2024consensus} were the first to examine the convergence properties of CBO in a mathematically rigorous way.
The authors were able to prove that,
under appropriate assumptions including uniqueness of the global minimizer,
the mean-field system
eventually converges to a Dirac distribution at a consensus point close to the global minimizer.
Furthermore, the distance between the consensus point and the minimizer can be controlled in terms of an inverse temperature parameter,
denoted by~$\alpha$ in the following.

Nonetheless, since actual implementations use a finite number of particles and a discrete-time evolution,
the convergence guarantees at the level of the  continuous-time, mean-field equation are not sufficient to ensure the convergence of the algorithm in practice.
Indeed, the mean-field limit of CBO provides us with an averaged, collective description of the system,
which constitutes a faithful description of the particle system only when the number of agents is very large.
The existence of a limiting macroscopic equation as the number of particles increases to infinity is closely related to
the phenomenon whereby any finite group of particles become asymptotically independent in the same limit
-- a property called \emph{propagation of chaos} in the literature.
The recent works~\cite{fornasier2024consensus,gerber2023meanfield,bonandin2025strongglobalconvergenceconsensusbasedoptimization} aim to bridge the gap between the microscopic system (particle regime) and the macroscopic equation (mean-field regime) for the~CBO algorithm. While these works represent a significant step forward in the analysis of the~CBO algorithm, the provided error estimates exhibit an exponential dependence on time. In order to leverage the asymptotic analysis available at the PDE level for understanding the behavior of the CBO algorithm as time goes to infinity, quantitative uniform-in-time error estimates for the mean-field limit are required.
In this paper,
we focus on proving a uniform-in-time (UiT) result for the convergence of continuous-time~CBO to the corresponding mean-field limit.

Two related works were very recently completed in this direction,
see \cite{huang2024uniformintimemeanfieldlimitestimate} and \cite{bayraktar2025uniformintimeweakpropagationchaos}.
Both were able to show the uniform-in-time convergence of the particle system to the mean-field equation.
However, the former work uses a modified version of the CBO algorithm,
coined \emph{rescaled CBO}, comprising an additional convex interaction,
whereas the latter only proves weak propagation of chaos for another modification of the CBO algorithm,
which includes a cut-off to confine the particle system in a bounded domain.
Hence, neither works provides a convergence estimate that holds uniformly in time for the original CBO method.
The main challenge of completing this estimate is in managing the lack of uniform convexity.
To simplify this challenge, these previous two works both modified the algorithm,
whereas we analyze the original method using an approach motivated at the end of~\cref{sec:propagation_chaos}.

In order to obtain convergence guarantees for the discrete-time, finite-ensemble CBO method used in practice,
one should also analyze the error introduced by time discretization of the continuous dynamics.
We do not conduct such an investigation in this paper,
but note that the time discretization error can be controlled in a finite time interval via classical results from numerical analysis~\cite{MR1214374,MR3097957},
as done in~\cite{fornasier2024consensus, bonandin2025strongglobalconvergenceconsensusbasedoptimization}. 
Whether the time discretization error can be bounded uniformly in time is a topic we leave for future work.

An alternative approach to analyzing the convergence of discrete-time CBO by means of a triangle inequality
with the continuous time dynamics as a pivot,
is to study instead the performance of the discrete-time dynamics for optimization tasks directly,
without reference to the continuous-time formulation.
Such an approach is undertaken in~\cite{MR4179193, MR4215338, MR4456850,byeon2024discreteconsensusbasedoptimization}.

In the rest of this introduction,
we provide the basic setting of the CBO algorithm in~\cref{sub:setup},
and provide a brief literature review on propagation of chaos in~\cref{sec:propagation_chaos}.
We then summarize our contributions and present a plan of the paper in~\cref{sub:our_contributions},
and finally introduce key notation for the rest of the paper in~\cref{sec:notations}.

\subsection{Mathematical setting}
\label{sub:setup}

For a given objective function $f\colon \R^d \rightarrow \R$ and  number of particles $J\in\nat_{>0}$,
we consider the following interacting particle system, known as Consensus-Based Optimization (CBO):
\begin{align}
    \label{eq:cbo}
    \d \xn{j}_t = - \Bigl(\xn{j}_t - \wmx{t}\Bigr) \, \d t + \sigma S\bra*{\xn{j}_t - \wmx{t}} \, \d \wn{j}_t\,,
    \qquad j = 1, \dots, J\,,
\end{align}
where $(W_t^j)_{j = 1}^J$ are independent $\R^d$-valued Brownian motions and noise strength $\sigma>0$.

Here, we write $\emp{t} := \frac{1}{J}\sum_{j = 1}^J \delta_{\xn{j}_t}$ to denote the empirical measure of the $J$-particle system at time~$t$,
and use the weighted average operator $\mathcal M_{\alpha} \colon \mathcal P_1(\real^d) \to \real^d$ defined by
\begin{align}
    \wm(\mu) := \frac{\int_{\R^d} x \e^{-\alpha f(x)} \, \mu(\d x)}{\int_{\R^d} \e^{-\alpha f(x)} \, \mu(\d x)} \, ,
\end{align}
for a probability measure $\mu \in \mathcal{P}(\R^d)$.
The noise operator $S\colon \R^d \rightarrow \R^{d \times d}$ is either $S=S^{(i)}$ or $S=S^{(a)}$.
The isotropic noise operator~$S^{(i)}$ is given by $S^{(i)}(x) := \abs{x} I_d$,
and the anisotropic noise operator $S^{(a)}$
given by $S^{(a)}(x) := \diag(x_1, \dots, x_d)$,
i.e., the diagonal matrix with the components of $x$ on the diagonal.
The isotropic noise operator was the first to be studied in the literature~\cite{CBO,carrillo2018analytical},
while anisotropic noise was introduced later to improve the performance of the method for high-dimensional problems~\cite{CBO-high-dim}.
To state our results in a compact manner, it will be useful to define
\begin{align}
    \label{eq:noise-prefactor}
    \tau(S):= \begin{cases}
        d & \text{if } S = S^{(i)} \qquad\text{(isotropic noise)}\,,\\
        1 & \text{if } S = S^{(a)} \qquad\text{(anisotropic noise)}\,.
    \end{cases}
\end{align}

Throughout this paper, we will require the following assumptions on the objective function $f$.
\begin{assumption}
    \label{assumption:bounded}
    The function $f\colon \R^d \to \R$ is bounded from below and above: $\underline f \leq f(x) \leq \overline f$.
\end{assumption}

\begin{assumption}
    \label{assumption:lip}
    The function $f\colon \R^d \to \R$  is globally Lipschitz with constant $L_f$.
\end{assumption}

From~\cite{gerber2023meanfield},
it is known that under these assumptions and on a finite time interval,
the flow of empirical measures $t\mapsto\emp{t}=\frac{1}{J}\sum_{j=1}^J \delta_{\xn{j}_t}$ of the particles given by~\eqref{eq:cbo} converges in an appropriate sense,
in the limit~$J \to \infty$ of infinitely many particles,
to the McKean--Vlasov process governed by the following equation:
\begin{equation}\label{eq:mfl_sde_first}
    \left\{
        \begin{aligned}
            \d\widebar{X}_t &= - \bigl(\xl_t - \wm(\mfldis)\bigr) \, \d t + \sigma S\bra*{\widebar{X}_t -\wm(\mfldis)} \, \d W_t \\
            \mfldis_t &= {\rm Law}(\xl_t).
        \end{aligned}
    \right.
\end{equation}
Furthermore, the law $(\mfldis_t)_{t \geq 0}$ is a solution to the following nonlinear, nonlocal Fokker--Planck equation:
\begin{align}
    \label{eq:mfl_pde_first}
    \partial_t\mfldis_t = \nabla \cdot \bra[\Big]{ \bigl(x - \wm(\mfldis_t) \bigr) \mfldis_t } + \frac{\sigma^2}{2} \nabla \cdot \nabla \cdot \bra[\Big]{
    D(\mfldis_t, x) \mfldis_t }\,,
\end{align}
where $D(\mfldis, x) := S\bra[\big]{x - \wm (\mfldis)} S\bra[\big]{x - \wm (\mfldis)}^\t$.

\subsection{Propagation of chaos}\label{sec:propagation_chaos}

Propagation of chaos refers to the property of some interacting particle systems whereby the particles decouple asymptotically as the number of agents tends to infinity~\cite{kac1956foundations, ReviewChaintronI,ReviewChaintronII}.
In order to prove propagation of chaos, one usually assumes that particles are initially independent and then shows that,
in the large particle limit, they are asymptotically independent also for later times.
Thus, the initial chaos is propagated forward in time.
The main focus of this work is on proving that the mean-field limit for CBO holds uniformly in time.
In this section, we first briefly review a few of the milestones in the vast literature on mean-field limits, then highlight recent works in this area that are specifically concerned with CBO, and finally give a brief description of the approach we follow in this manuscript.

A classical approach to prove quantitative mean-field limits is by synchronous coupling as proposed by McKean~\cite[Theorem 3.1]{ReviewChaintronII} and Sznitman~\cite{MR1108185}. Under appropriate convexity assumptions,
Sznitman's method can be extended to prove uniform-in-time estimates using ideas due to Malrieu and collaborators~\cite{MR1847094,MR2731396}.
Using reflection or sticky couplings, uniform-in-time mean-field limits can be shown for certain non-convex confinement and interaction potentials~\cite{Eberle_2015, MR4163850, MR4489825, sticky_coupling}.
Another active line of work, pioneered by D.\ Lacker and L.\ Le Flem \cite{MR4595391, MR4634344},
is based on an appropriate form of the BBGKY hierarchy and was recently extended to dynamics with non-constant diffusion coefficients~\cite{grass2024sharppropagationchaosmckeanvlasov}.

There is also a large body of works on mean-field limits (mostly non-uniform in time) in the presence of irregular or even singular interactions.
Some of these works extend the classical synchronous coupling approach to more singular interactions \cite{MR2860672}, see also \cite[Section 3.1.2]{ReviewChaintronII} for other works.
Let us also mention the modulated energy approach~\cite{MR4158670, MR4462479, MR4564418},
as well as results proved through an entropy-based approach by D. Lacker, P. E. Jabin, Z. Wang and others~\cite{10.1214/18-ECP150,jabir2019rate,MR3858403,MR4632269,MR4658923},
which were extended to the uniform-in-time setting in ~\cite{guillin2024uniform}.
For a thorough review of methods and applications of propagation of chaos, we refer the reader to the review papers~\cite{ReviewChaintronI,ReviewChaintronII}.

Given a finite-time propagation of chaos estimate,
a general approach to proving uniform-in-time propagation of chaos consists of combining the finite-time estimate
with an exponential contractivity estimate for the mean-field system.
A recent work~\cite{schuh2024conditionsuniformtimeconvergence} provides a unifying framework for this strategy,
with applications not only to propagation of chaos,
but also to averaging of fast/slow multiscale systems and time discretization of SDEs through numerical approximation.
We also refer to \cite{gerber2025uniformintimepropagationchaoscuckersmale},
where a similar strategy is deployed to prove uniform-in-time propagation of chaos for the Cucker--Smale model.

The framework provided in~\cite{schuh2024conditionsuniformtimeconvergence} could potentially be used in the context of CBO,
and may also prove useful to obtain uniform-in-time bounds on the discretization error for CBO in future work.
However, as the present manuscript demonstrates,
proving a uniform-in-time stability estimate for the CBO interacting particle system presents a level of difficulty similar to proving uniform-in-time propagation of chaos directly.
Therefore, we shall take a more direct and self-contained approach,
which is based on the classical synchronous coupling method from Sznitman~\cite{MR1108185}
and ideas from the work of Malrieu~\cite{MR1847094,MR2731396}. The rates in Malrieu~\cite{MR1847094} have been more recently improved by D. Lacker's group \cite{MR4634344} using the BBGKY hierarchy to obtain sharp uniform in time propagation of chaos results. This approach cannot directly be applied to the CBO algorithm and in particular, requires the noise coefficient to be non-vanishing, whereas our analysis requires the noise coefficient to be small enough.
This approach cannot directly be applied to the CBO algorithm since it requires the diffusion coefficient to be constant as in~\cite{MR4634344} or to be uniformly elliptic as in the finite-in-time result~\cite{grass2024sharppropagationchaosmckeanvlasov} for non-constant diffusion coefficients. In contract, the CBO algorithm has a non-constant vanishing diffusion coefficient, and for our analysis, we assume the multiplicative constant in front of that coefficient to be sufficiently small.

In its basic form, Sznitman's approach is applicable to an SDE with drift and diffusion coefficients that are globally Lipschitz continuous.
However, the drift and diffusion coefficients of the CBO dynamics are in general merely locally, not globally Lipschitz continuous,
which precludes the direct application of the classical synchronous coupling argument by Sznitman
to prove local-in-time propagation of chaos estimates.
This issue is circumvented in~\cite{gerber2023meanfield} by discarding an event of small probability in the main part of the analysis,
and appropriately controlling the probability of this event using an elementary concentration inequality.
A similar approach has been used previously in~\cite{Kalise_2023} for a variant of CBO based on jump processes,
but with a suboptimal rate of $\ln\bigl(\ln(J)\bigr)^{-1}$. 
There are several other works that investigate the mean-field limit for the original CBO dynamics,
but none of them proves quantitative, uniform-in-time propagation of chaos.
For instance, a non-quantitative mean-field result was shown in~\cite{CBO-mfl-Huang2021} using a compactness argument,
then adapted in~\cite{koß2024meanfieldlimitconsensus} to cover a more general class of~SDEs including consensus-based sampling~\cite{CBS-Carrillo2021}.
We also mention a partial finite-time propagation of chaos estimate from \cite{fornasier2024consensus},
where a quantitative result with the optimal Monte Carlo rate is obtained, but only if an event of small probability is discarded from the expectations. Finally, in the works of \cite{MR4468621, MR4204858, MR4329816}, particles are constrained to a compact manifold, over which the authors apply the usual synchronous coupling method and obtain a quantitative estimate with optimal rates.
A summary of finite-time mean-field results for CBO is presented in \cref{table:comparison_mfl_CBO}.

\renewcommand{\arraystretch}{1.2}
\begin{table}[h!]
\begin{center}
\begin{tabular}{c|c|c|c}
 & Result
 & Rate
 & Approach
\\
\hline
\cite{CBO-mfl-Huang2021, koß2024meanfieldlimitconsensus} &
Non-quantitative, finite-time & N/A & Compactness argument
\\
\hline
\cite{fornasier2024consensus}
& Semi-quantitative, finite-time
& Optimal, $J^{-\frac{1}{2}}$
& Synchronous coupling
\\
\hline
\cite{gerber2023meanfield}
& Quantitative, finite-time
& Optimal, $J^{-\frac{1}{2}}$
& Synchronous coupling
\\
\hline
\cite{Kalise_2023}
& Quantitative, finite-time
& Sub-optimal, $\ln(\ln(J))^{-\frac{1}{2}}$
& Synchronous coupling
\\
\hline
\cite{MR4468621, MR4204858, MR4329816}
& Quantitative, finite-time
& \makecell{Optimal, $J^{- \frac{1}{2}}$}
& \makecell{Synchronous coupling on the sphere}
\\
\end{tabular}
\end{center}
\caption{
    Comparison of finite-time mean-field limit results for CBO.
    The rates given refer to the rates of convergence as the number~$J$ of particles tends to infinity,
    of the Euclidean Wasserstein distance between the law of the~$J$-particle system and the $J$-times tensorized mean-field law,
    in presence of the normalization as in~\cite[Definition 3.5]{ReviewChaintronI} in the definition of the Wasserstein distance.
}
\label{table:comparison_mfl_CBO}
\end{table}

Extending local-in-time propagation of chaos estimates for CBO to the uniform-in-time setting is not straightforward,
because, at first sight, the CBO dynamics does not appear to exhibit sufficient convexity to deploy the approach of Malrieu.
In the recent work~\cite{huang2024uniformintimemeanfieldlimitestimate} mentioned in~\cref{sub:overview},
this issue is circumvented by modifying the CBO algorithm through a rescaling which,
if one looks at the drift, amounts to adding a convex confinement potential;
more precisely, the drift term $- \bigl(\xn{i} - \wm(\emp{t}) \bigr)$ of the original method is replaced by
$- \kappa \bigl(\xn{i} - \wm\bigl(\emp{t}\bigr) \bigr) - (1 - \kappa) \xn{i}$ for some $\kappa \in (0, 1)$.
The work~\cite{bayraktar2025uniformintimeweakpropagationchaos} takes a different approach,
by confining the particles to a manifold through truncation and showing only a weak type of propagation of chaos.
This is similar in spirit to the finite-time results from \cite{MR4468621, MR4204858, MR4329816},
but now achieving uniform-in-time estimates. 
A summary of uniform-in-time (UiT) mean-field results for CBO is presented in \cref{table:comparison_mfl_uit_CBO}.

\renewcommand{\arraystretch}{2.0}
\begin{table}[h!]
    \begin{center}
    \begin{tabular}{c|c|c|c}
     & Result
     & Rate
     & Approach
    \\
    \hline
    \cite{huang2024uniformintimemeanfieldlimitestimate}
    & Quantitative, UiT
    & \makecell{Optimal, $J^{- \frac{1}{2}}$ \\ in Wasserstein-2 metric}
    & \makecell{Synchronous coupling for modified CBO algorithm}
    \\
    \hline
    \cite{bayraktar2025uniformintimeweakpropagationchaos}
    & Quantitative, UiT
    & \makecell{Optimal, $J^{- \frac{1}{2}}$ \\ in a weak convergence metric}
    & \makecell{Modification of the algorithm to \\ confine particles on a compact manifold}
    \end{tabular}
    \end{center}
    \caption{Comparison of uniform-in-time mean-field limit results for CBO. }
\label{table:comparison_mfl_uit_CBO}
\end{table}

To motivate our approach in this paper,
notice that in the simple case where~$\alpha = 0$ and $\sigma = 0$,
the CBO dynamics~\eqref{eq:cbo} may be rewritten in terms of a convex interaction potential:
\begin{equation}
    \label{eq:cbo_simple}
    \d X_t^j
    = - \frac{1}{J} \sum_{k=1}^{J}\Bigl(\xn{j}_t - \xn{k}_t \Bigr) \, \d t
    = - \nabla W \star \emp{t} (\xn{j}_t), \qquad W(x) := \frac{|x|^2}{2}.
\end{equation}
Using the approach developed by Malrieu in~\cite{MR1847094,MR2731396},
which builds upon Sznitman's synchronous coupling method,
it is relatively straightforward to prove uniform-in-time propagation of chaos for this simple system.
The CBO dynamics is of course more difficult to analyze,
but the validity of uniform-in-time propagation of chaos for~\eqref{eq:cbo_simple} (also when adding additive noise) is a good indication that
a similar result should hold for CBO more generally,
at least for sufficiently small $\sigma$,
which is precisely what we show in this paper.
The key idea of our approach is to view the CBO drift term
$- \bigl(\xn{j}_t - \wm\bigl( \emp{t} \bigr)\bigr)$
as a perturbation of the drift $- \bigl(\xn{j}_t - \nm\bigl( \emp{t} \bigr)\bigr)$ from~\eqref{eq:cbo_simple},
which exhibits the convexity of the interaction,
and to control the contributions of the remainder terms,
which involve the difference of $\nm$ and $\wm$,
using novel stability and concentration estimates.

\subsection{Our contributions}
\label{sub:our_contributions}

The main contribution of this work is a rigorous proof of uniform-in-time propagation of chaos for the CBO interacting particle system \eqref{eq:cbo},
without any modification to the original algorithm~\cite{carrillo2018analytical}.
Using a similar strategy,
we also prove a uniform-in-time stability estimate for the CBO interacting particle system.
Our proof follows the synchronous coupling approach by Sznitman and McKean \cite[Theorem 3.1]{ReviewChaintronII} and relies on a number of novel auxiliary results, such as concentration estimates for the interacting particle  system as well as a local Lipschitz estimate for the map $\mu \mapsto \wmmu-  \nm(\mu)$, where $\wmmu$ and $\nm(\mu)$ denote the weighted and usual means for the probability measure $\mu$ respectively.

To make the presentation of this paper as simple and self-contained as possible,
we focus exclusively on the case of bounded, globally Lipschitz continuous objective functions.
We like this setting because it enables to track the constant prefactors explicitly,
and to exhibit their dependence on parameters such as the problem dimension and the temperature parameter.
Extension of our results to more general cost functions is left for future work.

\paragraph{Plan of the paper}
The rest of the paper is organized as follows.
After presenting the key notation in~\cref{sec:notations},
we state the main results of this work in~\cref{sec:main}.
These results are then proved rigorously in~\cref{sec:proofs_main},
and the auxiliary results on which they rely are stated precisely and proved in~\cref{sec:aux}.
We conclude with possible future directions in~\cref{sec:conclusion}.
Finally, the Burkholder--Davis--Gundy inequality with explicit constants is recalled in~\cref{sec:bdg},
and a summary of the constants that appear in the key estimates of this work, as well as their dependence on method and problem parameters, are presented in~\cref{sec:constants}.

\small
\subsection{Notation}
\label{sec:notations}
\begin{itemize}[leftmargin=*]
    \item
        The Euclidean distance in $\real^d$ is denoted by~$| \placeholder |$.
        The notation $\| \placeholder \|_{\rm F}$ denotes the Frobenius norm on matrices.

    \item
        For a random variables $X$,
        the notation~$\expect X$ or~$\expect [X]$ denotes its expected value.
        We give the symbol~$\expect$ a precedence lower than exponents,
        so that expressions such as $\expect |X|^2$ and $\expect (\e^{X})^{\frac{1}{2}}$
        are short-hand notations for~$\expect \bigl[ |X|^2 \bigr]$ and $\expect \bigl[ (\e^{X})^{\frac{1}{2}} \bigr]$,
        respectively.

    \item
        The notation~$\mathcal P(\real^d)$ denotes the space of probability measures on $\real^d$,
        and the notation~$\mathcal P_p(\real^d)$ denotes the subset of probability measures $\mu \in \mathcal P(\real^d)$ with finite moments up to order~$p$.
        Furthermore, for joint probability measures $\rho^J \in \mathcal P(\real^{dJ})$ of $J$ particles in $\R^d$, $\mathcal P_{\rm sym}(\real^{dJ})$ denotes the subset of joint laws for which particles are exchangeable, i.e.\ probability measures $\rho^J$ that remain invariant under permutation of their $J$ variables.

    \item
        The notation $\wasserstein_p$ denotes the standard Wasserstein-$p$ distance.

    \item
        For a probability measure~$\mu \in \mathcal P_1(\real^d)$,
        the notation $\nm (\mu)$ denotes the usual mean under~$\mu$, hence,
    \[
        \nm (\mu) = \int_{\mathbb{R}^d} x \, \mu(dx) \,.
    \]
    
    \item
        We write $\allx{J}_t = (\xn{j}_t)_{j=1}^{J}$,
        and similarly $\allxl{J}_t = (\xnl{j}_t)_{j=1}^{J}$
        and $\allxt{J}_t = (\xnt{j}_t)_{j=1}^{J}$.
    
    \item 
        We denote by~$\cptmfl \in \mathbf{R}^d$ the consensus point of the mean-field dynamics.
        Given $J \in \mathbf N\,,$ we define its (random) finite particle counterpart $\cpt \in \mathbf{R}^d$ as the consensus point of $\allx{J}_t\,.$ Both $\cptmfl$ and $\cpt$ depend on the parameter $\alpha$.
    
    \item
        For a collection of positions $\allx{J}$ in $\real^d$,
        we denote by $\mu_{\allx{J}}$ the associated empirical measure.
        In particular
        \[
            \emp{t} := \frac{1}{J}\sum_{j = 1}^J \delta_{\xn{j}_t}\,  \quad \text{ and } \quad
            \empl{t} := \frac{1}{J}\sum_{j = 1}^J \delta_{\xnl{j}_t}\,.
        \]

    \item
        For a probability measure~$\mu \in \mathcal P_1(\real^d)$,
        we use the following notation for the central and raw moments under $\mu$:
        \[
            \label{eq:notation_mp}
            {p}\momentp{p}{\mu} := \int_{\real^d} \Bigl\lvert x - \mathcal M(\mu) \Bigr\rvert^p \, \mu(\d x)
            \quad \text{and} \quad
            \rawmomentp{p}{\mu}:= \int_{\real^d} \lvert x \rvert^p \, \mu(\d x).
        \]
        Recall that $\momentp{2}{\mu} \leq \rawmomentp{2}{\mu}$,
        and more generally $\momentp{p}{\mu} \leq 2^{p} \rawmomentp{p}{\mu}$ for $p\ge 1$.
        \item Many parameters will depend on the following rate: 
    \begin{align}
        \label{eq:intro:exp-decay-factor}
        \edecay{p} :=   p \left[1-  \frac{1}{2}\bigl(p-2+\tau(S) \bigr) \sigma^2 \Bigl(1 + \e^{\frac{\alpha}{p}(\overline f - \underline f)} \Bigr)^2\right] .
    \end{align}
\end{itemize}
\normalsize

\section{Main results}
\label{sec:main}

In this section,
we present our main results,
the proofs of which can be found in~\cref{sec:proofs_main}.
First, in~\cref{sub:uniform_chaos},
we establish uniform-in-time propagation of chaos for the CBO interacting particle system~\eqref{eq:cbo},
which constitutes the central theorem of this work.
Next, in~\cref{sub:uit_stability},
we derive a stability result for the same system. 
Finally, in~\cref{sub:exp_concentration_as},
we state a result concerning the exponential concentration of the interacting particle system around the consensus point,
and using~\cref{thm:uit-mfl},
we quantify the expected distance between the consensus point of the interacting particle system and that of the mean field dynamics.

\subsection{Uniform-in-time propagation of chaos}
\label{sub:uniform_chaos}

The mean-field results from \cite{gerber2023meanfield, fornasier2024consensus} did not make use of the contractive properties of the dynamics \eqref{eq:cbo},
yielding estimates that are useful only in finite time.
Inspired by the uniform-in-time mean-field limit from \cite{MR1847094}, we are able to show uniform-in-time propagation of chaos for CBO.
The main result is the following, 
the proof of which can be found in~\cref{sec:proof:thm:uit-mfl}.

\begin{theorem}
    \label{thm:uit-mfl}
    Fix a probability measure $\mfldis_0\in \mathcal{P}\bra{\R^d}$ with finite moments of all orders.
    Let $(\Omega, \mathcal{F}, \proba)$ be a probability space supporting initial i.i.d.\ positions~$\bigl( \xn{j}_0 \bigr)_{j\in\N}$ with common law $\mfldis_0$,
    as well as independent standard $d$-dimensional Brownian motions $\bigl( \wn{j}_t \bigr)_{j\in\N}$.
    Assume that $f$ satisfies~\cref{assumption:bounded,assumption:lip}.
    For each $J\in\N$, consider the particle system
    \begin{align}
        \xn{j}_t = \xn{j}_0 - \int_0^t \Bigl( \xn{j}_s - \wmx{s}\Bigr) \, \d s + \sigma \int_0^t S\Bigl(  \xn{j}_s - \wmx{s}  \Bigr)\, \d \wn{j}_s,
    \end{align}
    where $j \in \{1, \ldots, J\}.$
    To this system we couple the system of i.i.d.\ mean-field particles
    \begin{align}
        \label{eq:thm-mfl:synchronously-coupled-system}
        \xnl{j}_t = \xn{j}_0 - \int_0^t \left( \xnl{j}_s - \wmeanmfl{s} \right) \, \d s + \sigma \int_0^t  S\bra*{ \xnl{j}_s - \wmeanmfl{s} } \, \d \wn{j}_s,
    \end{align}
    where $j \in \{1, \ldots, J\}$ and $\mfldis_s = \operatorname{Law}\bigl( \xnl{j}_s \bigr)$.
    The mean field particles are initialized at the same positions and driven by the same Brownian motions as $(\xn{j}_t)_{j\in\N}$.
    Assume that $\sigma \in [0,\widetilde{\sigma})$, where
    \begin{align}
        \label{eq:noise-upper-bound}
        \widetilde{\sigma} \coloneq \frac{\sqrt{2}}{\sqrt{6+ 3\tau(S)}\bra*{1+ \e^{\frac{\alpha}{2}\bra*{\overline{f}-\underline{f}}}}}.
    \end{align}
    Then there exists a finite constant $C_{\rm MFL}$ such that
    \begin{align}
        \forall t \geq 0, \qquad
        \forall J \in \nat_{>0}, \qquad
        \expect \pra*{ \abs{ \xn{j}_t - \xnl{j}_t}^2} \le
        \frac{C_{\rm{MFL}} }{J}\,,
    \end{align}
    where $C_{\rm MFL}:= \e^{ {2 c_1}} {2 c_2}$, and the constants $c_1, c_2>0$ are defined in \eqref{eq:constants-c1} and \eqref{eq:constants-c2}, see also \cref{sec:constants}.
\end{theorem}

\begin{remark}
    For anisotropic noise, the constant $C_{\rm MFL}$ does not depend on the dimension $d$, while it does for isotropic noise. Note that even in the case of isotropic noise, one may choose a sufficiently small noise coefficient $\sigma$ to off-set this dimension dependence.
\end{remark}

\subsection{Almost uniform-in-time stability for the interacting particle system}
\label{sub:uit_stability}

We proceed to show an almost uniform-in-time stability result for the interacting particle system~\eqref{eq:cbo}.
For the proof see~\cref{sub:proof-stab-particle}.

\begin{theorem}
    \label{thm:stab_particle}
    Assume that $\sigma \in [0,\widetilde{\sigma})$ with $\widetilde{\sigma}$ defined as in \eqref{eq:noise-upper-bound},
    and that~$f$ satisfies \cref{assumption:bounded,assumption:lip}.
    Consider two copies~$(\xn{j}_t)_{j=1}^J$ and~$(\xnt{j}_t)_{j=1}^J$ of the particle system~\eqref{eq:cbo} driven by the same Brownian motions~$(\wn{j}_t)_{j=1}^J$ but with possibly different i.i.d.\ initial conditions.
    More precisely, $\bra[\big]{\xn{j}_0}_{j \in \nat}$ are drawn i.i.d.\ from some $\rho_0 \in \mathcal P_{8q}(\real^d)$ and $\bra[\big]{\xnt{j}_0}_{j\in \nat}$ are drawn i.i.d.\ from $\widetilde{\rho}_0 \in \mathcal P_{8q}(\real^d)$ for some $q \geq \frac{1}{2}$.
    Then there exist finite constants~$C_{\rm Stab, 1}, C_{\rm Stab, 2}$ independent of~$J$ such that for all $t\ge 0$ and $J \in \nat_{>0}$, we have
    \begin{align}
        \expect \pra*{ \frac{1}{J} \sum_{j=1}^{J} \abs*{ \xn{j}_t - \xnt{j}_t }^2}
         & \leq  C_{\rm Stab, 1} \expect \pra*{ \frac{1}{J} \sum_{j=1}^{J} \abs*{ \xn{j}_0 - \xnt{j}_0 }^2}
        +
        \frac{ C_{\rm Stab, 2}}{J^q},
    \end{align}
    where $C_{\rm Stab, 1} = \exp \left( \frac{16 \tilde{c}_1}{\edecay{8} } \right) $ and $C_{\rm Stab, 2} = \frac{16 \tilde{c}_2}{\edecay{8} } \exp \left( \frac{16 \tilde{c}_1}{\edecay{8} } \right)$ for $\tilde{c}_1, \tilde{c}_2$ defined in \eqref{eq:constants-cstab1} and \eqref{eq:constants-cstab2},
    and $\lambda_8$ as defined in~\eqref{eq:intro:exp-decay-factor}, see also \cref{sec:constants}.
\end{theorem}

\subsection{Exponential concentration of the particle system close to target}
\label{sub:exp_concentration_as}

In this section, we prove concentration around a consensus point close to the target directly on the particle level,
complementing the results~\cite{carrillo2018analytical,fornasier2024consensus} for the mean-field dynamics. It is by now well-established that, under suitable conditions, the mean-field CBO dynamics \eqref{eq:mfl_sde_first} concentrate at a mean-field consensus point $\cptmfl\in \R^d$ \cite[Theorems~4.1 and Remark~4.3]{carrillo2018analytical}: For well-behaved objectives $f$, small enough $\sigma>0$ and large enough $\alpha>0$, it holds that $\momentp{2}{\mfldis_t}\to 0 $ and $\meanmfl{t}\to \cptmfl$ in the limit as~$t\to \infty$. Further, the Laplace principle enables to prove that the mean-field consensus point $\cptmfl$ approaches the global minimizer of the objective function $f$ as $\alpha$ increases~\cite[Theorems 4.2]{carrillo2018analytical}: For well-behaved objectives $f$ attaining a unique global minimum at $x_{\text{min}} \coloneq \argmin_{x\in\R^d} f(x) \in \supp \mfldis_0$ and for any given $0<\eps\ll 1$, there exist parameters~$\sigma>0$ and $\alpha > 1$ such that $\cptmfl\in B_\eps(x_{\text{min}})$. In this section, we show almost sure concentration around a consensus point $\cpt\in\R^d$ for the particle dynamics \eqref{eq:cbo}, and prove that $\cpt$ converges in $L^2(\Omega)$ to the mean-field consensus point $\cptmfl$ as the number of particles $J$ increases, provided that consensus-formation holds for the mean-field dynamics, that is to say under the assumption that $\wasserstein_2\bra[\big]{\mfldis_t, \delta_{\cptmfl}} \to 0$ as $t \to \infty$. 
The proof of the following theorem can be found in \cref{sec:proof:microscopic-concentration-around-consensus-point}.

\begin{theorem}
    \label{thm:microscopic-concentration-around-consensus-point}
    Fix $J\in\N$ and consider~\eqref{eq:cbo} with i.i.d. initial conditions sampled from some $\mfldis_0\in \mathcal P_4(\R^d)$.
    Assume that $\sigma >0$ is sufficiently small to guarantee that $\edecay{2},\edecay{4}>0$, with $\edecay{2}, \edecay{4}$ as defined in \eqref{eq:intro:exp-decay-factor}.
    Then there exists an $\R^d$-valued random variable $\cpt$ such that $\lim_{t \to +\infty} \meanx{t} = \cpt$ almost surely.
    In addition, the following statements hold:
    \begin{enumerate}[label=(\alph*)]
        \item\label{item:almost-sure-exponential-concentration} 
            (\textbf{Almost sure exponential concentration around the consensus point})            
            For all $\gamma\in (0,\frac{\edecay{2}}{2})$ and all~$\widetilde \gamma \in (0,\min\{\frac{\edecay{2}}{2}, \frac{\edecay{4}}{4} \})$,  
            it holds almost surely that
               \begin{subequations} \label{eq:exp-conv-of-emp-mean-and-particle}
                  \begin{align}
                      \label{eq:as_conv_mean}
                       & \lim_{t \to \infty} \e^{\gamma t} \abs*{\meanx{t}-\cpt} = 0  \\
                      \label{eq:as_conv_wmean}
                       & \lim_{t \to \infty} \e^{\widetilde \gamma t} \abs*{\wmeanx{t}-\cpt} = 0 \\
                      \label{eq:as_conv_particles}
                       & \lim_{t \to \infty} \e^{\widetilde \gamma t} \abs*{\xn{j}_t - \cpt} = 0.
                  \end{align}
              \end{subequations}
              In particular, we have almost sure convergence of the empirical mean towards the consensus point $\cpt$.
        \item\label{item:mean-square-exponential-concentration} 
            (\textbf{Mean square exponential convergence to the consensus point})
            It holds for any~$\gamma \in (0, \edecay{2})$ that
            \begin{subequations}
                \label{eq:L2-conv}
                \begin{align}
                    \label{eq:L2-conv-means}
                          &\expect \abs*{\meanx{t} - \cpt}^2 \le C_{\gamma} \e^{- \gamma t}, \\
                    \label{eq:L2-conv-wmeans}
                          &\expect \abs*{\wmx{t} - \cpt}^2  \le 2\Bigl(\e^{\alpha (\overline{f}-\underline{f})} \expect\empmomentp{2}{0} + C_{\gamma} \Bigr) \e^{- \gamma t}, \\
                          \label{eq:L2-conv-particles}
                          &
                          \expect \abs*{\xn{j}_t - \cpt}^2 \le 2 \Bigl(\expect\empmomentp{2}{0} + C_{\gamma}\Bigr) \e^{- \gamma t },
                \end{align}
            \end{subequations}
            where
            \begin{align}
                C_{\gamma} \coloneq  
                \left(\frac{2\e^{\alpha (\overline{f}-\underline{f})}}{\gamma (\edecay{2} - \gamma)}  
                + 
                \frac{4\sigma^2\tau(S)}{\edecay{2} J} \bra*{1 + \e^{\alpha (\overline f- \underline f)}} \right)\expect \empmomentp{2}{0}  \, .
        \end{align}

        \item\label{item:consensus-point-close-to-target} (\textbf{Mean-field behavior of the consensus point})
            Consider $\cpt$ the consensus point of the particle dynamics~\eqref{eq:cbo}, and assume that the mean-field dynamics 
            \eqref{eq:mfl_pde_first} for $\mfldis_t$ reaches a consensus point $\cptmfl$ in the sense that~$
        \wasserstein_2\bra[\big]{\mfldis_t, \delta_{\cptmfl}} \rightarrow 0 
    $ as $t \to \infty$.
            Assume $\expect \rawempmomentp{2}{0}<\infty$. 
    Then 
            \begin{align}
                \label{eq:consensus-point-vs-target}
                \expect\pra*{ \abs*{\cpt - \cptmfl}^2 } \le \frac{C_{\rm MFL}}{J}\,,
            \end{align}
            where $C_{\rm MFL}$ is the constant from \cref{thm:uit-mfl}.
    \end{enumerate}
\end{theorem}

\subsection{Comments on the proof strategy}

We will prove \cref{thm:uit-mfl} via a synchronous coupling approach based on~\cite{MR1847094}; see also \cite{MR1108185} and~\cite[Section 3.1.3]{ReviewChaintronII}.
Using the contractive nature of the CBO particle system,
we will show an estimate of the form
\begin{align}
    \label{eq:comment_proof_first}
    \expect \pra*{  \bigl\lvert \xn{j}_t - \xnl{j}_t \bigr\rvert^2}
    \le C_1 \int_0^t \expect \pra*{ \bigl\lvert  \xn{j}_s - \xnl{j}_s \bigr\rvert^2} \e^{-a s}  \d s  + C_2 J^{-1},
\end{align}
for some constants $C_1, C_2, a>0$,
from which the claim then easily follows by Gr\"onwall's inequality.
The proof of \eqref{eq:comment_proof_first} is based on the following results:
\begin{itemize}[leftmargin=*]
    \item
          \textbf{Local Lipschitz continuity of~$\mu \mapsto \wmmu-  \nm(\mu)$}.
          One of the main difficulties for proving uniform-in-time mean-field results for CBO is that
          while the mean operator~$\mathcal M(\placeholder)$ is globally Lipschitz continuous with Lipschitz constant 1 for the Wasserstein metric,
          \[
              \forall \mu, \nu \in \mathcal P_1(\real^d), \qquad
              \abs[\big]{ \mathcal M(\mu) - \mathcal M(\nu) } \leq \wasserstein_1 (\mu, \nu),
          \]
          the weighted mean $\wm(\placeholder)$ is in general not globally Lipschitz continuous with constant 1.
          In the recent work~\cite{huang2024uniformintimemeanfieldlimitestimate},
          the problem is circumvented by altering the dynamics.
          Here, we will prove and exploit the following stability estimate (see \cref{lem:stab_wmean}),
          which holds for all~$\mu, \nu \in \mathcal P_2(\real^d)$:
          \begin{equation}
              \label{eq:stab_summary}
              \left\lvert \wmmu - \mathcal M(\mu) - \wmnu + \mathcal M(\nu) \right\rvert
              \leq C_{\mathcal M}
              \left(  \sqrt{\momentp{2}{\mu}} + \sqrt{\momentp{2}{\nu}}  \right)
              \wasserstein_2 (\mu, \nu).
          \end{equation}
          This estimate allows us to control the remainder terms that arise
          when the weighted means are treated as perturbations of ordinary means.

    \item \textbf{Exponential decay of centered moments.}
          To obtain the exponentially decaying prefactors in~\eqref{eq:comment_proof_first},
          we show that under the conditions of~\cref{thm:uit-mfl},
          the following estimates hold:
          \begin{subequations}
              \label{eq:moments_summary}
              \begin{align+}
                  \expect \left[ \empmomentp{p}{t} \right] &\leq \expect \left[ \empmomentp{p}{0} \right] \e^{-\edecay{p} t},
                  \\
                  \expect  \Bigl\lvert \xl_t - \expect \xl_t \Bigr\rvert^{p}
                  &\leq \expect  \Bigl\lvert \xl_0 - \expect \xl_0 \Bigr\rvert^{p} \e^{- \edecay{p} t},
              \end{align+}
          \end{subequations}
          where $\lambda_p$ is given by \eqref{eq:intro:exp-decay-factor}.
          For precise statements and proofs,
          see \cref{lem:decay-centered-moments,thm:mfl-decay-p-p-cetered-moment}.

    \item
          \textbf{Uniform-in-time control of raw moments.}
          In~\cref{lem:uit-raw-moments-2norm,lem:uit-raw-moments-synchronous-system} we prove the following uniform-in-time bounds on the moments of the interacting particle system and associated mean-field dynamics:
          \begin{align}\label{eq:raw_moments_summary}
              \expect \biggl[ \sup_{t \geq 0} \Bigl\lvert \xn{j}_t \Bigr\rvert^p \biggr]^{\frac{1}{p}}
              \leq
              \craw{p} \expect \biggl[ \Bigl\lvert \xn{j}_0 \Bigr\rvert^p \biggr]^{\frac{1}{p}}
              \,, \qquad
              \expect \biggl[ \sup_{t \geq 0} \bigl\lvert \xl_t \bigr\rvert^p \biggr]^{\frac{1}{p}}
              \leq \craw{p} \expect \biggl[ \bigl\lvert \xl_0 \bigr\rvert^p \biggr]^{\frac{1}{p}}\,.
          \end{align}
    \item
          \textbf{Concentration inequalities.}
          We would like to apply the stability estimate~\eqref{eq:stab_summary} to $\mu=\empl{t}$ and $\nu=\emp{t}$.
          In order to control the bracketed factor on the right-hand side of~\eqref{eq:stab_summary},
          we show that $\momentp{2}{\empl{t}}$ and $\momentp{2}{\emp{t}}$ decay not only in expectation,
          but also almost surely in the complement of an event of small probability.
          To this end, we prove in \cref{prop:bound-on-bad-set-general-noise,prop:bound-on-bad-set-synchronous-mf}
          the following concentration inequalities,
          which hold under appropriate assumptions for all $q \geq 2$ and $\kappa < \min\left\{ \edecay{2}, \frac{\edecay{2q}}{q} \right\}$:
          \begin{equation}
              \label{eq:intro:concentration-summary}
              \proba \left[ \sup_{t \geq 0} \e^{\kappa t} \empmomentp{2}{t} \geq \momentp{2}{\mfldis_0} + 1 \right]
              \lesssim J^{- \frac{q}{2}}, \qquad
              \proba \left[ \sup_{t \geq 0} \e^{\kappa t} \emplmomentp{2}{t} \geq \momentp{2}{\mfldis_0} + 1 \right]
              \lesssim J^{- \frac{q}{2}}.
          \end{equation}
          Here, $\emplmomentp{2}{t}$ denotes the second centered moment of the empirical measure $\empl{t}$ of $J$ i.i.d. mean-field particles.

    \item
          \textbf{Monte Carlo convergence of the weighted mean.}
          Finally, we use an estimate stating that,
          for any probability measure $\pi \in \mathcal P_p(\R^d)$ and any $p \geq 2$,
          \begin{align}
              \label{eq:intro:mc_convergence}
               & \expect  \left\lvert
              \wm \bigl( \mu_{\mathcal Z^J} \bigr)
              - \wm\bigl(\pi\bigr)
              \right\rvert_p^p
              \leq \cwm{p}
              \expect  \Bigl\lvert \zn{1} - \expect \zn{1} \Bigr\rvert_p^{p} J^{- \frac{p}{2}},
              \qquad \mu_{\mathcal Z^J} \coloneq \frac{1}{J}\sum_{j=1}^{J} \delta_{Z^j},
              \qquad
              \left\{ \zn{j} \right\}_{j \in \N} \stackrel{\rm{i.i.d.}}{\sim} \pi \,.
          \end{align}
          A similar bound was already proved in \cite{gerber2023meanfield},
          but here we make the dependence of the sampling error on $\pi$  more explicit.
          For the precise description of the assumptions under which \eqref{eq:intro:mc_convergence} holds,
          see \cref{lem:convergence_weighted_mean_iid-2}.
\end{itemize}

The proof of \cref{thm:stab_particle} closely parallels that of \cref{thm:uit-mfl}.
Using again a synchronous coupling approach and leveraging the contractive property of the CBO particle system,
we prove an estimate analogous to \eqref{eq:comment_proof_first}, but now with~$\xnl{j}_t$ substituted with~$\xnt{j}_t$:
\begin{align}
    \label{eq:comment_proof_second}
    \expect \pra*{  \bigl\lvert \xn{j}_t - \xnt{j}_t \bigr\rvert^2}
    \le \expect \pra*{ \bigl\lvert  \xn{j}_0 - \xnt{j}_0 \bigr\rvert^2 } + \widetilde{C}_1 \int_0^t \expect \pra*{ \bigl\lvert  \xn{j}_s - \xnt{j}_s \bigr\rvert^2} \e^{- \widetilde{a} s}  \d s  + \widetilde{C}_2 J^{-q}\,,
\end{align}
The structure and ingredients of the proof are analogous to those for \cref{thm:uit-mfl},
except that we no longer make use of the Monte Carlo estimate~\eqref{eq:intro:mc_convergence},
as we do not pivot around $\empl{t}$.
Since \eqref{eq:intro:mc_convergence} becomes unnecessary for the proof,
we gain the better decay rate $J^{-q}$ with respect to~$J$,
compared to the rate from mean-field limit estimate in \cref{thm:uit-mfl} (see also \cref{rmk:stability_better_rate}).

Finally, to prove \cref{thm:microscopic-concentration-around-consensus-point},
we first show that the mean $\meanx{t}$ is almost surely a Cauchy sequence in time with an exponential rate
(\cref{claim:Cauchy-sequence}) by relying on the exponential decay of the centered moments $\empmomentp{p}{t}$ and the Borel--Cantelli lemma.
This leads to the conclusion that~$\meanx{t}$ almost surely converges to a limit $\cpt$ as stated in~\eqref{eq:as_conv_mean}.
A second key auxiliary result is the particle collapse for \eqref{eq:cbo} (\cref{claim:empirical-cov-conv-to-0-almost-surely}): We obtain a.s.\ exponential convergence of $\empmomentp{2}{t}$ to zero as a consequence of the concentration inequality \eqref{eq:intro:concentration-summary} (\cref{prop:bound-on-bad-set-general-noise}). All estimates \eqref{eq:as_conv_mean}, \eqref{eq:as_conv_wmean} and \eqref{eq:as_conv_particles} in Part~\ref{item:almost-sure-exponential-concentration} of \cref{thm:microscopic-concentration-around-consensus-point} then follow via pivoting around $\meanx{t}$. The mean square convergence estimates stated in Part~\ref{item:mean-square-exponential-concentration} of~\cref{thm:microscopic-concentration-around-consensus-point} follow by first deriving the SDE satisfied by the empirical mean, and then combining the first item with the exponential decay of centered moments in~\cref{lem:decay-centered-moments}, via an application of the dominated convergence theorem. The final item Part~\ref{item:consensus-point-close-to-target} in~\cref{thm:microscopic-concentration-around-consensus-point} is then obtained by using the first item together with the uniform-in-time propagation of chaos proved in~\cref{thm:uit-mfl}.

\section{Proof of the main results}
\label{sec:proofs_main}

\subsection{Proof of \texorpdfstring{\cref{thm:uit-mfl}}{Theorem~\ref{thm:uit-mfl}}}
\label{sec:proof:thm:uit-mfl}
\begin{proof}
    Let
    \[
        \mathcal E_t = \frac{1}{J} \sum_{j=1}^{J} \left\lvert \xn{j}_t - \xnl{j}_t \right\rvert^2.
    \]
    Observe that the requirement $\sigma \in [0,\widetilde{\sigma})$ implies that $0< \edecay{8} < 8\edecay{2}$, see \eqref{eq:noise-upper-bound} and \eqref{eq:intro:exp-decay-factor}.
    By It\^o's formula,
    it holds that
    \begin{align}
        \d \mathcal E_t
         & = - \frac{2}{J} \sum_{j=1}^{J} \left\langle \xn{j}_t - \xnl{j}_t, \xn{j}_t - \xnl{j}_t - \wmx{t} + \wmeanmfl{t}  \right\rangle \d t                                                             \\
         & \qquad + \frac{ \sigma^2}{J}  \sum_{j=1}^{J}
        \tr \pra*{ \bra*{S\bra*{\xn{j}_t - \wmx{t}} - S\bra*{\xnl{j}_t - \wmeanmfl{t}}}^2 } \d t
        \\
         & \qquad + \frac{2 \sigma}{J}  \sum_{j=1}^{J} \left\langle \xn{j}_t - \xnl{j}_t, S\bra*{ \xn{j}_t - \wmx{t} } \, \d \wn{j}_t - S\bra*{ \xnl{j}_t - \wmeanmfl{t} } \, \d \wn{j}_t \right\rangle\,.
    \end{align}
    In the case of isotropic noise $S(x)= \abs{x} I_d$, we have by the reverse triangle inequality
    \begin{align}
        \tr \pra*{ \bra*{S\bra*{\xn{j}_t - \wmx{t}} - S\bra*{\xnl{j}_t - \wmeanmfl{t}}}^2 }
        \le d \abs*{\xn{j}_t - \wmx{t} - \xnl{j}_t + \wmeanmfl{t} }^2,
    \end{align}
    while in the case of anisotropic noise $S(x) = \diag(x)$, we have
    \begin{align}
        \tr \pra*{ \bra*{S\bra*{\xn{j}_t - \wmx{t}} - S\bra*{\xnl{j}_t - \wmeanmfl{t}}}^2 } = \abs*{\xn{j}_t - \wmx{t} - \xnl{j}_t + \wmeanmfl{t} }^2.
    \end{align}
    Taking the expectation and rearranging,
    we obtain using the notation \eqref{eq:noise-prefactor} that
    \begin{align}
        \frac{\d}{\d t} \expect \mathcal E_t
         & \leq - \frac{2}{J} \expect \sum_{j=1}^{J} \left\langle \xn{j}_t - \xnl{j}_t, \xn{j}_t - \xnl{j}_t - \wmx{t} + \wmeanmfl{t}  \right\rangle \\
         & \qquad + \frac{\tau(S) \sigma^2}{J}   \expect \sum_{j=1}^{J}
        \abs*{\xn{j}_t - \wmx{t} - \xnl{j}_t + \wmeanmfl{t} }^2
        = \mathcal A_1(t) + \mathcal A_2(t) + \mathcal A_3 (t) \, ,
    \end{align}
    where~$\mathcal A_1(t), \mathcal A_2(t), \mathcal A_3(t)$ are defined as
    \begin{align}
        \mathcal A_1(t)
                        & := - \frac{2}{J} \expect \sum_{j=1}^{J} \left\langle \xn{j}_t - \xnl{j}_t, \xn{j}_t - \xnl{j}_t - \wmx{t} + \wmxl{t}  \right\rangle, \\
        \mathcal A_2(t) & := - \frac{2}{J} \expect \sum_{j=1}^{J} \left\langle \xn{j}_t - \xnl{j}_t, \wmeanmfl{t} - \wmxl{t}  \right\rangle,                   \\
        A_3(t)          & :=  \frac{\tau(S) \sigma^2}{J}   \expect \sum_{j=1}^{J} \abs*{\xn{j}_t - \wmx{t} - \xnl{j}_t + \wmeanmfl{t} }^2,
    \end{align}
    with $\empl{t}:= \frac{1}{J} \sum_{j=1}^J \delta_{\xnl{j}_t}$ being the empirical measure of the i.i.d.\ mean-field particles.
    \paragraph{Bounding $\mathcal A_1(t)$}
    We rewrite
    \begin{align}
        \mathcal A_1(t)
         & = -\frac{2}{J}\sum_{j=1}^{J} \expect \left\langle \xn{j}_t - \xnl{j}_t, \xn{j}_t - \xnl{j}_t \right\rangle
        + 2 \expect
        \left\langle \nmx{t} - \nml{t}, \wmx{t} - \wmxl{t}  \right\rangle                                                             \\
         & = - \frac{2}{J} \sum_{j=1}^J \expect \left[
            \left\lvert \xn{j}_t - \xnl{j}_t  \right\rvert^2 \right]
        + 2 \expect \left[
            \left\lvert \nmx{t} - \nml{t}  \right\rvert^2
        \right]                                                                                                                       \\
         & \qquad -  2 \expect \left[ \left\langle \nmx{t} - \nml{t}, \nmx{t} - \wmx{t} - \nml{t} + \wmxl{t}  \right\rangle
        \right]                                                                                                                       \\
         & = -\frac{2}{J}\sum_{j=1}^{J} \expect \left[ \left\lvert \xn{j}_t - \nmx{t} - \xnl{j}_t + \nml{t} \right\rvert^2
            \right]
        \\
         & \qquad -  2\expect \left[ \left\langle \nmx{t} - \nml{t}, \nmx{t} - \wmx{t} - \nml{t} + \wmxl{t}  \right\rangle \right]\,,
    \end{align}
    where we used Huygens' elementary identity,
    which holds for any collection $\{z_j\}_{j \in \range{1}{J}}$ in~$\real^d$:
    \begin{align}
        \label{eq:huygens}
         & \frac{1}{J}\sum_{j = 1}^J \left|z_j\right|^2 = \frac{1}{J}\sum_{j = 1}^J \left|z_j - m\right|^2 + \left|m \right|^2 , \qquad
        \text{ with }  \, m = \frac{1}{J}\sum_{j = 1}^J z_j\,.
    \end{align}
    Using the inequality~$\lvert \nmx{t} - \nml{t} \rvert^2 \leq \mathcal E_t$ and introducing
    \begin{align}
        \mathcal B(t) := \expect \Bigl\lvert  \nmx{t} - \wmx{t} - \nml{t} + \wmxl{t} \Bigr\rvert^2,
    \end{align}
    we bound the term $\mathcal A_1(t)$ as follows:
    \begin{align}
        \mathcal A_1(t)
         & \le -\frac{2}{J}\sum_{j=1}^{J} \expect \biggl[ \left\lvert \xn{j}_t - \nmx{t} - \xnl{j}_t + \nml{t} \right\rvert^2 \biggr]
        \label{eq:A1-final}
        + 2 \bra*{\expect \mathcal E_t}^{\frac{1}{2}} \mathcal B(t)^{\frac{1}{2}}.
    \end{align}
    \paragraph{Bounding $\mathcal A_2(t)$}
    By \eqref{eq:moments_summary} and \eqref{eq:intro:mc_convergence} (see \cref{lem:convergence_weighted_mean_iid-2,thm:mfl-decay-p-p-cetered-moment}), the random variable
    \[
        \mathfrak D^J_t = \Bigl\lvert \wmeanmfl{t} - \wmxl{t} \Bigr\rvert
    \]
    satisfies the following inequality:
    \begin{align}\label{eq:D_bound}
        \expect \Bigl[ \abs{\mathfrak D^J_t}^2 \Bigr] \le \frac{\cwm{2} } {J} \expect \abs*{\xnl{j}_t - \expect \xnl{j}_t}^2
        \le \frac{\cwm{2} } {J} \e^{-\edecay{2} t}
        \expect \abs*{\xnl{j}_0 - \expect \xnl{j}_0}^2.
    \end{align}
    Therefore, by the Cauchy--Schwarz inequality,
    {we obtain for any $\zeta \in (0, \edecay{2})$ to be determined later}
    \begin{align}
        \mathcal A_2(t)
         & = - 2 \expect   \left\langle \nmx{t} - \nmxl{t}, \wmeanmfl{t} - \wmxl{t}  \right\rangle \\
         &
        \leq 2 \sqrt{\expect \mathcal E_t} \, {\sqrt{\expect \Bigl[ \abs{\mathfrak D^J_t}^2 \Bigr]}}
        \le \e^{-{\zeta}t} \expect \mathcal E_t + \frac{\cwm{2}}{J} \e^{-{(\edecay{2} - \zeta)} t} \expect \abs*{\xnl{j}_0 - \expect \xnl{j}_0}^2.
    \end{align}
    \paragraph{Bounding $\mathcal A_3(t)$}
    From \eqref{eq:huygens} we obtain
    \begin{align}
        \mathcal A_3(t)
         & =  \frac{\tau(S)\sigma^2}{J} \sum_{j=1}^{J} \expect \biggl[ \Bigl\lvert \xn{j}_t - \wmx{t} - \xnl{j}_t + \wmeanmfl{t} \Bigr\rvert^2 \biggr] \\
         & =  \frac{\tau(S)\sigma^2}{J} \sum_{j=1}^{J} \expect \biggl[ \Bigl\lvert \xn{j}_t - \nmx{t} - \xnl{j}_t + \nml{t} \Bigr\rvert^2 \biggr]      \\
         & \qquad + \tau(S) \sigma^2 \expect \biggl[ \Bigl\lvert\nmx{t} - \wmx{t} - \nml{t} + \wmeanmfl{t} \Bigr\rvert^2 \biggr]\,.
    \end{align}
    Thus, using that $|a + b|^2 \leq 2|a|^2 + 2|b|^2$ we obtain
    \begin{align}
        \mathcal A_3(t)
        \le &
        \frac{\tau(S)\sigma^2}{J} \sum_{j=1}^{J} \expect\left[  \Bigl\lvert \xn{j}_t - \nmx{t} - \xnl{j}_t + \nmxl{t} \Bigr\rvert^2
            \right]
        + 2\tau(S)\sigma^2
        \mathcal B(t) +  2 \tau(S) \sigma^2 \expect \pra*{\abs*{\mathfrak D^J_t}^2}\,.
    \end{align}
    Observe that,
    since~$\tau(S)\sigma^2 \leq 2$ by assumption,
    the first summand can be compensated by the first term in~\eqref{eq:A1-final}.
    The last term can be controlled by~\eqref{eq:D_bound}.

    \paragraph{Conclusion}
    Summing up, {and noting that $\e^{- \edecay{2} t } \leq \e^{- (\edecay{2} - \zeta) t}$},
    we obtain
    \begin{align}
        \mathcal A_1(t) + \mathcal A_2(t) + \mathcal A_3(t)
         & \le
        2 \bra*{\expect \mathcal E_t}^{\frac{1}{2}} \, \mathcal B(t)^{\frac{1}{2}}
        + \e^{-{\zeta}t} \expect \mathcal E_t
        + \Bigl( 1+ {2} \tau(S) \sigma^2 \Bigr)\frac{\cwm{2}}{J} \e^{-{(\edecay{2} - \zeta)} t} \expect \abs*{\xnl{j}_0 - \expect \xnl{j}_0}^2
        + 2\tau(S)\sigma^2 \mathcal B(t) \\
         & \leq
        \Bigl(1 + 2 \tau(S) \sigma^2\Bigr) \mathcal B(t) \e^{\zeta t} + 2 \e^{-{\zeta}t} \expect \mathcal E_t
        + \Bigl( 1+ {2} \tau(S) \sigma^2 \Bigr) \frac{\cwm{2}}{J} \e^{-{(\edecay{2} - \zeta)} t} \momentp{2}{\mfldis_0}.
        \label{eq:conclusion_bound}
    \end{align}
    By~\eqref{eq:stab_summary} (see \cref{lem:stab_wmean}), we have that
    \begin{align}
        \mathcal B(t) \leq 2C_{\mathcal M}^2 \expect \pra*{
        \bra*{ {\empmomentp{2}{t}} + {\emplmomentp{2}{t}} }
        \wasserstein_2^2\bra*{\emp{t}, \empl{t}}}.
    \end{align}
    Further below, we will prove that
    there exists a finite constant $C_{\rm Q}$ such that
    \begin{align}
        \label{eq:technical-inequality}
        \expect \pra*{
        \bra*{ {\empmomentp{2}{t}} + {\emplmomentp{2}{t}} }
        \wasserstein_2^2\bra*{\emp{t}, \empl{t}}} \le
        C_{\rm Q} \left( J^{-1} \e^{- \frac{\edecay{8}}{4} t} + \e^{- \kappa t} \expect \mathcal{E}_t \right) \,,
    \end{align}
    where $\kappa = \frac{\edecay{8}}{8}$ satisfies $\kappa\le \edecay{2}$.
    Thus,
    substituting this bound into~\eqref{eq:conclusion_bound} and letting $\zeta = \frac{\kappa}{2} =  \frac{\edecay{8}}{16} $,
    we obtain
    \begin{align}
        \frac{\d}{\d t}\expect \mathcal{E}_t
         & \leq 2C_{\mathcal M}^2  C_{\rm Q} \Bigl(1 + 2 \tau(S) \sigma^2\Bigr) \Bigl( J^{-1}  + \expect \calE_t \Bigr) \e^{- \frac{\kappa}{2}}                                         \\
         & \qquad + 2  \expect \mathcal E_t \e^{-{\frac{\kappa}{2}}t} + \Bigl( 1+ {2} \tau(S) \sigma^2 \Bigr)\frac{\cwm{2}}{J} \e^{-{(\edecay{2} - \kappa/2)} t} \momentp{2}{\mfldis_0} \\
         & \leq \kappa\left( c_1 \expect \mathcal E_t + \frac{c_2}{J} \right) \e^{-{\frac{\kappa}{2}}t}  \,,
    \end{align}
    where
    \begin{align}
        \label{eq:constants-c1}
         & c_1:=
        \cone \,          \\
        \label{eq:constants-c2}
         & c_2:=\ctwo \,.
    \end{align}
    Thus, rewriting the inequality in integral form, we have
    \[
        \expect \mathcal E_t
        \leq \expect \calE_0 + \kappa c_1 \int_{0}^{t} \expect \mathcal E_s \e^{- \frac{\kappa}{2} s} \, \d s + \frac{2 c_2}{ J}.
    \]
    Finally, using the integral version of Grönwall's inequality,
    we conclude that
    \begin{align}
        \expect \calE_t \leq \left( \expect \calE_0 + \frac{2 c_2}{ J} \right) \e^{ 2 c_1 } \,.
    \end{align}
    \paragraph{Proof of \eqref{eq:technical-inequality}}
    To motivate \eqref{eq:technical-inequality},
    consider first the setting where $\sigma = 0$ and $\mfldis_0$ is compactly supported.
    In this setting, the terms $\empmomentp{2}{t}$ and $\emplmomentp{2}{t}$ are almost surely bounded from above by a decreasing exponential,
    in view of~\eqref{eq:moments_summary} (see~\cref{lem:decay-centered-moments}).
    Therefore, applying Hölder's inequality for the exponents $(\infty, 1)$
    and using the definition of the Wasserstein distance,
    we obtain for some appropriate constant~$C$ that
    \[
        \expect \pra*{
        \bra*{ {\empmomentp{2}{t}} + {\emplmomentp{2}{t}} }
        \wasserstein_2^2\bra*{\emp{t}, \empl{t}}}
        \le C \e^{- \edecay{2} t}
        \expect \left[ \wasserstein_2^2\bra*{\emp{t}, \empl{t}} \right]
        \leq C \e^{- \edecay{2} t}  \mathcal E_t.
    \]
    In the presence of noise,
    H\"older's inequality cannot be applied in this manner,
    because the exponential decay of the quadratic centered moments on the left-hand side does not hold almost surely.
    We circumvent this difficulty by using the concentration inequalities~\eqref{eq:intro:concentration-summary} (see~\cref{sub:concent} for the proofs),
    which show that these moments decay exponentially with high probability.
    Specifically,
    fix any $q \geq 2$,
    $\kappa :=\frac{\edecay{8}}{8}<\edecay{2}$
    and
    introduce
    \begin{align}
        \Omega_{\kappa}
         & =
        \Bigl\{
        \omega \in \Omega :
        \sup_{t \geq 0}
        \e^{\kappa t}  \empmomentp{2}{t}
        \geq
        \momentp{2}{\mfldis_0}  + 1
        \Bigr\}  \,, \\
        \overline \Omega_{\kappa}
         & =
        \Bigl\{
        \omega \in \Omega :
        \sup_{t \geq 0}
        \e^{\kappa t}  \emplmomentp{2}{t}
        \geq
        \momentp{2}{\mfldis_0} +1
        \Bigr\}\,,
    \end{align}
    where $B:=2\sigma^2 \tau(S) \e^{\alpha\bra*{\overline f - \underline f}} \expect \abs*{ \widebar{X}_0 - \nm(\mfldis_0)}^2$,
    and define $\Omega^{\star}_{\kappa} := \Omega_{\kappa} \cup \overline \Omega_{\kappa}$.
    By definition of this subset of the sample space, it holds almost surely that
    \begin{align}
        \label{eq:complement_bound}
        \forall  t \geq 0,
        \qquad
        \mathbbm 1_{\Omega \setminus \Omega^{\star}_{\kappa}}
        \left( {\empmomentp{2}{t}} + {\emplmomentp{2}{t}} \right)
        \leq 2 \e^{- \kappa t } \Bigl(\momentp{2}{\mfldis_0} + 1\Bigr).
    \end{align}
    Furthermore, using~\eqref{eq:intro:concentration-summary} (see~\cref{prop:bound-on-bad-set-general-noise,prop:bound-on-bad-set-synchronous-mf,rmk:bound-on-bad-set-general-noise-coarse}),
    we have that
    \begin{equation}
        \label{eq:probab_bad_set}
        \proba \bigl[ \Omega^{\star}_{\kappa} \bigr]
        \leq 2 \cbad{q}{\kappa}               J^{- \frac{q}{2}} \momentp{2q}{\mfldis_0}.
    \end{equation}
    We then decompose the expectation as follows
    \begin{align}
        Q_t
         & \coloneq
        \expect \pra*{
        \bra*{ {\empmomentp{2}{t}} + {\emplmomentp{2}{t}} }
        \wasserstein_2^2\bra*{\emp{t}, \empl{t}}}                                                                                                                                              \\
         & = \expect \left[ \mathbbm 1_{\Omega^{\star}_{\kappa}} \left( {\empmomentp{2}{t}} + {\emplmomentp{2}{t}} \right)  \wasserstein^2_2\bigl(\emp{t}, \empl{t}\bigr) \right] +
        \expect \left[ \mathbbm 1_{\Omega \setminus \Omega_{\kappa}^{\star}} \left( {\empmomentp{2}{t}} + {\emplmomentp{2}{t}} \right)  \wasserstein^2_2\bigl(\emp{t}, \empl{t}\bigr)  \right] \\
         & \leq
        \proba \bigl[ \Omega_{ \kappa}^{\star} \bigr]^{\frac{1}{2}}
        \expect \left[ \left( {\empmomentp{2}{t}} + {\emplmomentp{2}{t}} \right)^{2 }  \wasserstein^4_2\bigl(\emp{t}, \empl{t}\bigr) \right]^{\frac{1}{2}}
        +  2 \e^{- \kappa t }\Bigl(\momentp{2}{\mfldis_0} + 1\Bigr) \expect \left[ \wasserstein^2_2\bigl(\emp{t}, \empl{t}\bigr)
            \right],
    \end{align}
    where the last inequality follows from~\eqref{eq:complement_bound}.
    Now, using the elementary inequality $(a+b)^p \leq 2^{p-1} (a^p + b^p)$
    together with H\"older's inequality,
    we deduce that
    \begin{align}
         & \expect \left[ \left( {\empmomentp{2}{t}} + {\emplmomentp{2}{t}} \right)^{2}  \wasserstein_2\bigl(\emp{t}, \empl{t}\bigr)^{4} \right] \\
         & \qquad \leq 2^{4}
        \expect \left[ \left(\empmomentp{4}{t} + \emplmomentp{4}{t}\right) \left( \wasserstein_2\bigl(\emp{t}, \delta_0 \bigr)^4 + \wasserstein_2\bigl(\delta_0, \empl{t} \bigr)^{4} \right)
        \right] \\
         & \qquad \leq 2^{4}
        \expect \left[ \left(\empmomentp{4}{t} + \emplmomentp{4}{t}\right)
        \left(\rawempmomentp{4}{t} + \rawemplmomentp{4}{t}\right) \right]                                                                        \\
         & \qquad \leq 2^{5} \sqrt{ \expect \left[ \empmomentp{8}{t} + \emplmomentp{8}{t} \right]
            \expect \left[\rawempmomentp{8}{t} + \rawemplmomentp{8}{t}\right] }.
    \end{align}
    From \eqref{eq:moments_summary}, and \eqref{eq:raw_moments_summary} (see~\cref{lem:decay-centered-moments,lem:uit-raw-moments-2norm}),
    it holds that
    \begin{align}
        \forall t \geq 0, \qquad
        \left\{
        \begin{aligned}
            \expect \Bigl[ \empmomentp{8}{t} \Bigr]
             & \leq \expect \Bigl[ \empmomentp{8}{0} \Bigr] \e^{- \edecay{8} t} \\
            \expect \Bigl[ \rawempmomentp{8}{t} \Bigr]
             & \leq \craw{8}^8
            \expect \Bigl[ \rawempmomentp{8}{0} \Bigr],
        \end{aligned}
        \right.
    \end{align}
    and similarly for the mean-field particle system (see~\cref{thm:mfl-decay-p-p-cetered-moment,lem:uit-raw-moments-synchronous-system}).
    Therefore we have
    \begin{align}
        \expect \left[ \left( {\empmomentp{2}{t}} + {\emplmomentp{2}{t}} \right)^{2}  \wasserstein_2\bigl(\emp{t}, \empl{t}\bigr)^{4} \right] & \leq
        2^{6} \craw{8}^{4} \e^{- \frac{\edecay{8}}{2} t}
        \sqrt{
            \expect \left[\empmomentp{8}{0} \right]
            \expect \left[\rawempmomentp{8}{0} \right]
}                                                      \\
      & \leq 2^{10} \craw{8}^{4} \e^{- \frac{\edecay{8}}{2} t}
        \rawmomentp{8}{\mfldis_0}\,,
    \end{align}
    where the last inequality follows by
    \begin{align}\label{eq:raw-centered-moment-bound}
        \momentp{p}{\mu}
        =  \int \bigl\lvert x - \nm(\mu) \bigr\rvert^p  \, \mu(\d x)
        \leq
        2^{p-1} \int \bigl\lvert x \bigr\rvert^p  \, \mu(\d x)
        +  2^{p-1} \bigl\lvert \nm(\mu) \bigr\rvert^p
        \leq 2^p \rawmomentp{p}{\mu} \,.
    \end{align}
    By~\eqref{eq:probab_bad_set} together with the inequality $\wasserstein_2\bigl(\emp{t}, \empl{t}\bigr)^{2} \leq \mathcal E_t$,
    this leads to
    \begin{align}
        Q_t
         & \leq
        2^6
        \cbad{q}{\kappa}^{\frac{1}{2}}
        \craw{8}^{2}
        \sqrt{\momentp{2q}{\mfldis_0}} \sqrt{\rawmomentp{8}{\mfldis_0}} J^{-\frac{q}{4}} \e^{ - \frac{\edecay{8}}{4} t}
        + 2 \e^{- \kappa t } \Bigl( \momentp{2}{\mfldis_0} + 1 \Bigr)  \E \calE_t.
    \end{align}
    In particular,
    taking~$q = 4$,
    then using the inequality~$\momentp{2q}{\mfldis_0} \leq 2^8\rawmomentp{8}{\mfldis_0}$,
    we obtain the claimed inequality~\eqref{eq:technical-inequality} with constant
    \begin{align}
        C_{\rm Q}
         & = 2^{10} \cbad{4}{\kappa}^{\frac{1}{2}}
        \craw{8}^{2}
        \rawmomentp{8}{\mfldis_0} + 2 \Bigl( \momentp{2}{\mfldis_0} + 1 \Bigr) \\
         & \leq
        2^{11}
        \cbad{4}{\kappa}^{\frac{1}{2}}
        \craw{8}^{2}
        \Bigl( \rawmomentp{8}{\mfldis_0} + 1 \Bigr),
        \label{eq:C_q}
    \end{align}
    which concludes the proof.
\end{proof}

\subsection{Proof of \texorpdfstring{\cref{thm:stab_particle}}{Theorem~\ref{thm:stab_particle}}}
\label{sub:proof-stab-particle}
\begin{proof}
    Define
    \begin{align}
        \mathcal G_t :=  \frac{1}{J} \sum_{j=1}^{J} \abs*{ \xn{j}_t - \xnt{j}_t }^2\,.
    \end{align}
    By It\^o's formula,
    it holds that
    \begin{align}
        \d \mathcal G_t
         & = - \frac{2}{J} \sum_{j=1}^{J} \left\langle \xn{j}_t - \xnt{j}_t, \xn{j}_t - \xnt{j}_t - \wmx{t} + \wmt{t}  \right\rangle \d t
        \\
         & \qquad + \frac{\sigma^2}{J}  \sum_{j=1}^{J}
        \tr \bra*{ \abs*{ S\bra*{\xn{j}_t - \wmx{t}} - S\bra*{\xnt{j}_t - \wmt{t}} }^2 } \d t
        \\
         & \qquad + \frac{2 \sigma}{J}  \sum_{j=1}^{J} \left\langle \xn{j}_t - \xnt{j}_t, S\bra*{ \xn{j}_t - \wmx{t} } \, \d \wn{j}_t - S\bra*{ \xnt{j}_t - \wmt{t} } \, \d \wn{j}_t \right\rangle.
    \end{align}
    Similarly to the proof of~\cref{thm:uit-mfl}, we obtain by taking expectations and rearranging that
    \begin{align}
        \frac{\d}{\d t} \expect \mathcal G_t
         & \le - \frac{2}{J} \expect \sum_{j=1}^{J} \left\langle \xn{j}_t - \xnt{j}_t, \xn{j}_t - \xnt{j}_t - \wmx{t} + \wmt{t}  \right\rangle \\
         & \qquad + \frac{\tau(S)\sigma^2}{J}   \expect \sum_{j=1}^{J} \abs*{\xn{j}_t - \xnt{j}_t - \wmx{t}  + \wmt{t} }^2
        =:  \mathcal A_1(t) +  \mathcal A_2(t).
    \end{align}
    To bound the terms, we define \[
        \mathcal B(t) := \expect \Bigl\lvert \nmx{t} - \wmx{t} - \nmt{t} + \wmt{t} \Bigr\rvert^2\,,
    \]
    where we can use \eqref{eq:stab_summary} (see~\cref{lem:stab_wmean}) to get \begin{align}
        \mathcal B(t) & \leq 2 C_{\mathcal M}^2 \E\left[ \left( \momentp{2}{\emp{t}} + \momentp{2}{\empt{t} }  \right)
            \wasserstein_2^2( \emp{t} , \empt{t} ) \right] .
    \end{align}

    \paragraph{Bounding $\mathcal A_1(t)$}
    Note that by \eqref{eq:huygens}, we have
    \begin{align}
        \mathcal A_1(t)
         & =- \frac{2}{J}\sum_{j=1}^{J}\expect \left\langle \xn{j}_t - \xnt{j}_t, \xn{j}_t - \xnt{j}_t \right\rangle
        + 2\expect \left\langle \nmx{t} - \nmt{t}, \wmx{t} - \wmt{t}  \right\rangle                                   \\
         & = -\frac{2}{J}\sum_{j=1}^{J} \expect\left\lvert \xn{j}_t - \nmx{t} - \xnt{j}_t + \nmt{t} \right\rvert^2    \\
         & \qquad -  2\expect\left\langle \nmx{t} - \nmt{t}, \nmx{t} - \wmx{t} - \nmt{t} + \wmt{t}  \right\rangle
        \\
         & \le  -\frac{2}{J}\sum_{j=1}^{J}\expect \left\lvert \xn{j}_t - \nmx{t} - \xnt{j}_t + \nmt{t} \right\rvert^2
        + 2 \bra*{\expect \mathcal G_t}^{1/2} \bra*{\mathcal B(t)}^{1/2}                                              \\
         & \leq -\frac{2}{J}\sum_{j=1}^{J}\expect \left\lvert \xn{j}_t - \nmx{t} - \xnt{j}_t + \nmt{t} \right\rvert^2
        +\e^{- \zeta t} \expect \mathcal G_t +
        \e^{ \zeta t} \mathcal B(t) \,.
    \end{align}

    \paragraph{Bounding $\mathcal A_2(t)$}
    On the other hand, using again \eqref{eq:huygens}, we have
    \begin{align}
        \mathcal A_2(t)
         & = \tau(S)\sigma^2{J} \sum_{j=1}^{J} \expect  \Bigl\lvert \xn{j}_t - \wmx{t} - \xnt{j}_t + \wmt{t} \Bigr\rvert^2                         \\
         & = \frac{\tau(S) \sigma^2}{J}\sum_{j=1}^{J} \expect \Bigl\lvert \xn{j}_t - \nmx{t} - \xnt{j}_t + \nmt{t} \Bigr\rvert^2 + \tau(S)\sigma^2
        \expect\left\lvert \nmx{t} - \wmx{t} - \nmt{t} + \wmt{t} \right\rvert^2
        \\
         & =
        \frac{\tau(S)\sigma^2}{J}
        \sum_{j=1}^{J} \expect \Bigl\lvert \xn{j}_t - \nmx{t} - \xnt{j}_t + \nmt{t} \Bigr\rvert^2
        +
        \tau(S)\sigma^2
        \mathcal B(t) \,.
    \end{align}
    Thus, we deduce that for a positive $\zeta \in (0,  \frac{\edecay{8}}{4}) \,,$
    \begin{align}
        \mathcal A_1(t) + \mathcal A_2(t)
         & \leq - \left( \frac{  2 - \tau(S)\sigma^2 }{J}\right)
        \sum_{j=1}^{J}
        \expect
        \left\lvert \xn{j}_t - \nmx{t} - \xnt{j}_t + \nmt{t} \right\rvert^2 + \e^{- \zeta t} \expect \mathcal G_t
        + \left(\e^{ \zeta t} + \tau(S)\sigma^2
        \right) \mathcal B(t) \,.
    \end{align}
    Using a similar argument as for proving \eqref{eq:technical-inequality} in the proof of \cref{thm:uit-mfl},
    for any $\tilde{q} \geq 2$ such that $\expect \pra*{\empmomentp{2\tilde{q}}{0}} < \infty $, we have that
    \begin{align}\label{eq:technical-inequality_for_stab}
        \E\left[ \left( \momentp{2}{\emp{t}} + \momentp{2}{\empt{t} }  \right)
            \wasserstein_2^2( \emp{t} , \empt{t} ) \right]
         & \leq \widetilde{C}_{\rm Q} \left( J^{-\frac{\tilde{q}}{4}}\e^{ - \frac{\edecay{8}}{4} t } + \e^{- \frac{\edecay{8}}{8} t} \expect \mathcal{G}_t \right).
    \end{align}
    where $\kappa=\frac{\edecay{8}}{8}< \edecay{2}$
    due to $\sigma < \tilde{\sigma}$ from \eqref{eq:noise-upper-bound}.

    \paragraph{Conclusion}
    Defining $q:= \frac{\tilde{q}}{4} \geq \frac{1}{2}$, and $\zeta = \frac{\edecay{8}}{16}$,  we obtain
    \begin{align}
        \frac{\d}{\d t}\expect \mathcal{G}_t
         & \leq\e^{- \zeta t} \expect \mathcal G_t
        + 2 C_{\mathcal M}^2 \widetilde{C}_{\rm Q} \bra*{\e^{\zeta t} + \tau(S)\sigma^2}
        \bra*{ J^{-q}\e^{ - \frac{\edecay{8}}{4} t  } +\e^{- \frac{\edecay{8}}{8}  t} \expect \mathcal{G}_t }                                                             \\
         & \leq \tilde{c}_1 \exp \left( - \frac{\edecay{8}}{16 } t  \right) \expect \mathcal G_t + \tilde{c}_2  \exp \left( - \frac{\edecay{8}}{16 } t  \right) J^{-q}\,,
    \end{align}
    where
    \begin{align}
         & \tilde{c}_1:= \ctildeone \label{eq:constants-cstab1}\,,  \\
         & \tilde{c}_2:= \ctildetwo \,. \label{eq:constants-cstab2}
    \end{align}
    Integrating this bound, we obtain
    \begin{align}
        \expect  \mathcal G_t \leq \expect \mathcal{G}_0  + \frac{16 \tilde{c}_2}{ \edecay{8} J^q  }
        + \int_0^t \tilde{c}_1\e^{- \frac{\edecay{8}}{16} s} \expect \mathcal G_s ds\,,
    \end{align}
    and we conclude by applying Gr\"onwall's inequality that
    \begin{align}
        \expect \mathcal G_t \leq \left( \expect \mathcal{G}_0  + \frac{16\tilde{c}_2}{\edecay{8}J^q} \right)\e^{\frac{16\tilde{c}_1}{\edecay{8}} } \,.
    \end{align}

    \paragraph{Proof of \eqref{eq:technical-inequality_for_stab}} For $\kappa > 0$,
    we define $\Omega_{\kappa}^*:= \Omega_{\kappa} \cup \widetilde{\Omega}_{\kappa}\,,$ where
    \begin{align}
         & \Omega_{\kappa} :=  \Bigl\{
        \omega \in \Omega :
        \sup_{t \geq 0}
        \e^{\kappa t}  \empmomentp{2}{t}
        \geq
        \expect \pra*{ \empmomentp{2}{0}} + 1
        \Bigr\} \,,                             \\
         & \widetilde{\Omega}_{\kappa}:=\Bigl\{
        \omega \in \Omega :
        \sup_{t \geq 0}
        \e^{\kappa t}  \mathfrak{M}_2(\mu_{\widetilde{X}_t^J})
        \geq
        \expect \pra*{ \emptmomentp{2}{0}} + 1
        \Bigr\}  \,.
    \end{align}
    Then we have
    \begin{align}
        \label{eq:stab:complement_bound}
        \forall  t \geq 0,
        \qquad
        \mathbbm 1_{\Omega \setminus \Omega^{\star}_{\kappa}}
        \left( {\empmomentp{2}{t}} + {\emptmomentp{2}{t}} \right)
        \leq \e^{- \kappa t } \Bigl( \expect \pra*{ \empmomentp{2}{0} + \emptmomentp{2}{0}} +
        2\Bigr)
    \end{align}
    and from \eqref{eq:intro:concentration-summary} (see \cref{prop:bound-on-bad-set-general-noise,rmk:bound-on-bad-set-general-noise-coarse}) for $\tilde{q} \geq 2$ and $\kappa < \min\left\{\frac{\edecay{2\tilde{q}}}{\tilde{q}}, \edecay{2}\right\}\,,$
    \begin{equation}
        \proba \bigl[ \Omega^{\star}_{\kappa} \bigr]
        \leq \proba\bigl[ \Omega_{\kappa} \bigr] + \proba\bigl[ \widetilde{\Omega}_{\kappa} \bigr] \leq  \cbad{\tilde{q}}{\kappa}
        J^{- \frac{\tilde{q}}{2}} \expect \pra*{ \empmomentp{2\tilde{q}}{0} + \emptmomentp{2\tilde{q}}{0}} \,.
    \end{equation}
    Splitting the expectation on the left-hand side of \eqref{eq:technical-inequality_for_stab} into two parts, we have
    \begin{align}
         & \E\left[ \left( \momentp{2}{\emp{t}} + \momentp{2}{\empt{t} }  \right) \wasserstein_2^2( \emp{t} , \empt{t} ) \right] \\
         & \qquad \qquad \leq
        \expect \left[ \mathbbm 1_{\Omega^{\star}_{\kappa}} \left( {\empmomentp{2}{t}} + {\emptmomentp{2}{t}} \right)  \wasserstein^2_2\bigl(\emp{t}, \empt{t}\bigr)  \right] +
        \expect \left[ \mathbbm 1_{\Omega \setminus \Omega_{\kappa}^{\star}} \left( {\empmomentp{2}{t}} + {\emptmomentp{2}{t}} \right)  \wasserstein^2_2\bigl(\emp{t}, \empt{t}\bigr)  \right]
        \\  & \qquad \qquad \leq
        \proba \bigl[ \Omega_{ \kappa}^{\star} \bigr]^{\frac{1}{2}}
        \expect \left[ \left( {\empmomentp{2}{t}} + {\emptmomentp{2}{t}} \right)^{2 }  \wasserstein^4_2\bigl(\emp{t}, \empt{t}\bigr) \right]^{\frac{1}{2}}
        +   \e^{- \kappa t } \Bigl( \expect \pra*{ \empmomentp{2}{0} + \emptmomentp{2}{0}} +
        2\Bigr) \expect \mathcal G_t.
    \end{align}
    From \eqref{eq:moments_summary} and \eqref{eq:raw_moments_summary} (see \cref{lem:decay-centered-moments,lem:uit-raw-moments-2norm}), we have that
    \begin{align}
         & \expect \left[ \left( {\empmomentp{2}{t}} + {\emptmomentp{2}{t}} \right)^{2 }  \wasserstein^4_2\bigl(\emp{t}, \empt{t}\bigr) \right]
        \\ & \qquad \qquad \leq 2^{5}
        \sqrt{ \expect \left[ \empmomentp{8}{t} + \emptmomentp{8}{t} \right]
            \expect \left[\rawempmomentp{8}{t} + \rawemptmomentp{8}{t}\right] }
        \\ & \qquad \qquad \leq  2^9 \craw{8}^{4}
        \bra*{
            \expect \left[\rawempmomentp{8}{0} \right] +
            \expect \left[\rawemptmomentp{8}{0} \right]
        } \e^{- \frac{\edecay{8}}{2} t}\,,
    \end{align}
    where the last inequality follows by \eqref{eq:raw-centered-moment-bound} once again.
    Therefore, we obtain \eqref{eq:technical-inequality_for_stab} with
    \begin{align}
        \widetilde{C}_{\rm Q} & := \sqrt{2}^9 \cbad{\tilde q}{\kappa}^{\frac{1}{2}} \craw{8}^{4}
        \sqrt{\expect \pra*{ \empmomentp{2\tilde q}{0} + \emptmomentp{2q}{0}} }
        \sqrt{\expect \left[\rawempmomentp{8}{0} \right] +
            \expect \left[\rawemptmomentp{8}{0} \right] }
        \\
                              & \qquad+ \expect \pra*{ \empmomentp{2}{0} + \emptmomentp{2}{0}} + 2. \qedhere
    \end{align}
\end{proof}

\begin{remark}\label{rmk:stability_better_rate}
    Note that we obtain a better rate $J^{-q}$ in \cref{thm:stab_particle} than what we had in \cref{thm:uit-mfl} (rate $J^{-1}$).
    This is because for the stability estimate we do not need to estimate the error on the weighted mean stemming from the difference between the empirical measure composed of i.i.d.\  sampled mean-field particles and the mean-field solution $\bar{\rho}$.
    This error yields the usual Monte Carlo rate of $J^{-1}$ as per \cref{lem:convergence_weighted_mean_iid-2}.
\end{remark}

\subsection{Proof of \texorpdfstring{\cref{thm:microscopic-concentration-around-consensus-point}}{Theorem~\ref{thm:microscopic-concentration-around-consensus-point}}}
\label{sec:proof:microscopic-concentration-around-consensus-point}

We first prove that almost surely, $\meanx{t}$ is a Cauchy sequence in time.
\begin{lemma}
    \label{claim:Cauchy-sequence}
    Let~\cref{assumption:bounded} hold. Consider the CBO dynamics \eqref{eq:cbo} with initial condition $(\xn{j}_0)_{j=1}^J\sim \rho^J_0 $, where $\rho^J_0 \in \mathcal P_{\rm sym} (\real^{dJ})$ is such that
    \(
        \expect \empmomentp{2}{0} < +\infty\,.
    \)
    Assume that $\sigma>0$ is sufficiently small to ensure $\edecay{2}>0$,
    and let $\gamma\in \bigl(0, \frac{\edecay{2}}{2}\bigr)$.
    Then almost surely, there exists some time $T$ such that for all times $s,t \ge T$ it holds that
    \begin{align}
        \label{eq:cauchy_sequence}
        \abs*{\meanx{t}-\meanx{s}} \le \e^{-\gamma \min\{s,t\}}.
    \end{align}
\end{lemma}
\begin{proof}
    By taking the average over $j=1,\dots,J$, it follows from~\eqref{eq:cbo} that the empirical mean satisfies the SDE
    \begin{align}
        \label{eq:proof-cauchy:meanx-sde}
        \d \meanx{t}
         & =        - \Bigl(\meanx{t} - \wmx{t}\Bigr) \, \d t
        + \frac{\sigma}{J} \sum_{k=1}^{J} S \Bigl( \xn{k}_t - \wmx{t}\Bigr) \, \d \wn{k}_t.
    \end{align}
    Fix $A > 0$ and define for all $n\in\N$ the set of realizations 
    \begin{align}
        E_n := \set*{ \omega \in \Omega \st \sup_{t\ge n} \abs[\big]{\meanx{t} - \meanx{n}} \ge A \e^{-\gamma n}},
    \end{align}
    where $\Omega$ is the underlying probability space. Fix $n\in\N$. By Markov's inequality,
    it holds that
    \begin{align}
        \label{eq:proba-En-bound}
        \proba\pra*{ E_n } \le \expect\pra*{\sup_{t \ge n} \abs[\Big]{ \meanx{t}-\meanx{n}}} \,
        \frac{\e^{\gamma n}}{A} .
    \end{align}
    For all real numbers $t\ge n$, it follows from~\eqref{eq:proof-cauchy:meanx-sde} that
    \begin{align}
        \label{eq:meanx-t-n-diff-bound}
        \abs*{ \meanx{t} - \meanx{n}} \le \int_n^t \abs*{\meanx{s}-\wmeanx{s}} \, \d s + \abs*{M_t},
    \end{align}
    where $M_t :=\frac{1}{J} \sum_{j=1}^J \int_n^t \scp*{\sigma_j(s),\d \wn{j}_s}$ and
    \begin{align}
        \sigma_j(s) := \begin{cases}
                           \sigma\abs*{\xn{j}_s - \wmeanx{s}}
                           \begin{pmatrix}
                1 & 1 & \dots & 1
            \end{pmatrix}^T                   &
                           \text{ if } S = S^{(i)},             \\
                           \sigma\bra*{\xn{j}_s - \wmeanx{s}} &
                           \text{ if } S = S^{(a)}.
                       \end{cases}
    \end{align}
    Taking the supremum over $t\in[n,\infty)$ in \eqref{eq:meanx-t-n-diff-bound}, we obtain
    \begin{align}
        \label{eq:meanx-sup-bound}
        \sup_{t \ge n} \abs*{ \meanx{t} - \meanx{n}}
        \le \int_n^\infty \abs*{\meanx{s}-\wmeanx{s}} \, \d s + \sup_{t \ge n} \abs*{M_t} =:
        \mathcal T_1 + \mathcal T_2.
    \end{align}
    \paragraph{Bounding $\mathcal T_1$}
    By Fubini--Tonelli, \cref{lem:m-malpha} and from \eqref{eq:moments_summary} (see \cref{lem:decay-centered-moments}), we have
    \begin{align}
        \expect\pra*{\mathcal T_1} & =
        \int_n^\infty \expect \abs[\Big]{\meanx{s}-\wmeanx{s}} \, \d s
        \le \int_n^\infty \e^{\alpha(\overline{f}-\underline{f})} \bra[\Big]{\expect \empmomentp{2}{s}}^{\frac{1}{2}} \, \d s                                                          \\
                                   & \le \e^{\alpha(\overline{f}- \underline{f})} \left(\int_n^\infty \e^{-\frac{\edecay{2}s}{2}} \, \d s\right) \,  \expect \pra*{\empmomentp{2}{0}}^{\frac{1}{2}}
        = \frac{2 \e^{\alpha(\overline{f}-\underline{f})}}{\edecay{2}}  \expect\pra*{\empmomentp{2}{0}}^{\frac{1}{2}} \e^{-\frac{\edecay{2}n}{2}}.
    \end{align} 
    \paragraph{Bounding $\mathcal T_2$}
    For each $T\ge n$ we obtain from \cref{lem:concentration-ineq} and $\abs*{\sigma_j(s)}^2 \le \tau(S) \abs{\xn{j}_s - \wmeanx{s}}^2$, that
    \begin{align}
        \expect \pra*{ \sup_{T\ge t \ge n} M_t^2} & \le \frac{\cbdg{2}}{J^2}\sum_{j=1}^J \int_n^T \expect{\abs[\big]{\sigma_j(s)}^2} \d s
        \le \frac{\cbdg{2}\tau(S)}{J^2} \sum_{j=1}^J \int_n^T \expect{\abs[\big]{\xn{j}_s - \wmeanx{s}}^2} \, \d s.
    \end{align}
    \cref{lem:m-malpha} and \eqref{eq:moments_summary} imply that
    \begin{align}
        \frac{1}{J}\sum_{j=1}^J \expect{\abs[\Big]{\xn{j}_s - \wmeanx{s}}^2} &
        \le \frac{1}{J}\sum_{j=1}^J 2 \expect \abs[\Big]{ \xn{j}_s - \meanx{s}
        }^2 + 2 \expect \abs[\Big]{\meanx{s} - \wmeanx{s}}^2
        \\&
           \le 2\bra*{1 + \e^{\alpha(\overline{f}-\underline{f})}} \expect \empmomentp{2}{s}
        \le 2\bra*{1 + \e^{\alpha(\overline{f}-\underline{f})}} \expect \empmomentp{2}{0} \e^{-\edecay{2}s}.
    \end{align}
    This gives
    \begin{align}
        \expect \pra*{ \sup_{T\ge t\ge n} M_t^2} & \le
        \frac{2\cbdg{2}}{\edecay{2}}  \tau(S)\bra*{1 + \e^{\alpha(\overline{f}-\underline{f})}} \frac{\expect \empmomentp{2}{0}}{J}\e^{-\edecay{2}n}.
    \end{align}
    By the monotone convergence theorem, we obtain
    \begin{align}
        \expect \pra*{\mathcal T_2} \le \expect \pra*{\mathcal T_2^2}^{\frac{1}{2}}
        \le C_{\mathcal T_2} \frac{\expect \empmomentp{2}{0}^{\frac{1}{2}}}{\sqrt{J}} \e^{-\frac{\edecay{2}n}{2}},
        \quad \text{where }
        C_{\mathcal T_2} := \bra*{\frac{2\cbdg{2}}{\edecay{2}}  \tau(S)\bra*{1+ \e^{\alpha(\overline{f}-\underline{f})}}}^{\frac{1}{2}}.
    \end{align}
    \paragraph{Conclusion}
    Altogether, we have from \eqref{eq:proba-En-bound} and \eqref{eq:meanx-sup-bound} that
    \begin{align}
        \proba\pra*{ E_n } \le \expect\pra*{\sup_{t \ge n} \abs*{ \meanx{t}-\meanx{n}}} \frac{\e^{\gamma n}}{A} \le
        \bra*{\frac{2\e^{\alpha(\overline{f}-\underline{f})}}{\edecay{2}} + \frac{C_{\mathcal T_2}}{\sqrt{J}}} \expect\pra*{\empmomentp{2}{0}}^{\frac{1}{2}} \frac{\e^{-\bigl(\frac{\edecay{2}}{2}-\gamma\bigr) n}}{A}.
    \end{align}
    Since $\gamma < \frac{\edecay{2}}{2}$, it is clear that
    \(
        \sum_{n=1}^{+\infty} \proba\pra*{ E_n } < + \infty.
    \)
    Thus, by the Borel--Cantelli lemma,
    it holds almost surely that
    \begin{align}
        \exists N \in \N \, \text{ s.t. } \, \forall n
        \in \N \text{ with } \, n\ge N : \quad \sup_{t\geq n} \abs*{\meanx{t}-\meanx{n}} \le A \e^{-\gamma n}.
    \end{align}
    It follows that,
    almost surely,
    it holds for the same $N = N(\omega)$ that
    \begin{align}
        \forall s,t \in [N, \infty): \quad
        \abs*{\meanx{t}-\meanx{s}}
         & \le
        \abs*{\meanx{s}-\meanx{n(s,t)}}
        +
        \abs*{\meanx{t}-\meanx{n(s,t)}} \\
         & \leq
        2 \sup_{u \geq n(s,t)} \abs*{\meanx{u}-\meanx{n(s,t)}}
        \leq 2A \e^{- \gamma n(s, t)}
        \leq 2A \e^{\gamma - \gamma \min\{s,t\}} \, ,
    \end{align}
    where $n(s, t) = \lfloor \min\{s, t\} \rfloor$.
    By setting $A = \frac{\e^{-\gamma}}{2}$, we obtain the conclusion~\eqref{eq:cauchy_sequence}.
\end{proof}

In order to prove~\eqref{eq:as_conv_wmean} and~\eqref{eq:as_conv_particles}, we show almost sure exponential decay of the second centered moment.
\begin{lemma}
    \label{claim:empirical-cov-conv-to-0-almost-surely}
    Let $f$ satisfy \cref{assumption:bounded}.
    Consider the CBO dynamics~\eqref{eq:cbo} with i.i.d.\ initial conditions sampled from some $\mfldis_0 \in \mathcal P_{4}(\real^d)$,
    and assume that $\sigma>0$ is sufficiently small to ensure $\edecay{2},\edecay{4}>0$.
    Fix $\kappa \in \bra*{0,\min\bigl\{\edecay{2}, \frac{\edecay{4}}{2}\bigr\}}$. Then, almost surely, there exists some time $T\ge 0$ such that
    \begin{align}
        \forall t \geq T, \qquad
        \empmomentp{2}{t} \le \e^{-\kappa t}\empmomentp{2}{0}.
    \end{align}
    That is, the empirical covariance converges exponentially to zero as $t\to\infty$, almost surely.
\end{lemma}

\begin{proof}
    Fix $\widetilde \kappa \in \bra*{\kappa, \min\bigl\{\lambda_2,\frac{\edecay{4}}{2}\bigr\}}$.
    The assumptions of \cref{prop:bound-on-bad-set-general-noise} are satisfied with $q = 2$,
    and so for any $A > 0$ it holds up to an irrelevant constant factor on the right-hand side that
    \[
        \proba \left[ \sup_{t \geq 0} \e^{\widetilde \kappa t} \empmomentp{2}{t} \geq \expect \Bigl[ \empmomentp{2}{0} \Bigr] + A \right]
        \lesssim A^{-2}  .
    \]
    Since $\expect [X] = \int_{0}^{\infty} \proba [X > s] \, \d s$ for any $[0, \infty)$-valued random variable~$X$,
    it holds that
    \[
        \expect \left[ \sup_{t \geq 0} \e^{\widetilde \kappa t} \empmomentp{2}{t} \right]
        < + \infty.
    \]
    Therefore $\sup_{t \geq 0} \e^{\widetilde \kappa t} \empmomentp{2}{t}$ is almost surely finite,
    and so, almost surely,
    \[
        \e^{\kappa t} \empmomentp{2}{t}
        \leq
        \e^{- (\widetilde \kappa - \kappa) t} \left( \e^{\widetilde \kappa t} \empmomentp{2}{t} \right)
        \xrightarrow[t \to \infty]{} 0,
    \]
    which concludes the proof.
\end{proof}

\begin{remark}
        Note that for the claim in \cref{claim:Cauchy-sequence} to hold, we do not require the initial samples to be i.i.d.\ 
        as the result follows without using the concentration inequality \cref{prop:bound-on-bad-set-general-noise},
        unlike \cref{claim:empirical-cov-conv-to-0-almost-surely}.
    \end{remark}

\begin{proof} 
    [Proof of~\cref{thm:microscopic-concentration-around-consensus-point}]
    The Cauchy property stated in~\cref{claim:Cauchy-sequence} implies that
    the limit $\lim_{t \to \infty} \meanx{t}$ exists almost surely.
    We divide the rest of the proof into several steps.

    \paragraph{Proof of~\cref{item:almost-sure-exponential-concentration}}
    To prove~\eqref{eq:as_conv_mean}, fix~$\gamma \in \bigl(0, \frac{\edecay{2}}{2}\bigr)$.
    By using~\cref{claim:Cauchy-sequence}
    and taking the limit~$t \to +\infty$ in~\eqref{eq:cauchy_sequence},
    we deduce that almost surely, there exists $T > 0$ such that for all~$s \geq T$ it holds that
    \[
        \abs*{\meanx{s} - \cpt} \le \e^{- \gamma s}.
    \]
    The statement~\eqref{eq:as_conv_mean} immediately follows.
    In order to prove~\eqref{eq:as_conv_wmean},
    fix $\widetilde \gamma \in \Bigl(0, \min \{ \frac{\edecay{2}}{2}, \frac{\edecay{4}}{4} \} \Bigr)$.
    By the triangle inequality and~\cref{lem:m-malpha},
    it holds that
    \begin{align}
        \e^{\widetilde \gamma t} \abs*{\wmx{t} - \cpt} 
        & \leq \e^{\widetilde \gamma t} \abs*{\wmx{t} - \meanx{t}}  +
               \e^{\widetilde \gamma t} \abs*{\meanx{t} - \cpt} \notag                                                                                           \\
        & \leq  \e^{\frac{\alpha(\overline{f} - \underline{f})}{2}} \left( \e^{2\widetilde \gamma t}\empmomentp{2}{t} \right)^{\frac{1}{2}} +  
                \e^{\widetilde \gamma t} \abs*{\meanx{t} - \cpt}
        \xrightarrow[t \to \infty]{} 0 \, ,
        \label{eq:pivoting-mean}
    \end{align}
    where the limit of the first term is justified by~\cref{claim:empirical-cov-conv-to-0-almost-surely},
    while the limit of the second term follows from~\eqref{eq:as_conv_mean}.
    In order to prove~\eqref{eq:as_conv_particles},
    we use \cref{claim:empirical-cov-conv-to-0-almost-surely} and~\eqref{eq:as_conv_mean} to find that almost surely,
    \begin{align}
        \e^{\widetilde \gamma t} \abs*{ \xn{j}_t - \cpt} 
        & \le \abs*{\xn{j}_t - \meanx{t}} + \abs*{\meanx{t} - \cpt} \\
        & \le  J \left(\e^{2 \widetilde \gamma t} \empmomentp{2}{t} \right)^{\frac{1}{2}} + \e^{\widetilde \gamma t} \abs*{\meanx{t}- \cpt}
        \xrightarrow[t \to \infty]{} 0 \, .
    \end{align}

    \paragraph{Proof of~\cref{item:mean-square-exponential-concentration}}
    From~\eqref{eq:proof-cauchy:meanx-sde}, we have that
    \begin{align}
        \notag
        \expect \abs*{ \meanx{t} - \meanx{s}}^2 & \le
        2 \expect \abs*{\int_s^t \meanx{r} - \wmx{r} \, \d r}^2 
        + 2 \expect \abs*{ \frac{\sigma}{J} \sum_{k=1}^J \int_s ^t S\Bigl(\xn{k}_r - \wmeanx{r}\Bigr) \, \d \wn{k}_r}^2  \\
        \label{eq:terms_1_and_2}
        &\eqcolon \mathcal A_1 + \mathcal A_2.
    \end{align}
    By \eqref{eq:holder_trick} and \cref{lem:m-malpha},
    together with \eqref{eq:moments_summary} ( see~\cref{lem:decay-centered-moments}), we have for any $\ell \in (0, \edecay{2})$ that
    \begin{align}
    \mathcal A_1 
    &\le \frac{2}{\ell} \int_s^t \e^{\ell r} \expect \abs*{\meanx{r} - \wmx{r}}^2 \, \d r \\
    &\le \frac{2\e^{\alpha (\overline{f}-\underline{f})}}{\ell} \int_s^t \e^{\ell r} \expect\empmomentp{2}{r} \, \d r
    \le \frac{2\e^{\alpha (\overline{f}-\underline{f})}}{\ell (\edecay{2} - \ell)}  \expect\empmomentp{2}{0} \e^{-(\edecay{2} - \ell) s} \, .
    \end{align}
    By the Itô isometry, the second term in on the right-hand side of~\eqref{eq:terms_1_and_2} equals
    \begin{align}
        \mathcal A_2 & = \frac{2\sigma^2}{J^2} \sum_{k=1}^J \expect \int_s^t 
            \Bigl\lVert S\Bigl(\xn{k}_r - \wmx{r}\Bigr) \Bigr\rVert_{\rm F}^2 \d r
        \le \frac{2\sigma^2\tau(S)}{J^2} \sum_{k=1}^J \expect \int_s^t \abs*{
            \xn{k}_r - \wmx{r}}^2 \d r,
    \end{align}
    and the integrand on the right-hand side can be bounded as
    \begin{align}
        \sum_{k=1^J}\expect \abs*{\xn{k}_r - \wmx{r}}^2 & \le 2 \sum_{k=1^J} \expect \abs*{\xn{k}_r - \meanx{r}}^2 + 2J  \expect \abs*{\meanx{r} - \wmx{r}}^2
        \\ & \le 2 J\bra*{1 + \e^{\alpha (\overline f- \underline f)}}\expect \empmomentp{2}{r}  \le  2 J\bra*{1 + \e^{\alpha (\overline f- \underline f)}}\expect \empmomentp{2}{0} \e^{- \edecay{2}r},
    \end{align}
    where we used~\cref{lem:m-malpha} and \eqref{eq:moments_summary}.
    Therefore,
    \begin{align}
        \mathcal A_2 \le \frac{4\sigma^2\tau(S)}{\edecay{2} J} \bra*{1 + \e^{\alpha (\overline f- \underline f)}}
        \expect \empmomentp{2}{0} \e^{- \edecay{2}s} \, .
    \end{align}
    Combining the bounds on~$\mathcal A_1$ and~$\mathcal A_2$, 
    we obtain the inequality
    \begin{align}
        \label{eq:empirical-mean-diff}
        \expect \abs*{ \meanx{t} - \meanx{s}}^2 \le
        \frac{2\e^{\alpha (\overline{f}-\underline{f})}}{\ell(\edecay{2} - \ell)}  \expect\empmomentp{2}{0} \e^{- (\edecay{2} - \ell) s}
        + \frac{4\sigma^2\tau(S)}{\edecay{2} J} \bra*{1 + \e^{\alpha (\overline f- \underline f)}}
        \expect \empmomentp{2}{0} \e^{- \edecay{2}s} \, .
    \end{align}
    Together with \cref{lem:uit-raw-moments-2norm}, we can apply the dominated convergence theorem to
    obtain~\eqref{eq:L2-conv-means} by taking  $t\to\infty$ in~\eqref{eq:empirical-mean-diff}. 
    In order to prove~\eqref{eq:L2-conv-wmeans}, we use that
    \begin{align}
        \expect\abs*{\wmx{t}-\mathcal M ^X}^2 \le 2 \expect\abs*{\wmx{t} - \meanx{t}}^2 + 2 \expect \abs*{\meanx{t} - \mathcal M ^X}^2\, ,
    \end{align}
    together with the inequality $\expect\abs*{\wmx{t} - \meanx{t}}^2 \le \e^{\alpha(\overline{f}-\underline{f})} \expect\empmomentp{2}{0} \e^{-\edecay{2}t}$.
    Finally, to prove~\eqref{eq:L2-conv-particles},
    we use that the initial joint distribution of the particles is symmetric under permutations of the particles to find that
    \begin{align}
        \expect \abs*{\xn{j}_t - \mathcal M ^X}^2
        \le 2 \expect \empmomentp{2}{t} + 2 \expect \abs*{\meanx{t} - \cpt}^2.
    \end{align}

    \paragraph{Proof of~\cref{item:consensus-point-close-to-target}}
    The decomposition
    \begin{align}
        \nmx{t} - \cptmfl = \frac{1}{J}\sum_{j=1}^J\bra*{\xn{j}_t - \xnl{j}_t} + \frac{1}{J}\sum_{j=1}^J\bra*{\xnl{j}_t - \cptmfl}
    \end{align}
    implies by the triangle and Jensen inequalities that
    \begin{align}
        \left( \expect \abs*{\nmx{t} - \cptmfl}^2 \right)^{\frac{1}{2}} 
        \le \left( \expect \frac{1}{J} \sum_{j=1}^J \abs*{\xn{j}_t - \xnl{j}_t}^2 \right)^{\frac{1}{2}} 
        + \left( \expect \frac{1}{J} \sum_{j=1}^J \abs*{\xnl{j}_t - \cptmfl}^2 \right)^{\frac{1}{2}}.
    \end{align}
   Therefore, we obtain from \cref{thm:uit-mfl} that
    \begin{align}
        \label{eq:empirical-mean-vs-target}
        \left( \expect \abs*{\nmx{t}- \cptmfl}^2 \right)^{\frac{1}{2}}
         & \le \left( \frac{C_{\rm MFL}}{J} \right)^{\frac{1}{2}} + \wasserstein_2\bra[\big]{\mfldis_t, \delta_{\cptmfl}}.
    \end{align}
    In order to apply the dominated convergence theorem, 
    note that
    \[
        \abs*{\nmx{t}- \cptmfl}^2
        \leq 2\abs*{\nmx{t}}^2 + 2\lvert \cptmfl \rvert^2
        \leq \frac{2}{J} \sum_{j=1}^{J} \bigl\lvert \xn{j}_{t} \bigr\rvert^2 + 2 \lvert \cptmfl \rvert^2
        \leq \frac{2}{J} \sum_{j=1}^{J} \sup_{s \geq 0} \bigl\lvert \xn{j}_{s} \bigr\rvert^2 + 2 \lvert \cptmfl \rvert^2 \, .
    \]
    The random variable on the right-hand side is independent of~$t$ and also furthermore integrable by \cref{lem:uit-raw-moments-2norm}.
    Therefore, since $\nmx{t}$ converges almost surely to $\cpt$,
    we can apply the dominated convergence theorem and let $t\to\infty$ in \eqref{eq:empirical-mean-vs-target}, and  since $
        \wasserstein_2\bra[\big]{\mfldis_t, \delta_{\cptmfl}} \rightarrow 0 
    $ as $t \to \infty$ by assumption,
    we get~\eqref{eq:consensus-point-vs-target}.
\end{proof}

\section{Conclusion and perspectives}\label{sec:conclusion}
To our knowledge, this work provides the first quantitative uniform-in-time propagation of chaos result for the original CBO dynamics. Our results rely on a number of novel auxiliary estimates, including a local Lipschitz estimate for the difference between the mean and the weighted mean, exponential decay of centered moments and several concentration inequalities.
These technical hurdles have been avoided in recent related work \cite{huang2024uniformintimemeanfieldlimitestimate,bayraktar2025uniformintimeweakpropagationchaos} by modifying the dynamics.
Combining the raw-moment control with concentration inequalities and standard Monte Carlo estimates yields the desired uniform-in-time propagation of chaos for the unmodified CBO algorithm.

Our analysis has limitations that naturally suggest future work. We treat objective functions $f$
that are globally Lipschitz and bounded. It is possible to generalize to locally Lipschitz and globally bounded objectives; removing the boundedness assumption is more challenging, and is an important next step, since many practical cost functions are unbounded. Our current setting can still be applied in these cases by imposing a cutoff on a domain of interest. In addition, our quantitative rate exhibits an unfavorable dependence on $\alpha$ appearing as $e^{\alpha(\overline{f} - \underline{f})}$. It would be valuable to refine the estimates to eliminate this exponential dependence. 

Building on the work in this manuscript, \cite{choi2025modifiedconsensusbasedoptimizationmodel} shows a uniform-in-time propagation of chaos result for a modified CBO dynamics for which our approach can be adapted to treat a more general class of objective functions. This is achieved by adding a strictly positive constant to the denominator of the weighted mean, avoiding degeneracy.

Beyond generalizing the assumptions on the objective and improving parameter dependence, it would be useful to extend this framework to other interacting particle systems. A promising direction is the second-order (or kinetic) CBO model, which is obtained by including momentum effects, and has been studied in, e.g., \cite{sara_pso,Huang_2023}; establishing a uniform-in-time propagation of chaos result for that model is an open question. Related variants, such as jump–diffusion formulations \cite{Kalise_2023}, also appear amenable to our approach. Finally, developing an analogous uniform-in-time theory directly for time-discrete schemes would tighten the link between analysis and implementations in an effort towards a full theory for consensus-based approaches as they are used in applications.

\section{Proof of auxiliary results}
\label{sec:aux}

We will make use of the following lemma frequently:

\begin{lemma}
    \label{lem:m-malpha}
    Let $q\ge 2$ and let \cref{assumption:bounded} hold. 
    Then for any vector norm~$| \placeholder |$ it holds that
    \begin{align}
        \forall \mu \in \mathcal P_{q}, \qquad
        \Bigl\lvert \wm(\mu) - \mathcal M(\mu) \Bigr\rvert^{q}
     &
     \leq \e^{ \alpha (\overline f - \underline f)} \int  \bigl\lvert x - \mathcal M(\mu) \bigr\rvert^{q} \, \mu(\d x).
    \end{align}
\end{lemma}

\begin{proof}
    This follows directly from Jensen's inequality by estimating
    \begin{align}
        \Bigl\lvert \wm(\mu) - \mathcal M(\mu) \Bigr\rvert^{q}
     & = \left\lvert \frac{\displaystyle \int  \Bigl(x - \mathcal M(\mu) \Bigr) \e^{- \alpha f(x)} \, \mu(\d x)}{\displaystyle\int \e^{-\alpha f(x)} \, \mu(\d x)} \right\rvert^{q}\\
     & \leq \frac{\displaystyle\int  \left\lvert x - \mathcal M(\mu) \right\rvert^{q} \e^{- \alpha f(x)} \, \mu(\d x)}{
        \displaystyle\int \e^{-\alpha f(x)} \, \mu(\d x)}
     \leq \e^{\alpha (\overline f - \underline f)} \int  \bigl\lvert x - \mathcal M(\mu) \bigr\rvert^{q} \, \mu(\d x)\,.  \qedhere
    \end{align}
\end{proof}

\subsection{Moment bounds}

In this section, we prove exponential decay for the centered moments and uniform-in-time raw moment bounds for both the interacting particle system and the mean-field process. 

\subsubsection{Decay of centered moments: interacting particle system}
\label{sec:proof_decay-centered-moments-pnorm-general-noise}

\begin{lemma}
    [Exponential decay of centered moments]
    \label{lem:decay-centered-moments}
    Let $p \geq 2$ and define
    \[
        \empmomentp{p}{t} =  \frac{1}{J} \sum_{j=1}^{J} \Bigl\lvert \xn{j}_t - \nmx{t} \Bigr\rvert^p\,.
    \]
    Under~\cref{assumption:bounded},
    and for any initial law~$\rho^J_0 \in \mathcal P_{\rm sym} (\real^{dJ})$
    it holds that
    \begin{equation}
        \label{eq:lemma_moment_decay}
        \expect \left[ \empmomentp{p}{t} \right] \leq \expect \left[ \empmomentp{p}{0} \right] \e^{-\edecay{p} t},
        \qquad
        \edecay{p} :=   p \left[1-  \frac{1}{2}\bigl(p-2+\tau(S) \bigr) \sigma^2 \Bigl(1 + \e^{\frac{\alpha}{p}(\overline f - \underline f)} \Bigr)^2\right] .
    \end{equation}
    In particular, for sufficiently small~$\sigma$,
    it holds that~$\expect \left[ \empmomentp{p}{t} \right] \to 0$ in the limit as~$t \to \infty$.
\end{lemma}

\begin{proof}[Proof of \cref{lem:decay-centered-moments}]
    We assume that $\expect \left[ \empmomentp{p}{0} \right] < \infty$,
    as otherwise~\eqref{eq:lemma_moment_decay} is trivially satisfied.
    By It\^o's formula,
    it holds that
    \begin{align}
        \d \meanx{t}
        &=
        - \frac{1}{J} \sum_{k=1}^{J} \Bigl(\xn{k}_t - \wmx{t}\Bigr) \, \d t
        + \frac{\sigma}{J} \sum_{k=1}^{J} S\Bigl( \xn{k}_t - \wmx{t} \Bigr) \, \d \wn{k}_t \\
        &=
        - \Bigl(\meanx{t} - \wmx{t}\Bigr) \, \d t
        + \frac{\sigma}{J} \sum_{k=1}^{J} S \Bigl( \xn{k}_t - \wmx{t}\Bigr) \, \d \wn{k}_t.
    \end{align}
    It follows that
    \begin{align}
        \d \Bigl( \xn{j}_t - \meanx{t} \Bigr)
         & =
        - \left(\xn{j}_t - \meanx{t}\right) \, \d t
        + \sigma \left( 1 - \frac{1}{J} \right) S \Bigl( \xn{j}_t - \wmx{t} \Bigr) \, \d \wn{j}_t \\
         & \qquad - \frac{\sigma}{J} \sum_{k\neq j}^{J} S \Bigl( \xn{k}_t - \wmx{t}\Bigr) \, \d \wn{k}_t.
    \end{align}
    In order to formally justify the application of It\^o's formula following,
    it is useful to recall that,
    for any $x, \delta \in \real^d$,
    the following equation holds by a Taylor expansion:
    \begin{align}\label{eq:taylor_exp_1}
        |x + \delta|^p
        = |x|^p + p |x|^{p-2} \langle x, \delta \rangle
        + \frac{p(p-2)}{2} |x|^{p-4} \langle x, \delta \rangle^2
        + \frac{p}{2} |x|^{p-2} \langle \delta, \delta \rangle
        + \mathcal O\bigl(\lvert \delta \rvert^3\bigr).
    \end{align}
    Therefore, it holds that
    \begin{align}
        \d \Bigl\lvert  \xn{j}_t - \meanx{t} \Bigr \rvert^p
         &=
         - p \Bigl\lvert  \xn{j}_t - \meanx{t} \Bigr \rvert^{p} \, \d t \\
         &\qquad
         + \frac{p\bigl(p-2\bigr)}{2} \sigma^2 \left( 1 - \frac{1}{J} \right)^2 \Bigl\lvert  \xn{j}_t - \meanx{t} \Bigr \rvert^{p-4} \, \left\lvert  S\Bigl(\xn{j}_t - \wmx{t}\Bigr) \Bigl(\xn{j}_t - \nmx{t}\Bigr) \right\rvert^2 \, \d t \\
         &\qquad
         + \frac{p\tau(S)}{2} \sigma^2 \left( 1 - \frac{1}{J} \right)^2 \Bigl\lvert  \xn{j}_t - \meanx{t} \Bigr\rvert^{p-2} \, \left\lvert \xn{j}_t - \wmx{t}\right\rvert^2 \, \d t \\
         &\qquad
         + \frac{p\bigl(p-2\bigr)}{2J^2} \sigma^2 \sum_{k \neq j}^{J} \Bigl\lvert  \xn{j}_t - \meanx{t} \Bigr\rvert^{p-4} \, \left\lvert  S\Bigl(\xn{k}_t - \wmx{t}\Bigr) \Bigl(\xn{j}_t - \nmx{t}\Bigr) \right\rvert^2 \, \d t \\
         &\qquad
         + \frac{p\tau(S)}{2J^2} \sigma^2 \sum_{k\neq j}^{J} \Bigl\lvert  \xn{j}_t - \meanx{t} \Bigr \rvert^{p-2} \, \left\lvert \xn{k}_t - \wmx{t}\right\rvert^2 \, \d t \\
         &\qquad
         + p \sigma  \Bigl\lvert  \xn{j}_t - \meanx{t} \Bigr \rvert^{p-2} \left\langle \xn{j}_t - \meanx{t}, S \Bigl( \xn{j}_t - \wmx{t} \Bigr) \, \d \wn{j}_t \right\rangle \\
         &\qquad
         - \frac{p \sigma}{J} \sum_{k = 1}^{J} \Bigl\lvert  \xn{j}_t - \meanx{t} \Bigr \rvert^{p-2} \left\langle \xn{j}_t - \meanx{t}, S \Bigl( \xn{k}_t - \wmx{t} \Bigr) \, \d \wn{k}_t \right\rangle.
    \end{align}
    Using the elementary inequality $\sum_{i=1}^{d} x_i^2 y_i^2 \leq \sum_{i=1}^{d} x_i^2 \sum_{j=1}^{d} y_j^2$,
    we have for any $j,k \in \range{1}{J}$ that
    \begin{equation}
        \label{eq:bound_product_centered_moment}
        \left\lvert  S\Bigl(\xn{k}_t - \wmx{t}\Bigr) \Bigl(\xn{j}_t - \nmx{t}\Bigr) \right\rvert^2
        \leq
        \Bigl\lvert \xn{j}_t - \nmx{t} \Bigr\rvert^2 \Bigl\lvert \xn{k}_t - \wmx{t} \Bigr\rvert^2.
    \end{equation}
    Thus, we obtain the upper bound
    \begin{align}
        \d \Bigl\lvert  \xn{j}_t - \meanx{t} \Bigr \rvert^p
         & \le
         - p \Bigl\lvert  \xn{j}_t - \meanx{t} \Bigr \rvert^{p} \, \d t \\
         &\qquad
         + \frac{p\bigl(p-2 + \tau(S) \bigr)}{2} \sigma^2 \left( 1 - \frac{2}{J} \right) \Bigl\lvert  \xn{j}_t - \meanx{t} \Bigr \rvert^{p-2} \left\lvert \xn{j}_t - \wmx{t} \right\rvert^2 \, \d t \\
         &\qquad
         + \frac{p\bigl(p-2 + \tau(S)\bigr)}{2J^2} \sigma^2 \sum_{k = 1}^{J} \Bigl\lvert  \xn{j}_t - \meanx{t} \Bigr\rvert^{p-2} \, \Bigl\lvert \xn{k}_t - \wmx{t}\Bigr\rvert^2  \, \d t \\
         &\qquad
         + p \sigma  \Bigl\lvert  \xn{j}_t - \meanx{t} \Bigr \rvert^{p-2}
         \left\langle
             \xn{j}_t - \meanx{t},
             S \Bigl( \xn{j}_t - \wmx{t} \Bigr) \, \d \wn{j}_t
         \right\rangle \\
         &\qquad
         - \frac{p \sigma}{J} \sum_{k = 1}^{J}
         \Bigl\lvert  \xn{j}_t - \meanx{t} \Bigr \rvert^{p-2}
         \left\langle
        \xn{j}_t - \meanx{t},
         S\Bigl( \xn{k}_t - \wmx{t} \Bigr) \, \d \wn{k}_t
         \right\rangle.
    \end{align}
    Summing over all the particles and dividing by~$J$,
    we deduce that
    \begin{align}
        \d \empmomentp{p}{t}
         & \le
         - p \empmomentp{p}{t} \, \d t \\
         &\qquad
         + \frac{p\bigl(p-2+\tau(S)\bigr)}{2J} \sigma^2 \left( 1 - \frac{2}{J} \right) \sum_{j=1}^{J} \Bigl\lvert  \xn{j}_t - \meanx{t} \Bigr \rvert^{p-2} \, \left\lvert \xn{j}_t - \wmx{t}\right\rvert^2 \, \d t \\
         &\qquad
         + \frac{p\bigl(p-2+\tau(S)\bigr)}{2J^3} \sigma^2 \sum_{j=1}^{J} \sum_{k=1}^{J} \left\lvert \xn{j}_t - \nmx{t}\right\rvert^{p-2}  \left\lvert \xn{k}_t - \wmx{t}\right\rvert^2 \, \d t \\
         &\qquad
         + \frac{p \sigma}{J} \sum_{j=1}^{J} \Bigl\lvert  \xn{j}_t - \meanx{t} \Bigr \rvert^{p-2}
         \left\langle
             \xn{j}_t - \meanx{t},
             S\Bigl(\xn{j}_t - \wmx{t} \Bigr) \, \d \wn{j}_t
         \right\rangle
         \\
         &\qquad
         - \frac{p \sigma}{J^2}
         \sum_{j=1}^{J}
         \sum_{k=1}^{J}
         \Bigl\lvert  \xn{j}_t - \meanx{t} \Bigr \rvert^{p-2}
         \left\langle
             \xn{j}_t - \meanx{t},
             S\Bigl( \xn{k}_t - \wmx{t} \Bigr) \, \d \wn{k}_t
         \right\rangle.
    \end{align}
    By H\"older's inequality,
    it holds that
    \begin{align}
        \sum_{j=1}^{J} \Bigl\lvert  \xn{j}_t - \meanx{t} \Bigr \rvert^{p-2} \, \left\lvert \xn{j}_t - \wmx{t}\right\rvert^2
        &\leq \biggl( \sum_{j=1}^{J} \Bigl\lvert  \xn{j}_t - \meanx{t} \Bigr \rvert^{p} \biggr)^{\frac{p-2}{p}}
        \biggl(  \sum_{j=1}^{J} \left\lvert \xn{j}_t - \wmx{t}\right\rvert^p \biggr)^{\frac{2}{p}}, \\
        \sum_{j=1}^{J} \sum_{k=1}^{J} \left\lvert \xn{j}_t - \nmx{t}\right\rvert^{p-2}  \left\lvert \xn{k}_t - \wmx{t}\right\rvert^2
        &\leq \biggl( J \sum_{j=1}^{J} \Bigl\lvert  \xn{j}_t - \meanx{t} \Bigr \rvert^{p} \biggr)^{\frac{p-2}{p}}
        \biggl( J \sum_{k=1}^{J} \left\lvert \xn{k}_t - \wmx{t}\right\rvert^p \biggr)^{\frac{2}{p}}.
    \end{align}
    Furthermore, by the triangle inequality,
    we have
    \[
        \biggl( \frac{1}{J} \sum_{j=1}^{J} \left\lvert \xn{j}_t - \wmx{t}\right\rvert^p \biggr)^{\frac{1}{p}}
        \leq \biggl( \frac{1}{J} \sum_{j=1}^{J} \left\lvert \xn{j}_t - \meanx{t}\right\rvert^p \biggr)^{\frac{1}{p}}
        + \Bigl\lvert \meanx{t} - \wmx{t}\Bigr \rvert.
    \]
    By~\cref{lem:m-malpha},
    it holds that $\Bigl\lvert \wm(\emp{t}) - \mathcal M(\emp{t}) \Bigr \rvert^p \leq \e^{\alpha (\overline f - \underline f)} \empmomentp{p}{t}$,
    and so we have
    \begin{align}
        \biggl( \frac{1}{J} \sum_{j=1}^{J} \left\lvert \xn{j}_t - \wmx{t}\right\rvert^p \biggr)^{\frac{1}{p}}
        &\leq \empmomentp{p}{t}^{\frac{1}{p}} + \e^{\frac{\alpha}{p}(\overline f - \underline f)} \empmomentp{p}{t}^{\frac{1}{p}}
        = \Bigl(1 + \e^{\frac{\alpha}{p}(\overline f - \underline f)} \Bigr) \empmomentp{p}{t}^{\frac{1}{p}}.
    \end{align}
    Combining these estimates,
    we deduce that
    \begin{align}
        \d \empmomentp{p}{t}
         &\leq
         - \left( p - \frac{p\bigl(p-2+\tau(S)\bigr)}{2} \sigma^2 \left( 1 - \frac{1}{J} \right) \Bigl(1 + \e^{\frac{\alpha}{p}(\overline f - \underline f)} \Bigr)^2 \right) \empmomentp{p}{t} \, \d t \\
         &\qquad
         + \frac{p \sigma}{J} \sum_{j=1}^{J}
         \Bigl\lvert  \xn{j}_t - \meanx{t} \Bigr \rvert^{p-2}
         \left\langle
             \xn{j}_t - \meanx{t},
             S\Bigl( \xn{j}_t - \wmx{t} \Bigr) \, \d \wn{j}_t
         \right\rangle \\
         &\qquad
         - \frac{p \sigma}{J^2} \sum_{j=1}^{J}\sum_{k=1}^{J}
         \Bigl\lvert  \xn{j}_t - \meanx{t} \Bigr \rvert^{p-2}
         \left\langle
             \xn{j}_t - \meanx{t},
             S\Bigl( \xn{k}_t - \wmx{t} \Bigr) \, \d \wn{k}_t
         \right\rangle.
    \end{align}
    Rewriting this inequality in its integral form,
    and taking the expectation, we obtain that
    \begin{align}
        \expect \left[ \empmomentp{p}{t} \right] \le \expect \left[ \empmomentp{p}{0} \right] - \int_{0}^{t} \edecay{p} \, \expect \left[ \empmomentp{p}{s} \right] \d s.
    \end{align}
    The conclusion then follows from Gr\"onwall's inequality.
\end{proof}

\subsubsection{Decay of centered moments: mean-field process}
\label{sub:mf_consensus}
Here we prove the counterpart of~\cref{lem:decay-centered-moments}
for the mean-field process.

\begin{lemma}[Exponential decay of mean-field centered moments]
    \label{thm:mfl-decay-p-p-cetered-moment}
    Let $p \geq 2$. Suppose that $f\colon \real^d \rightarrow \R$ satisfies \cref{assumption:bounded},
    and that~$\mfldis_0$ has finite moments of all orders.
    Then, for $\left(\xl_t\right)_{t \geq 0}$ that solves~\eqref{eq:mfl_sde_first} we have
    \begin{align}
        &\expect  \Bigl\lvert \xl_t - \expect \xl_t \Bigr\rvert^{p}
        \leq \expect  \Bigl\lvert \xl_0 - \expect \xl_0 \Bigr\rvert^{p} \e^{- \edecay{p} t}
        \quad \text{ for all } t \geq 0\,,
    \end{align}
    where $\edecay{p} >0$ is defined as in \cref{lem:decay-centered-moments}.
\end{lemma}

\begin{proof}
    From \eqref{eq:mfl_sde_first},
    we have $\frac{\d}{\d t} \expect \xl_t = - \bra[\big]{\expect{\widebar{X}_t} - \wm(\mfldis_t)}$.
     Therefore, we obtain
    \begin{align}
        \d \bra*{ \widebar{X}_t - \expect {\widebar{X}_t}}
        = - \bra*{\widebar{X}_t - \expect {\widebar{X}_t}} \, \d t
          + \sigma S\bra[\Big]{\widebar{X}_t -\wm(\mfldis_t)} \, \d  W_t\,,
    \end{align}
    Recalling that \eqref{eq:taylor_exp_1} holds,
    we obtain by It\^o's formula that
    \begin{align}
        \d \abs[\Big]{\widebar{X}_t - \expect {\widebar{X}_t}}^p
        & \le -p \abs[\Big]{\widebar{X}_t - \expect {\widebar{X}_t}}^p \d t
        \\ & \qquad +
        \frac{p(p-2)\sigma^2}{2}
        \left\lvert \widebar{X}_t - \expect \widebar{X}_t \right\rvert^{p-4}
        \Bigl\lvert S\bra[\Big]{\widebar{X}_t - \wm(\mfldis_t)} \bigl(\widebar{X}_t - \expect \widebar{X}_t\bigr) \Bigr\rvert^2 \, \d t
        \\ & \qquad
        + \frac{p\sigma^2}{2} \tau(S)
        \left\lvert \widebar{X}_t - \expect \widebar{X}_t \right\rvert^{p-2}
        \Bigl\lvert \widebar{X}_t - \wm(\mfldis_t) \Bigr\rvert^2 \, \d t \\
           & \qquad + p \sigma \left\lvert \widebar{X}_t - \expect \widebar{X}_t \right\rvert^{p-2}
        \Bigl\langle
            \widebar{X}_t - \expect \widebar{X}_t,
            S\bra[\Big]{\widebar{X}_t - \wm(\mfldis_t)} \, \d W_t
        \Bigr\rangle.
    \end{align}
    Using \eqref{eq:bound_product_centered_moment} and H\"older's inequality similarly as in the proof of~\cref{lem:decay-centered-moments},
    then taking the expectation,
    we obtain
    \begin{align}
        \frac{\d}{\d t} \expect \abs[\Big]{\widebar{X}_t - \expect {\widebar{X}_t}}^p
        & \le -p \expect \left[ \abs[\Big]{\widebar{X}_t - \expect {\widebar{X}_t}}^p \right]
        \\ & \qquad +
        \frac{1}{2} p\bigl(p-2 + \tau(S)\bigr)\sigma^2
        \expect \left[ \left\lvert \widebar{X}_t - \expect \widebar{X}_t \right\rvert^{p} \right]^{\frac{p-2}{p}}
        \expect \left[ \left\lvert \widebar{X}_t - \wm(\mfldis_t) \right\rvert^{p} \right]^{\frac{2}{p}}\,.
    \end{align}
    For the last factor of the second term on the right-hand side,
    we have by~\cref{lem:m-malpha} that
    \begin{align}
        \expect \pra*{ \abs[\Big]{ \widebar{X}_t - \wm(\mfldis_t)}^{p} }^{\frac{1}{p}}
        &\le \expect \pra*{ \abs[\Big]{ \widebar{X}_t - \expect \widebar{X}_t}^{p} }^{\frac{1}{p}} + \abs[\Big]{ \mathcal{M}\bra*{\mfldis_t}- \wm(\mfldis_t)} \\
        &\le \bra*{ 1+ \e^{\frac{\alpha}{p} \bra*{\overline{f} - \underline{f}}}}
        \expect \pra*{ \abs[\Big]{ \widebar{X}_t - \expect \widebar{X}_t}^{p}}^{\frac{1}{p}}.
    \end{align}
    In summary, we obtain
    \begin{align}
        \frac{\d}{\d t} \expect \pra*{ \abs[\Big]{ \widebar{X}_t - \expect \widebar{X}_t}^{p}}
        & \le - p \bra*{1- \frac{1}{2} \bigl(p-2 + \tau(S)\bigr) \sigma^2 \bra*{1+ \e^{\frac{\alpha}{p} \bra*{\overline{f} - \underline{f}}}}^2} \expect \pra*{ \abs[\Big]{ \widebar{X}_t - \expect \widebar{X}_t}^{p} },
    \end{align}
    from which the claim follows.
\end{proof}

\begin{remark}
    We presented a self-contained proof of the result for the reader's convenience,
    and because intermediate calculations will be reused in the proof of~\cref{prop:bound-on-bad-set-synchronous-mf}.
    However, note that~\cref{thm:mfl-decay-p-p-cetered-moment} can also be obtained by combining the finite-time mean-field limit result from~\cite[Theorem 2.6]{gerber2023meanfield}
    with the moment decay estimate for the interacting particle system shown in \cref{lem:decay-centered-moments}.
    Here we give a short sketch of this argument.
    Fix $J\in\N$ and consider particles $\xn{1}_t, \dots, \xn{J}_t$ evolving according to~\eqref{eq:cbo} with i.i.d.\ initial conditions $\xn{j}_0 \sim \mfldis_0$,
    coupled to i.i.d.\ copies $\xnl{1}_t, \dots, \xnl{J}_t$ of the mean-field dynamics~\eqref{eq:mfl_sde_first}
    with the same initial conditions and the same driving Brownian motions.
    Then, we have
    \begin{align}
    \left(\expect  \abs[\big]{ \xl_t - \expect \xl_t }^p \right)^{\frac{1}{p}}
    & \le   \left(\expect  \Bigl\lvert  \xnl{}_t - \xn{1}_t  \Bigr\rvert^p \right)^{\frac{1}{p}}
    + \left(\expect  \Bigl\lvert  \xn{1}_t - \nmx{t}  \Bigr\rvert^p \right)^{\frac{1}{p}} \\
    & \qquad
    + \left(\expect  \Bigl\lvert  \nmx{t} - \nmxl{t}  \Bigr\rvert^p \right)^{\frac{1}{p}}
    + \left(\expect  \Bigl\lvert  \nmxl{t} - \expect \xl_t  \Bigr\rvert^p \right)^{\frac{1}{p}}.
    \end{align}
    Taking the limit $J\to\infty$, the first, the third and the fourth term on the right-hand side vanish by \cite[Theorem 2.6]{gerber2023meanfield}.
    For the second term on the right-hand side we have by \cref{lem:decay-centered-moments} that
    \begin{align}
        \expect  \Bigl\lvert  \xn{1}_t - \nmx{t}  \Bigr\rvert^p = \expect \left [ \empmomentp{p}{t} \right]
        \le
        \expect \Bigl\lvert  \xnl{1}_0 - \nmxl{0}  \Bigr\rvert^p
        \e^{-\edecay{p} t}.
    \end{align}
    By the reverse triangle inequality,
    it holds that
    \begin{align}
        \abs*{
            \left(\expect \abs*{ \xnl{1}_0 - \nmxl{0} }^p \right)^{\frac{1}{p}}
            - \left(\expect \abs*{ \xnl{1}_0 - \expect \xnl{1}_0 }^p \right)^{\frac{1}{p}}
        }
        \le \left(\expect \abs*{ \nmxl{0} - \expect \xnl{1}_0 }^p \right)^{\frac{1}{p}} \, \xrightarrow[J\to\infty] \,
        0,
    \end{align}
    and so the claim follows.
\end{remark}

\begin{remark}
    Note that the rate~$\edecay{p}$ of exponential decay of the centered moments depends on $\e^{\alpha (\overline f - \underline f)}$,
    leading to stringent restrictions on the noise coefficient~$\sigma$ if $\alpha \gg 1$.
    We mention that the result proved in \cite{carrillo2018analytical}
    establishes exponential decay of centered moments with a rate that enjoys a better dependence on $\alpha$.
    In particular, if $f$ satisfies~\cref{assumption:bounded,assumption:lip} and has uniformly bounded second derivatives,
    it is shown in~\cite[Theorem 4.1]{carrillo2018analytical} that
    \begin{equation}
        \label{eq:refined_moment}
        \expect  \Bigl\lvert \xl_t - \expect \xl_t \Bigr\rvert^{2}
        \leq
        \e^{- 2\Lambda t}
        \expect  \Bigl\lvert \xl_0 - \expect \xl_0 \Bigr\rvert^{2},
        \qquad \Lambda :=  1 -  \frac{d\sigma^2}{\int_{\R^d} \e^{-\alpha \bigl(f(x) - \underline f\bigr)} \, \mfldis_0(\d x)}  \,.
    \end{equation}
    A similar analysis is conducted in \cite{fornasier2024consensus}.
    For simplicity, we refrain from using refined estimates such as~\eqref{eq:refined_moment} in this work,
    but investigating the extent to which these estimates can be exploited would be a worthwhile direction for future work.
\end{remark}

\subsubsection{Uniform-in-time raw moment bounds: interacting particle system}
\label{sec:uit_raw_interacting}

\begin{lemma}
    [Uniform-in-time bounds for the raw moments]
    \label{lem:uit-raw-moments-2norm}
    Let $p\ge 2$ and assume that~$\rho^J_0 \in \mathcal P_{\rm sym}(\real^{dJ})$.
    Then for all~$J \ge 1$ it holds that
\begin{align}
    \expect \biggl[ \sup_{t \geq 0} \Bigl\lvert \xn{j}_t \Bigr\rvert^p \biggr]^{\frac{1}{p}}
    \leq
    \craw{p}(\sigma, \tau(S), \alpha, \underline f, \overline f) \, \expect \biggl[ \Bigl\lvert \xn{j}_0 \Bigr\rvert^p \biggr]^{\frac{1}{p}} \,,
\end{align}
where $\bra[\big]{\xn{j}_t}_{j \in \range{1}{J}}$ solves~\eqref{eq:cbo} and
\[
    \craw{p}(\sigma, \tau(S), \alpha, \overline f, \underline f)
    = 1 + \frac{p}{\edecay{p}} \Bigl(1 + \sigma \sqrt{\tau(S)} \cbdg{p}^{\frac{1}{p}}\Bigr) \left( 1 + \e^{\frac{\alpha}{p} (\overline f - \underline f)} \right) \,.
\]
\end{lemma}
\begin{proof}
    Rewriting~\eqref{eq:cbo} in integral form and using the triangle inequality,
    we obtain 
    \begin{align}
        \Bigl\lvert \xn{j}_t - \xn{j}_0 \Bigr\rvert
        & \leq \biggl\lvert  \int_{0}^{t} \Bigl( \xn{j}_s - \wmx{s} \Bigr) \, \d s \biggr\rvert
        + \sigma\left\lvert  \int_0^t S\bra[\Big]{ \xn{j}_s - \wmx{s}} \, \d \wn{j}_s \right\lvert \,.
    \end{align}
    Fix $T > 0$. Taking the supremum over $t \in [0, T]$, 
    then taking the $L^p(\Omega)$ norm and using the triangle inequality,
    we have
    \begin{align}\label{eq:int_ineq_1}
        \expect \biggl[ \sup_{t \in [0, T]} \Bigl\lvert \xn{j}_t - \xn{j}_0 \Bigr\rvert^p \biggr]^{\frac{1}{p}}
        &\leq \expect \left[ \sup_{t \in [0, T]} \left\lvert \int_{0}^{t}  \Bigl(\xn{j}_s - \wmx{t}\Bigr) \, \d s \right\rvert^p \right]^{\frac{1}{p}} \\
        & \qquad + \sigma \expect \left[ \sup_{t \in [0, T]} \abs*{\int_{0}^{t} S \bra[\Big]{ \xn{j}_s - \wmx{s}} \, \d \wn{j}_s }^p \,  \right]^{\frac{1}{p}}.
    \end{align}
    Thus, by using the Burkholder--Davis--Gundy inequality (\cref{thm:BDG_ineq}),
    we can bound the last term above as
    \begin{align}
        \cbdg{p}^{1/p}\sigma \expect \left[   \biggl(\int_{0}^{T} \Bigl\lVert S \bra*{\xn{j}_t - \wmx{t}} \Bigr\rVert_{\rm F}^2 \, \d t\biggr)^{\frac{p}{2}} \right]^{\frac{1}{p}} \,.
    \end{align}
    By H\"older's inequality, it holds for any function $h\colon \real \to \real$ and any $r \geq 1$ and $\ell > 0$ that
    \begin{align}
        \left\lvert \int_{0}^{T}  h(t)  \, \d t \right\rvert^r
        &= \left\lvert \int_{0}^{T}  \e^{- \frac{r-1}{r} \ell t} \cdot \e^{\frac{r-1}{r} \ell t} h(t)  \, \d t \right\rvert^r \\
        &\leq  \left( \int_{0}^{T}  \e^{- \ell t} \, \d t \right)^{r-1}
        \int_{0}^{T} \e^{(r-1)\ell t} \bigl\lvert h(t) \bigr\rvert^r  \, \d t
        \leq \frac{1}{\ell^{r-1}}
        \int_{0}^{T} \e^{(r-1)\ell t} \bigl\lvert h(t) \bigr\rvert^r  \, \d t.
        \label{eq:holder_trick}
    \end{align}
    Fixing $\ell = \frac{\edecay{p}}{p} \leq 1$ with $\edecay{p}$ as defined in \eqref{eq:lemma_moment_decay},
    we apply \eqref{eq:holder_trick} to both integrals on the right-hand side of \eqref{eq:int_ineq_1},
    with $r = p$ and $r = \frac{p}{2}$ respectively.
    Then, using that
    \[
        \tr \pra*{ S\bra*{\xn{j}_t - \wmx{t}}^2 }
        = \tau(S) \abs*{\xn{j}_t - \wmx{t} }^2\,,
    \]
    we obtain
    \begin{equation}
        \label{eq:bound_raw_moment}
        \expect \biggl[ \sup_{t \in [0, T]} \Bigl\lvert \xn{j}_t - \xn{j}_0 \Bigr\rvert^p \biggr]^{\frac{1}{p}}
        \leq \left( \frac{p}{\edecay{p}} \right)^{\frac{p-1}{p}}  \Bigl(1 + \sigma \sqrt{\tau(S)} \cbdg{p}^{\frac{1}{p}}\Bigr)
        \left( \int_{0}^{T} \e^{\frac{p-1}{p} \edecay{p} t} \expect \left[ \left\lvert \xn{j}_t - \wmx{t}\right\rvert^p \right]  \, \d t \right)^{\frac{1}{p}}.
    \end{equation}
    Now note that, by~\cref{lem:decay-centered-moments} and~\cref{lem:m-malpha}, we have
    \begin{align}
        \expect \left[ \left\lvert \xn{j}_t - \wmx{t}\right\rvert^p \right]^{\frac{1}{p}}
        &\leq \expect \left[ \left\lvert \xn{j}_t - \nmx{t}\right\rvert^p \right]^{\frac{1}{p}}
        + \expect \left[ \left\lvert \nmx{t} - \wmx{t}\right\rvert^p \right]^{\frac{1}{p}} \\
        &\leq \left( 1 + \e^{\frac{\alpha}{p} (\overline f - \underline f)} \right) \Bigl(  \expect \left[ \empmomentp{p}{0} \right]  \, \e^{-\edecay{p} t}  \Bigr)^{\frac{1}{p}},
    \end{align}
    where we used exchangeability of the initial law~$\rho^J_0 \in \mathcal P_{\rm sym}(\real^{dJ})$,
    so that
    \[
        \expect \biggl[ \left\lvert \xn{j}_{0} - \nmx{0}\right\rvert^p \biggr]
        = \expect \biggl[ \frac{1}{J} \sum_{k=1}^{J} \left\lvert \xn{k}_{0} - \nmx{0}\right\rvert^p \biggr]
        = \expect \Bigl[ \empmomentp{p}{0} \Bigr].
    \]
    Substituting this bound in~\eqref{eq:bound_raw_moment} leads to
    \[
        \expect \biggl[ \sup_{t \in [0, T]} \Bigl\lvert \xn{j}_t - \xn{j}_0 \Bigr\rvert^p \biggr]^{\frac{1}{p}}
        \leq \left( \frac{p}{\edecay{p}} \right)^{\frac{p-1}{p}}  \Bigl(1 + \sigma \sqrt{\tau(S)}  \cbdg{p}^{\frac{1}{p}} \Bigr) \left( 1 + \e^{\frac{\alpha}{p} (\overline f - \underline f)} \right) \expect \left[ \empmomentp{p}{0} \right]^{\frac{1}{p}} \left( \frac{p}{\edecay{p}} \right)^{\frac{1}{p}}.
    \]
    Since~$T$ was arbitrary,
    it follows from the monotone convergence theorem that
    \[
        \expect \biggl[ \sup_{t \geq 0} \Bigl\lvert \xn{j}_t \Bigr\rvert^p \biggr]^{\frac{1}{p}}
        \leq
        \expect \biggl[ \Bigl\lvert \xn{j}_0 \Bigr\rvert^p \biggr]^{\frac{1}{p}}
        +
        \frac{p}{\edecay{p}} \Bigl(1 + \sigma \sqrt{\tau(S)} \cbdg{p}^{\frac{1}{p}} \Bigr) \left( 1 + \e^{\frac{\alpha}{p} (\overline f - \underline f)} \right) \expect \left[ \empmomentp{p}{0} \right]^{\frac{1}{p}} \,.
    \]
    Recall that from \eqref{eq:raw-centered-moment-bound}, the centered moments are bounded in terms of the raw moments by $\momentp{p}{\mu}
            \leq 2^p \rawmomentp{p}{\mu}$, and so the conclusion follows.
\end{proof}

\subsubsection{Uniform-in-time raw moment bounds: mean-field process}
\label{sec:uit_raw_mean_field}
\begin{lemma}
    [Uniform-in-time mean-field raw moment bound]
    \label{lem:uit-raw-moments-synchronous-system}
    Assume that~$\mfldis_0 \in \mathcal P(\real^{d})$.
    Then it holds for all~$p \geq 2$ that
    \begin{align}
        \expect \biggl[ \sup_{t \geq 0} \bigl\lvert \xl_t \bigr\rvert^p \biggr]^{\frac{1}{p}}
        \leq \craw{p}(\sigma, \tau(S), \alpha, \overline f, \underline f)  \expect \biggl[ \bigl\lvert \xl_0 \bigr\rvert^p \biggr]^{\frac{1}{p}} \,,
    \end{align}
    where~$\craw{p}$ is the constant from~\cref{lem:uit-raw-moments-2norm}.
\end{lemma}
\begin{proof}
    \label{remark:alternative_proof}
    We prove the statement by combining the finite-time mean-field limit result from~\cite[Theorem 2.6]{gerber2023meanfield}
    with the raw moment bounds for the interacting particle system given in~\cref{lem:uit-raw-moments-2norm}.
    To be more precise,
    for~$J \in \nat_{>0}$,
    we consider a synchronous coupling between the interacting particle system~\eqref{eq:cbo} of size~$J$
    and the same number of copies of the mean-field system.
    By the triangle inequality, it holds that
    \begin{align}
        \expect
        \biggl[ \sup_{t \in [0, T]} \Bigl\lvert \xnl{1}_t \Bigr\rvert^p \biggr]^{\frac{1}{p}}
        & \leq \expect \biggl[
        \sup_{t \in [0, T]}
        \Bigl\lvert \xnn{1}{J}_t \Bigr\rvert^p \biggr]^{\frac{1}{p}} + \expect \biggl[
        \sup_{t \in [0, T]}
        \Bigl\lvert \xnl{1}_t  - \xnn{1}{J}_t \Bigr\rvert^p \biggr]^{\frac{1}{p}},
    \end{align}
    where we write~$\xnn{1}{J}_t$ instead of our usual notation~$\xn{1}_t$ to emphasize the size of the system.
    Taking the limit~$J \to \infty$, we deduce from the finite-time mean-field limit theorem \cite[Theorem 2.6]{gerber2023meanfield}
    and~\cref{lem:uit-raw-moments-2norm} that
    \begin{align}
        \expect
        \biggl[ \sup_{t \in [0, T]} \Bigl\lvert \xnl{1}_t \Bigr\rvert^p \biggr]^{\frac{1}{p}}
        \leq \lim_{J \to \infty} \expect \biggl[ \sup_{t \in [0, T]} \Bigl\lvert \xnn{1}{J}_t \Bigr\rvert^p \biggr]^{\frac{1}{p}}
        \leq \craw{p}(\sigma, 	\tau(S), \alpha, \underline f, \overline f)  \expect \biggl[ \Bigl\lvert \xnl{1}_0 \Bigr\rvert^p \biggr]^{\frac{1}{p}}.
    \end{align}
    Since this holds for all~$T > 0$,
    the result follows by taking~$T \to +\infty$ and using the monotone convergence theorem.
\end{proof}

\subsection{Concentration inequalities}
\label{sub:concent}

The following simple observation,
which is based on the Burkholder--Davis--Gundy inequality,
turns out to be quite powerful since it enables to show concentration bounds for the microscopic CBO interacting particle system \eqref{eq:cbo}.
\begin{lemma}
    \label{lem:concentration-ineq}
    Fix $q\ge2$ and $J\in\N$.
    Let $\wn{1}_t, \dots, \wn{J}_t$ be independent Brownian motions in $\R^d$ and $\bra*{\mathcal F_t}_{t\ge 0}$ be the filtration generated by them.
    Let $\bigl(\sigma_j(t)\bigr)_{t\ge 0}$ for $j=1,\dots, J$ be $\R^d$-valued $\mathcal F_t$-adapted stochastic processes such that
    the function~$s \to \expect\pra[\big]{\abs*{ \sigma_j(s)}^q}$ belongs to $L^1(0, T)$.
    Consider the~$\R$-valued martingale
    \begin{align}
        M_t := \frac{1}{J} \sum_{j=1}^{J} \int_{0}^{t} \scp*{\sigma_j(s) , \d \wn{j}_s}.
    \end{align}
    Then, it holds for any $\ell > 0$ and $t\le T$ that
    \begin{align}
        \label{eq:concentration:sup-martingale}
        \expect \pra*{ \sup_{s\in[0,t]} \abs*{M_s}^q} \le \frac{\cbdg{q}}{J^{\frac{q}{2}}} \frac{1}{J}\sum_{j=1}^J \frac{1}{\ell^{\frac{q}{2}-1}}
        \int_{0}^{t} \e^{(\frac{q}{2}-1)\ell s} \expect\pra[\Big]{\abs*{ \sigma_j(s)}^q} \, \d s.
    \end{align}
    Furthermore,
    if $\bra*{Y_t}_{t\ge0}$ is a $\R$-valued stochastic process such that $Y_t \le Y_0 + M_t$ for all $0\le t\le T$, then
    \begin{align}
        \label{eq:concentration:sup-Y-bound}
        \forall A > 0, \qquad
       & \proba \left[ \sup_{s\in[0, t]} Y_s  \ge \expect Y_0 + A \right]
       \le \frac{2^q}{A^q} \expect \Bigl[ \bigl\lvert Y_0 - \expect Y_0 \bigr\rvert^q \Bigr]
       + \frac{2^q}{A^q}   \expect \pra*{ \sup_{s\in[0,t]} \abs*{M_s}^q}.
    \end{align}
\end{lemma}

\begin{proof}
    For $s\ge0$, let
    \[
        g(s) \coloneq \frac{1}{J}\begin{pmatrix} \sigma_1(s)^T, \dots, \sigma_J(s)^T \end{pmatrix} \in \R^{1\times (dJ)}
        \quad \text{ and } \quad W_s:=\begin{pmatrix} \wn{1}_s \\ \vdots \\ \wn{J}_s \end{pmatrix} \in \R^{dJ}.
    \]
    Applying the Burkholder--Davis--Gundy inequality~\cref{thm:BDG_ineq} to $g$, we have
    \begin{align}
        \expect \pra*{ \sup_{s\in[0,t]} \abs*{M_s}^q} \le \cbdg{q} \expect \pra*{ \scp*{M}_t^{\frac{q}{2}}} && \text{where} &&\scp*{M}_t = \int_0^t \norm*{g(s)}_F^2 \, \d s = \frac{1}{J^2}\sum_{j=1}^J \int_0^t \abs*{ \sigma_j(s)}^2 \, \d s.
    \end{align}
    Using \eqref{eq:holder_trick} with $r = \frac{q}{2}$, we obtain
    \begin{align}
        \scp*{M}_t^{\frac{q}{2}} &= \bra*{\frac{1}{J^2}\sum_{j=1}^J \int_0^t \abs*{ \sigma_j(s)}^2 \, \d s}^{\frac{q}{2}}
        \le \frac{1}{J^{\frac{q}{2}}} \frac{1}{J}\sum_{j=1}^J \bra*{\int_0^t \abs*{ \sigma_j(s)}^2 \, \d s}^{\frac{q}{2}} \\
        &\le \frac{1}{J^{\frac{q}{2}}} \frac{1}{J}\sum_{j=1}^J \frac{1}{\ell^{\frac{q}{2}-1}}
        \int_{0}^{t} \e^{(\frac{q}{2}-1)\ell s} \abs*{ \sigma_j(s)}^q \, \d s.
    \end{align}
    The second claim follows from
    \begin{align}
        & \proba \left[ \sup_{s\in[0, t]} Y_s  \ge \expect Y_0 + A \right]
         \le \proba \left[ Y_0 - \expect Y_0 \geq  \frac{A}{2}\right]
         + \proba \left[ \sup_{s \in [0, t]}  Y_s - Y_0  \geq \frac{A}{2} \right]
     \end{align}
     and Markov's inequality.
\end{proof}

\subsubsection{Concentration inequality: interacting particle system}
\begin{lemma}[Bound on probability of large excursions]
    \label{prop:bound-on-bad-set-general-noise}
    Assume that $f$ satisfies \cref{assumption:bounded} and let $q\ge 2$. Consider the CBO dynamics~\eqref{eq:cbo}
    where $\bra[\big]{\xn{j}_0}_{j \in \range{1}{J}}$ are sampled i.i.d.\ from some $\mfldis_0 \in \mathcal P_{2q}(\real^d)$.
    Then, for any $\kappa < \min\bigl\{\edecay{2}, \frac{\edecay{2q}}{q}\bigr\}$,
    there exists a finite constant $\icbad{q}{\kappa} $ such that for all $A > 0$, the following holds for all $J \in \N_{+}$:
    \begin{align}
        \proba \left[ \sup_{t \geq 0} \e^{\kappa t} \empmomentp{2}{t} \geq \expect \Bigl[ \empmomentp{2}{0} \Bigr] + A \right]
        \le \icbad{q}{\kappa} A^{-q} J^{- \frac{q}{2}}  \momentp{2q}{\mfldis_0} .
    \end{align}
    The constant $\icbad{q}{\kappa}$ is given by
    \[
        \icbad{q}{\kappa}
        =  2^{3q-1} \cmz{2q}
        + 2^{4 q + 1} \cbdg{q} \sigma{^{q}}
        \left(\frac{q-2}{\edecay{2q } - q \kappa} \right)^{\frac{q}{2} - 1}
        \frac{
            \bra*{1 + \e^{\alpha  (\overline f - \underline f) }}^{\frac{1}{2}}
           }{\edecay{2q } - q \kappa},
    \]
    with the convention that $0^0=1$ if $q=2$.
\end{lemma}
\begin{remark}\label{rmk:bound-on-bad-set-general-noise-coarse}
    Note that $\expect \Bigl[ \empmomentp{2}{0} \Bigr] \leq  \momentp{2}{\mfldis_0} $,
    so it also holds that
    \begin{align}
        \proba \left[ \sup_{t \geq 0} \e^{\kappa t} \empmomentp{2}{t} \geq \momentp{2}{\mfldis_0} + A \right]
        \le \icbad{q}{\kappa} A^{-q} J^{- \frac{q}{2}}  \momentp{2q}{\mfldis_0}.
    \end{align}
    This form of the estimate is convenient as it is similar to that of~\cref{prop:bound-on-bad-set-synchronous-mf}.
\end{remark}
\begin{proof}
    We proved in \cref{lem:decay-centered-moments} that
    \begin{align}
        \d \empmomentp{2}{t}
         & \le
         - \edecay{2}
         \empmomentp{2}{t}
         \, \d t + \frac{2\sigma}{J} \sum_{j = 1}^{J}
         \scp*{ \xn{j}_t - \meanx{t}, S\bra*{ \xn{j}_t - \wmx{t}} \d \wn{j}_t} \\
         & \qquad - \frac{2\sigma}{J^2} \sum_{j = 1}^{J} \sum_{k=1}^{J}
         \scp*{ \xn{j}_t - \meanx{t}, S\bra*{ \xn{k}_t - \wmx{t}} \d \wn{k}_t}.
    \end{align}
    Observe that the second noise term vanishes.
    Define $Y_t := \e^{\kappa t}  \empmomentp{2}{t}$.
    Since $\kappa \leq \edecay{2}$,
    we have by It\^o's formula that~$Y_t \le Y_0 + M_t$, where $M_t$ is defined as in \cref{lem:concentration-ineq} with $\sigma_j$ given by
    \begin{align}
        \sigma_j(s) &:=
        2\sigma \e^{\kappa s}
            S\bra[\Big]{ \xn{j}_s - \wmx{s}}\bra[\Big]{\xn{j}_s - \meanx{s}}.
    \end{align}
    Therefore, we obtain for both $S\in \{S^{(i)}, S^{(a)}\}$ that
    \begin{align}
        \abs*{\sigma_j(s)}^q
        & \le 2^q \sigma^q \e^{q\kappa s}  \abs[\Big]{ \xn{j}_s - \wmx{s}}^q\cdot\abs[\Big]{\xn{j}_s - \meanx{s}}^q.
    \end{align}
    From the inequality $\abs{x+y}^{2q} \le 2^{2q-1} \abs{x}^{2q} + 2^{2q-1} \abs{y}^{2q}$ for all $x,y\in\R^d$ and \cref{lem:m-malpha}
    we have
    \begin{align}
        \expect\pra[\Big]{  \left\lvert \xn{j}_{s} - \wmx{s}\right\rvert^{2q}}
        \le
        2^{2q-1} \left( 1 + \e^{\alpha  (\overline f - \underline f) } \right)  \expect \pra[\Big]{\empmomentp{2q}{s}}.
    \end{align}
    Therefore, we obtain from Hölder's inequality and \cref{lem:decay-centered-moments} that
    \begin{align}
        \expect \pra[\Big]{\abs[\big]{\sigma_j(s)^q}}
         &\le
         2^{q}\sigma^q \e^{q\kappa s}
         \bra*{\expect\pra[\Big]{  \left\lvert \xn{j}_{s} - \wmx{s}\right\rvert^{2q}}\cdot\expect\pra[\Big]{  \left\lvert \xn{j}_{s} - \meanx{s}\right\rvert^{2q}}}^{\frac{1}{2}} \\
         & \le 2^{2q} \sigma^q \e^{q\kappa s}  \bra*{1 + \e^{\alpha  (\overline f - \underline f) }}^{\frac{1}{2}} \expect \pra[\Big]{\empmomentp{2q}{s}} \\
         & \le 2^{2q} \sigma^q \e^{(q\kappa  - \edecay{2q})s}  \bra*{1 + \e^{\alpha  (\overline f - \underline f) }}^{\frac{1}{2}} \expect \pra[\Big]{\empmomentp{2q}{0}}.
    \end{align}
    Hence, $\expect\pra[\big]{\abs*{ \sigma_j(s)}^q} \in L^1([0, \infty))$ since $q\kappa < \edecay{2q}$ by assumption,
    which allows to apply \cref{lem:concentration-ineq} to obtain
    \begin{align}
        \expect \left[  \sup_{s \in [0,t]} \abs{M_s}^q  \right]
        \leq
        \frac{2^{2q}\cbdg{q} \sigma^q  \bra*{1 + \e^{\alpha  (\overline f - \underline f) }}^{\frac{1}{2}}}{J^{\frac{q}{2}}\ell^{\frac{q}{2} - 1}}
          \expect \Bigl[ \empmomentp{2q}{0} \Bigr]
        \int_{0}^{t} \e^{(\frac{q}{2} - 1) \ell s + (q \kappa  - \edecay{2q}) s} \d s\,.
    \end{align}
    Note that $q \kappa < \edecay{2q } $ ensures that the exponential in the integral is decreasing if $q = 2$.
    For $q > 2$, we fix $\ell = \frac{\edecay{2q} - q \kappa}{q - 2}$ so that~$(\frac{q}{2} - 1) \ell = \frac{1}{2}(\edecay{2q} - q \kappa)$,
    and the exponential is decreasing again.
    For all $q \geq 2$ cases, it holds that
    \begin{align}
        \expect \left[  \sup_{s \in [0,t]} \abs{M_s}^q  \right]
        \leq
         2^{2q+1}\cbdg{q} \sigma^q
            \left(\frac{q-2}{\edecay{2q} - q \kappa} \right)^{\frac{q}{2} - 1}
            \frac{
                \bra*{1+\e^{\alpha  (\overline f - \underline f)}}^{\frac{1}{2}}}{\edecay{2q} - q \kappa}
        J^{- \frac{q}{2}}
        \expect \Bigl[ \empmomentp{2q}{0} \Bigr],
    \end{align}
    with the convention that $0^0 = 1$ for $q = 2$.
    Using 
    Huygens' identity \eqref{eq:huygens}, 
    together with the Marcinkiewicz--Zygmund inequality,
    Jensen's inequality and
    the elementary inequality $\expect |Z - \expect Z|^q \leq 2^q \expect |Z|^q$ for any real-valued random variable~$Z$ with finite first moment,
    we deduce
    \begin{align}
        \expect \Bigl[ \bigl\lvert Y_0 - \expect Y_0 \bigr\rvert^q \Bigr]
            &= \expect  \left( \left\lvert \empmomentp{2}{0} - \expect \empmomentp{2}{0} \right\rvert \right)^q \\
            &= \expect
            \left\lvert
                \frac{1}{J} \sum_{j=1}^{J} \left\lvert \xn{j}_0 - \nm (\mfldis_0) \right\rvert^2
                - \left\lvert \nm(\emp{0}) - \nm(\mfldis_0) \right\rvert^2
                -  \momentp{2}{\mfldis_0}
                + \expect \left[ \left\lvert \nm(\emp{0}) - \nm(\mfldis_0) \right\rvert^2\right]
            \right\rvert^q
            \\
            &\leq 2^{q-1}
            \expect
            \left\lvert
                \frac{1}{J} \sum_{j=1}^{J} \left\lvert \xn{j}_0 - \nm (\mfldis_0) \right\rvert^2 - \momentp{2}{\mfldis_0}
            \right \rvert^q \\
            &\qquad +
            2^{q-1}
            \expect
            \left\lvert
                \left\lvert \nm(\emp{0}) - \nm(\mfldis_0) \right\rvert^2
                -
                \expect \left[ \left\lvert \nm(\emp{0}) - \nm(\mfldis_0) \right\rvert^2 \right]
            \right\rvert^q
            \\
            &\leq
            2^{q-1} \cmz{q} J^{- \frac{q}{2}}
            \expect
            \left\lvert
                \left\lvert \xn{1}_0 - \nm (\mfldis_0) \right\rvert^2 - \momentp{2}{\mfldis_0}
            \right\rvert^q
            +
            2^{2q-1}
            \expect  \left\lvert \nm(\emp{0}) - \nm(\mfldis_0) \right\rvert^{2q}
            \\
            &\leq
            2^{2q-1} \cmz{q} J^{- \frac{q}{2}}
            \expect \left\lvert \xn{1}_0 - \nm (\mfldis_0) \right\rvert^{2q}
            + 2^{2q-1}  \cmz{2q} J^{- \frac{q}{2}}   \expect \left\lvert \xn{1}_0 - \nm (\mfldis_0) \right\rvert^{2q} \\
            &= 2^{2q} \cmz{2q} J^{- \frac{q}{2}} \momentp{2q}{\mfldis_0} ,
            \label{eq:proof:concentration:y-Marcinkiewicz}
    \end{align}
    where we used that $\cmz{2q} \geq \cmz{q}$.
    Thus, equation \eqref{eq:concentration:sup-Y-bound} and the inequality $\empmomentp{2q}{0} \leq 2^q \momentp{2q}{\mfldis_0}$ imply that
    \begin{align}
        \proba \left[ \sup_{s \in[0,t]} \e^{\kappa s} \empmomentp{2}{s} \geq \expect \Bigl[ \empmomentp{2}{0} \Bigr] + A \right]
        \le \cbad{q}{\kappa} J^{- \frac{q}{2}}  \momentp{2q}{\mfldis_0} .
    \end{align}
    Note that the right-hand side of this inequality is independent of~$t$,
    so the same inequality holds when the supremum on the left-hand side is taken over $[0, \infty)$ by monotone convergence. This implies the claim.
\end{proof}

\begin{remark}
    Similar statements with $\empmomentp{2}{t}$ replaced by $\empmomentp{p}{t}$ can be obtained in the same way,
    but they are not required for our purposes in this paper.
\end{remark}

\subsubsection{Concentration inequality: synchronously coupled mean-field system}
\begin{lemma}[Bound on probability of large excursions for the synchronously coupled system]
    \label{prop:bound-on-bad-set-synchronous-mf}
    Fix $q \geq 2$ and assume that $f$ satisfies \cref{assumption:bounded}.
    Consider the system \eqref{eq:thm-mfl:synchronously-coupled-system} where $\bra[\big]{\xn{j}_0}_{j}$ are sampled i.i.d.\ from $\mfldis_0^{\otimes J}$, with $\mfldis_0 \in \mathcal P_{2q}(\real^d)$.
    Then for all $\kappa < \min\bigl\{\edecay{2}, \frac{\edecay{2q}}{q}\bigr\}$
    and for all $A > 0$,
    the following holds for all $J \in \N_{+}$:
    \begin{align}
        \proba \left[ \sup_{t \geq 0} \e^{\kappa t} \emplmomentp{2}{t} \geq \momentp{2}{\mfldis_0} + A \right]
        & \le \cbad{q}{\kappa} A^{-q} J^{-\frac{q}{2}} \Bigl [ \momentp{2q}{\mfldis_0} \Bigr] ,
    \end{align}
    The constant $\cbad{q}{\kappa}$ is given by
    \[
        \cbad{q}{\kappa}
        =  \frac{3^q}{2^q} \icbad{q}{\kappa}
        +
        3^q \cwm{2q} 2^{q+1} \sigma^{2q} \tau(S)^q
        \left( \frac{2(q-1)}{\edecay{2q} - q \kappa} \right)^{q-1}
        \left( 1 + \frac{2 \sigma^2 \tau(S)}{\edecay{2} - \kappa}  \left(1 + \e^{\frac{\alpha}{2} (\overline f - \underline f)} \right)^{2} \right)^q
        \frac{1}{\edecay{2q} - \kappa}\,,
    \]
    where $\icbad{q}{\kappa}$ is as in \cref{prop:bound-on-bad-set-general-noise}.
\end{lemma}

\begin{proof}
    Recall from the proof of~\cref{thm:mfl-decay-p-p-cetered-moment} that
    \begin{align}
        \forall j \in \range{1}{J}, \qquad
        \d \abs[\Big]{\xnl{j}_t - \expect {\xnl{j}_t}}^2
        & \le -2 \abs[\Big]{\xnl{j}_t - \expect {\xnl{j}_t}}^2 \d t
        + \sigma^2 \tau(S)
        \Bigl\lvert \xnl{j}_t - \wm(\mfldis_t) \Bigr\rvert^2 \, \d t \\
           & \qquad +
           2 \sigma
        \Bigl\langle
            \xnl{j}_t - \expect \xnl{j}_t,
            S\bra[\Big]{\xnl{j}_t - \wm(\mfldis_t)} \, \d \wn{j}_t
        \Bigr\rangle.
    \end{align}
    By the triangle inequality and \cref{lem:m-malpha},
    together with the inequality~$\emplmomentp{2}{t} \leq \frac{1}{J} \sum_{j=1}^{J} \left\lvert \xnl{j}_t - \nm(\mfldis_t) \right\rvert^2$,
    we have
    \begin{align}
        \left( \frac{1}{J} \sum_{j=1}^{J} \abs*{\xnl{j}_t - \wm(\mfldis_t)}^2 \right)^{\frac{1}{2}}
    & \le 
    \left( \frac{1}{J} \sum_{j=1}^{J} \abs*{\xnl{j}_t - \nm(\mfldis_t)}^2 \right)^{\frac{1}{2}}
    + \abs*{\nm(\mfldis_t) - \nm(\empl{t})} 
 \\ &\qquad
 + \abs*{\nm(\empl{t}) - \wm(\empl{t})}
 + \abs*{\wm(\empl{t}) - \wm(\mfldis_t)}
 \\
    & \le 
    \left(1 + \e^{\frac{\alpha}{2} (\overline f - \underline f)} \right) \left( \frac{1}{J} \sum_{j=1}^{J} \abs*{\xnl{j}_t - \nm(\mfldis_t)}^2 \right)^{\frac{1}{2}}
 \\ &\qquad 
 + \abs*{\nm(\mfldis_t) - \nm(\empl{t})}  + \abs*{\wm(\empl{t}) - \wm(\mfldis_t)}.
    \end{align}
    Therefore, since $(a + b)^2 \leq (1 + \varepsilon) a^2 + (1 + \frac{1}{\varepsilon}) b^2$ for all $\varepsilon > 0$,
    it follows that
    \begin{align}
        \frac{1}{J} \sum_{j=1}^{J} \abs*{\xnl{j}_t - \wm(\mfldis_t)}^2
        &\leq (1 + \varepsilon) \left(1 + \e^{\frac{\alpha}{2} (\overline f - \underline f)} \right)^2 \left( \frac{1}{J} \sum_{j=1}^{J} \abs*{\xnl{j}_t - \nm(\mfldis_t)}^2 \right) \\
        &\qquad
        + \left( 1 + \frac{1}{\varepsilon} \right)
        \left(
            \abs*{\nm(\mfldis_t) - \nm(\empl{t})}  + \abs*{\wm(\empl{t}) - \wm(\mfldis_t)}
        \right)^2.
    \end{align}
    We take $\varepsilon = \frac{1}{\sigma^2 \tau(S)} (\edecay{2} - \kappa) \left(1 + \e^{\frac{\alpha}{2} (\overline f - \underline f)} \right)^{-2}$,
    so that
    we have for  $Y_t = \frac{1}{J} \sum_{j=1}^{J} \e^{\kappa t} \left\lvert \xnl{j}_t - \nm(\mfldis_t) \right\rvert^2$ that
    \begin{align}
        \d Y_t
    & \le \sigma^2 \tau(S) \e^{\kappa t} \left( 1 + \frac{1}{\varepsilon} \right)
    \left(
        \abs*{\nm(\mfldis_t) - \nm(\empl{t})}  
        + \abs*{\wm(\empl{t}) - \wm(\mfldis_t)}
    \right)^2 \, \d t
    \\
    & \qquad + \frac{2\sigma}{J}\e^{\kappa t} \sum_{j=1}^{J} \Bigl\langle \xnl{j}_t - \nm(\mfldis_t), S\Bigl(\xnl{j}_t - \wm(\mfldis_t)\Bigr) \, \d \wn{j}_t \Bigr\rangle\,.
    \end{align}
    Since $\kappa < \edecay{2}$,
    the first term is negative, and so
    \begin{align}
        \label{eq:proof-mfl-concentration-Y-bound}
        Y_t &\le Y_0 + \sigma^2 \tau(S)
        \left( 1 + \frac{1}{\varepsilon} \right)
        \int_{0}^{t} 
        \e^{\kappa s}
        \left(
            \abs*{\nm(\mfldis_s) - \nm(\empl{s})}  
            + \abs*{\wm(\empl{s}) - \wm(\mfldis_s)}
        \right)^2 \, \d s + M_t \\
            &\le  Y_0 + Z_t + M_t\,, 
    \end{align}
    where we introduced
    \[
        Z_t := \sigma^2 \tau(S)
        \left( 1 + \frac{1}{\varepsilon} \right)
        \int_{0}^{t} 
        \e^{\kappa s}
        \left(
            \abs*{\nm(\mfldis_s) - \nm(\empl{s})}  
            + \abs*{\wm(\empl{s}) - \wm(\mfldis_s)}
        \right)^2 \, \d s    
    \]
    and $M_t$ is defined as in \cref{lem:concentration-ineq} with $\sigma_j$ given by
    \begin{align}
        \sigma_j(s) &:= 2\sigma \e^{\kappa s} S\bra[\Big]{ \xnl{j}_s - \wm(\mfldis_s)}\bra[\Big]{\xnl{j}_s - \nm(\mfldis_s)}.
    \end{align}
    As in the proof of~\cref{prop:bound-on-bad-set-general-noise},
    we obtain for both $S\in \{S^{(i)}, S^{(a)}\}$ that
    \begin{align}
        \abs*{\sigma_j(s)}^q
        & \le 2^q \sigma^q \e^{q\kappa s}  \abs[\Big]{\xnl{j}_s - \wm(\mfldis_s)}^q\cdot\abs[\Big]{\xnl{j}_s - \nm(\mfldis_s)}^q.
    \end{align}
    From \cref{lem:m-malpha} we have that
    \begin{align}
    \expect\pra[\Big]{\left\lvert \xnl{j}_{s} - \wm(\mfldis_s)\right\rvert^{2q}}
    \le
    2^{2q-1 }
    \left( 1 + \e^{ \alpha (\overline f - \underline f) } \right)
\expect\pra[\Big]{\left\lvert \xnl{j}_{s} - \nm(\mfldis_s)\right\rvert^{2q}}.
\end{align}
Therefore,
using~\cref{thm:mfl-decay-p-p-cetered-moment} and Hölder's inequality,
we obtain that
\begin{align}
\expect \pra[\Big]{\abs[\big]{\sigma_j(s)}^q}
         &\le 2^{2q} \sigma^q \e^{(q\kappa - \edecay{2q}) s} \left( 1 + \e^{ \alpha (\overline f - \underline f) }\right)^{\frac{1}{2}}
     \expect\pra[\Big]{\left\lvert \xnl{j}_{0} - \nm(\mfldis_0)\right\rvert^{2q}}\,.
\end{align}
Once again,
it holds that $\expect\pra[\big]{\abs*{ \sigma_j(s)}^q} \in L^1\bigl([0, \infty)\bigr)$ since $\kappa < \frac{\edecay{2q}}{q}$,
so we can apply \cref{lem:concentration-ineq} to obtain
\begin{align}
    \expect \pra*{ \sup_{s\in[0,t]} \abs*{M_s}^q}
    \le \frac{2^{2q}\cbdg{q}
        \sigma^q
         \left( 1 + \e^{ \alpha (\overline f - \underline f) } \right)^{\frac{1}{2}}
        }{J^{\frac{q}{2}}\ell^{\frac{q}{2}-1}}
        \expect\pra[\Big]{\left\lvert \xl_{0} - \nm(\mfldis_0)\right\rvert^{2q}}
        \int_{0}^{t} \e^{(\frac{q \ell}{2}- \ell + q\kappa - \edecay{2q}) s}  \d s\,.
    \end{align}
    As before we deduce that
    \begin{align}
        \expect \left[  \sup_{s \in [0,t]} \abs{M_s}^q  \right]
        \leq
         2^{2q+1}\cbdg{q} \sigma^q \left(\frac{q-2}{\edecay{2q} - q \kappa} \right)^{\frac{q}{2} - 1}
            \frac{\bra*{1+\e^{\alpha  (\overline f - \underline f)}}^{\frac{1}{2}}}{\edecay{2q} - q \kappa} J^{- \frac{q}{2}}
         \momentp{2q}{\mfldis_0} ,
    \end{align}
    with the convention that $0^0=1$ for $q=2$.
From \eqref{eq:proof-mfl-concentration-Y-bound}, 
we have that
\begin{align}
    \proba \left[ \sup_{s\in[0, t]} Y_s  \ge \expect Y_0 + A \right]
     &\le 
     \proba \left[ \sup_{s\in[0, t]} Y_0 + Z_s + M_s  \ge \expect Y_0 + A \right] \\
     &\leq
     \proba \left[ Y_0 - \expect Y_0 \geq  \frac{A}{3}\right]
     +
     \proba \left[ \sup_{s \in [0, t]} \abs{Z_s} \geq  \frac{A}{3}\right]
     +
     \proba \left[ \sup_{s \in [0, t]} \abs{M_s} \geq \frac{A}{3} \right] \\
     &\leq
     \frac{3^q}{A^q}\expect  \left\lvert Y_0 - \expect Y_0 \right\rvert^q
     +
     \frac{3^q}{A^q} \expect  \left[ \sup_{s \in [0, t]} \abs{Z_s}^q \right]
     +
     \frac{3^q}{A^q}   \expect \pra*{ \sup_{s\in[0,t]} \abs*{M_s}^q}.
\end{align}
The first and third terms can be bounded as previously.
For the second term, 
using~\eqref{eq:holder_trick} with parameter $\ell = \frac{\edecay{2q} - q \kappa}{2(q-1)}$ so that~$(q - 1) \ell = \frac{1}{2}(\edecay{2q} - q \kappa)$,
then using~\cref{lem:convergence_weighted_mean_iid-2} and~\cref{thm:mfl-decay-p-p-cetered-moment},
we have that
\begin{align}
    \expect  \left[ \sup_{s \in [0, t]} \abs{Z_s}^q \right]
    &\leq 
    \frac{2^{q-1}\sigma^{2q} \tau(S)^q}{\ell^{q-1}}
    \left( 1 + \frac{1}{\varepsilon} \right)^q
    \int_{0}^{+\infty} 
    \e^{(q-1)\ell s + \kappa s} \left( \expect  \abs*{\nm(\mfldis_s) - \nm(\empl{s})}^{2q} + \expect \abs*{\wm(\empl{s}) - \wm(\mfldis_s)}^{2q} \right)  \, \d s     \\
    &\leq 
    \cwm{2q} \frac{2^{q} \sigma^{2q} \tau(S)^q}{\ell^{q-1} J^{q}}
    \left( 1 + \frac{1}{\varepsilon} \right)^q
    \int_{0}^{+\infty} 
    \e^{(q-1)\ell s + \kappa s}  \momentp{2q}{\mfldis_s} \, \d s     \\
    &\leq 
    \cwm{2q} \frac{2^{q+1} \sigma^{2q} \tau(S)^q}{\ell^{q-1} J^{q}}
    \left( 1 + \frac{1}{\varepsilon} \right)^q
    \frac{1}{\edecay{2q} - \kappa}.
\end{align}
Using that
\(
Y_t \geq \e^{\kappa t} \emplmomentp{2}{t},
\)
we can then conclude in the same way as in the proof of~\cref{prop:bound-on-bad-set-general-noise}.
\end{proof}

\subsection{Stability estimate for the weighted mean}
\label{sec:proof:lem:stab_wmean}
\begin{lemma}[Local Lipschitz continuity of~$\mu \mapsto \wmmu-  \nm(\mu)$]
    \label{lem:stab_wmean}
    Suppose that~\cref{assumption:bounded,assumption:lip} are satisfied.
    Then it holds for all $\mu, \nu \in \mathcal P_2(\mathbb R^d)$ that
    \[
        \left\lvert \wmmu - \mathcal M(\mu) - \wmnu + \mathcal M(\nu) \right\rvert
        \leq        C_{\mathcal M}
        \left(  \sqrt{\momentp{2}{\mu}} + \sqrt{\momentp{2}{\nu}}  \right)
        \wasserstein_2 (\mu, \nu), \qquad
        C_{\mathcal M} := 2 \alpha L_f \e^{2 \alpha (\overline f - \underline f)}.
    \]
\end{lemma}
\begin{proof}[Proof of \cref{lem:stab_wmean}]
    Let $g(x) = \bigl(x - \mathcal M(\nu) \bigr) \, (\e^{- \alpha  f(x)} - Z_{\mu})$ and
    \[
        Z_{\mu} = \int \e^{- \alpha  f(x)} \, \mu(\d x), \qquad
        Z_{\nu} = \int \e^{- \alpha  f(x)} \, \nu(\d x).
    \]
    It holds for any coupling $\pi \in \Pi(\mu, \nu)$ that
    \begin{align} \label{eq:wasser_mean_splitting}
        \wmmu - \mathcal M(\mu) - \wmnu + \mathcal M(\nu)
         & = \int  \bigl(x - \mathcal M(\nu) \bigr) \, \left(\e^{- \alpha  f(x)} - Z_{\mu}\right)  \,
         \left( \frac{\mu(\d x)}{Z_{\mu}}  - \frac{\nu(\d x)}{Z_{\nu}} \right)                       \\
         & =
         \frac{1}{Z_{\mu}} \iint \bigl( g(x) - g(y) \bigr) \pi(\d x \, \d y)
         + \left( \frac{1}{Z_{\mu}} - \frac{1}{Z_{\nu}} \right) \int g(x) \, \nu(\d x).
    \end{align}
    By assumption, it holds that $Z_{\mu} \geq \e^{- \alpha \overline f}$ and $Z_{\nu} \geq \e^{- \alpha \overline f}$,
    which enables to control the denominators.
    \paragraph{First term}
    Since $x \mapsto \e^{- \alpha f(x)}$ is Lipschitz-continuous with constant $\alpha L_f \e^{ -\alpha \underline f}$, the function $g$ satisfies
    \begin{align}
        \left\lvert g(x) - g(y) \right\rvert
         & \leq |x - y| \bigl\lvert \e^{- \alpha f(x)} - Z_{\mu}\bigr\lvert + \bigl\lvert y - \mathcal M(\nu) \bigr \rvert \cdot \bigl\lvert \e^{- \alpha f(y)} - \e^{- \alpha f(x)} \bigr\rvert \\
         & \leq |x - y| \bigl\lvert \e^{- \alpha f(x)} - Z_{\mu}\bigr\lvert + \alpha L_f \e^{- \alpha \underline f} \bigl\lvert y - \mathcal M(\nu) \bigr \rvert  |x - y|.
    \end{align}
    Therefore, we deduce that
    \begin{align}
        \iint \bigl\lvert  g(x) - g(y) \bigr \rvert \pi(\d x \, \d y)
         & \leq
         \left( \iint \bigl\lvert  x - y \bigr \rvert^2 \pi(\d x \, \d y) \right)^{\frac{1}{2}}
         \left( \iint \lvert \e^{-\alpha f(x)} - Z_{\mu} \rvert^2 \pi(\d x \, \d y) \right)^{\frac{1}{2}} \\
         & \qquad
         + \alpha L_f \e^{- \alpha \underline f} \left( \iint \bigl\lvert  x - y \bigr \rvert^2 \pi(\d x \, \d y) \right)^{\frac{1}{2}}
         \left( \iint \lvert y - \mathcal M(\nu) \rvert^2  \pi(\d x \, \d y) \right)^{\frac{1}{2}}.
    \end{align}
    Infimizing over couplings, we deduce that
    \begin{align}
        \inf_{\pi \in \Pi(\mu, \nu)}
        \iint \bigl\lvert  g(x) - g(y) \bigr \rvert \pi(\d x \, \d y)
         & \leq \bra*{
             \left( \int \left\lvert \e^{-\alpha f(x)} - Z_{\mu} \right\rvert^2 \mu(\d x) \right)^{\frac{1}{2}}
         + \alpha L_f \e^{- \alpha \underline f} \sqrt{\momentp{2}{\nu}}} \wasserstein_2(\mu, \nu).
    \end{align}
    Recall the following classical inequality:
    for i.i.d.\ random vectors $X$ and $Y$ and any~$L_f$-globally Lipschitz function~$f$,
    it holds that
    \begin{align}
        \var \bigl( f(X) \bigr)
        =
        \frac{1}{2} \expect \left[  \bigl\lvert f(X) - f(Y)  \bigr\rvert^2 \right]
        \leq \frac{L_{f}^2}{2} \expect \left[  \left\lvert X - Y  \right\rvert^2 \right]
        = L_{f}^2 \expect  \left[\left\lvert X - \expect X \right\rvert^2\right].
    \end{align}
    Since $x \mapsto \e^{- \alpha f(x)}$ is Lipschitz-continuous with constant $\alpha L_f \e^{ -\alpha \underline f}$,
    it therefore holds that
    \[
        \int \left\lvert \e^{-\alpha f(x)} - Z_{\mu} \right\rvert^2 \, \mu(\d x)
        \leq \left( \alpha L_f \e^{- \alpha \underline f} \right)^2 \momentp{2}{\mu},
    \]
    and so
    \[
        \inf_{\pi \in \Pi(\mu, \nu)}
        \iint \bigl\lvert  g(x) - g(y) \bigr \rvert \, \pi(\d x \, \d y)
        \leq \alpha L_f \e^{- \alpha \underline f} \left( \sqrt{\momentp{2}{\mu}} + \sqrt{\momentp{2}{\nu}} \right) \wasserstein_2(\mu, \nu).
    \]

    \paragraph{Second term}
    It holds that
    \begin{align}
        \left\lvert \left( \frac{1}{Z_{\mu}} - \frac{1}{Z_{\nu}} \right) \int g(x) \, \nu(\d x) \right\rvert
         & \leq \e^{2 \alpha \overline f} \lvert Z_{\nu} - Z_{\mu} \rvert
         \int \bigl\lvert x - \mathcal M(\nu) \bigr \rvert \, \e^{- \alpha  f(x)} \, \nu(\d x)                                                                                                                               \\
         & \leq \e^{2 \alpha \overline f} \iint \left\lvert \e^{- \alpha f(x)} - \e^{- \alpha f(y)} \right\rvert \, \pi(\d x \, \d y)
         \, \e^{- \alpha \underline f} \int \bigl\lvert x - \mathcal M(\nu) \bigr \rvert \, \nu(\d x)                                                                                                                        \\
         & \leq \e^{2 \alpha \overline f - \alpha \underline f} \iint \alpha L_f \e^{-\alpha \underline f} \left\lvert x - y \right\rvert \, \pi(\d x \, \d y)  \int \bigl\lvert x - \mathcal M(\nu) \bigr \rvert \, \nu(\d x).
    \end{align}
    Using Jensen's inequality and infimizing over all couplings,
    we deduce that
    \[
        \left\lvert \left( \frac{1}{Z_{\mu}} - \frac{1}{Z_{\nu}} \right) \int g(x) \, \nu(\d x) \right\rvert
        \leq \alpha L_f \e^{2 \alpha\bra{ \overline f - \underline f}} \, \wasserstein_1 (\mu, \nu) \, \momentp{1}{\nu}.
    \]

    \paragraph{Concluding the proof}
    Gathering the bounds,
    we obtain
    \begin{align}
        \left\lvert \wmmu - \mathcal M(\mu) - \wmnu + \mathcal M(\nu) \right\rvert
        \le{}
        \alpha L_f \e^{2\alpha
            \bra{\overline f - \underline f}} \bra*{\bra*{ \sqrt{\momentp{2}{\mu}} + \sqrt{\momentp{2}{\nu}} } \wasserstein_2(\mu, \nu)
            +
            \momentp{1}{\nu}\,
        \wasserstein_1 (\mu, \nu) }.
    \end{align}
    Using that $\momentp{1}{\nu} \leq \sqrt{\momentp{2}{\nu}}$
    as well as $\wasserstein_1(\mu, \nu) \leq \wasserstein_2(\mu, \nu)$  we conclude the proof.
\end{proof}

\subsection{Monte Carlo estimate for the weighted mean}

\begin{lemma}
    [Convergence of the weighted mean for i.i.d.\ samples]
    \label{lem:convergence_weighted_mean_iid-2}
    Fix~$p \geq 2$.
    Suppose that $f$ satisfies \cref{assumption:bounded},
    and that $\mfldis \in \mathcal P_p(\real^d)$ has finite moments up to order~$p$.
    Then there exists a constant $\cwm{p}\bigl(\alpha, \overline{f}, \underline{f} \bigr)$ such that
    \begin{align}
        &\expect  \left\lvert
            \wmxl{}
            - \wm\bigl(\mfldis\bigr)
    \right\rvert_p^p
        \leq \cwm{p}
        \expect  \Bigl\lvert \xnl{1} - \expect \xnl{1} \Bigr\rvert_p^{p} J^{- \frac{p}{2}},
        \qquad \mu_{\overline {\mathcal X}^J} := \frac{1}{J}\sum_{j=1}^{J} \delta_{ \widebar{X}^j },
        \qquad
        \left\{ \xnl{j} \right\}_{j \in \N} \stackrel{\rm{i.i.d.}}{\sim} \mfldis \,,
    \end{align}
    where
    \[
        \cwm{p}\bigl(\alpha, \underline f, \overline f\bigr) :=
        \cmz{p} \e^{p\alpha (\overline f - \underline f)}
        \Bigl(1 + \e^{\frac{\alpha}{p} (\overline f - \underline f)} \Bigr)^p \,.
    \]
\end{lemma}

\begin{proof}
    Since~$f$ is bounded from above, we have
    \begin{align}
        \expect  \Bigl\lvert
            \wmxl{}
            - \wm\bigl(\mfldis\bigr)
        \Bigr\rvert_p^p
        &= \expect \left\lvert \frac{\frac{1}{J} \sum_{j=1}^{J} \Bigl(\xnl{j} - \wm\bigl(\mfldis\bigr) \Bigr) \e^{-\alpha f(\xnl{j})}}{\frac{1}{J} \sum_{j=1}^{J} \e^{-\alpha f(\xnl{j})}} \right\rvert_p^p \\
        &\leq \e^{p \alpha \overline f} \expect \left\lvert \frac{1}{J} \sum_{j=1}^{J} \Bigl(\xnl{j} - \wm\bigl(\mfldis\bigr) \Bigr) \e^{-\alpha f(\xnl{j})} \right\rvert_p^p.
    \end{align}
    Applying the Marcinkiewicz--Zygmund inequality to each component of the vector on the right-hand side, we deduce from Jensen's inequality
    \begin{align}
        \expect  \Bigl\lvert
            \wmxl{}
            - \wm\bigl(\mfldis\bigr)
        \Bigr\rvert_p^p
        &\leq \frac{\cmz{p}}{J^{\frac{p}{2}}} \e^{p\alpha \overline f} \expect \left\lvert \Bigl(\xnl{1} - \wm\bigl(\mfldis\bigr) \Bigr) \e^{-\alpha f(\xnl{1})} \right\rvert_p^p \\
        &\leq \frac{\cmz{p}}{J^{\frac{p}{2}}} \e^{p\alpha (\overline f - \underline f)} \expect \bigl\lvert \xnl{1} - \wm\bigl(\mfldis\bigr)  \bigr\rvert_p^p \,.
    \end{align}
    By the triangle inequality, we deduce that
    \begin{align}
        \left(\expect  \Bigl\lvert
            \wmxl{}
            - \wm\bigl(\mfldis\bigr)
        \Bigr\rvert_p^p\right)^{\frac{1}{p}}
        &\leq \frac{\cmz{p}^{\frac{1}{p}}}{\sqrt{J}}
        \e^{\alpha (\overline f - \underline f)}
        \biggl(
        \Bigl( \expect \bigl\lvert \xnl{1} - \expect \xnl{1} \bigr\rvert_p^p \Bigr)^{\frac{1}{p}}
        +
        \Bigl( \expect \bigl\lvert \expect{\xnl{1}} - \wm\bigl(\mfldis\bigr)  \bigr\rvert_p^p \Bigr)^{\frac{1}{p}}
        \biggr) \\
        &\leq \frac{\cmz{p}^{\frac{1}{p}}}{\sqrt{J}}
        \e^{\alpha (\overline f - \underline f)}
        \Bigl(1 + \e^{\frac{\alpha}{p} (\overline f - \underline f)} \Bigr)
        \Bigl( \expect \bigl\lvert \xnl{1} - \expect \xnl{1} \bigr\rvert_p^p \Bigr)^{\frac{1}{p}},
    \end{align}
    where we used~\cref{lem:m-malpha} in the last inequality.
    This implies the claim.
\end{proof}

\appendix
\section{The Burkholder-Davis-Gundy inequality}
\label{sec:bdg}

The Burkholder-Davis-Gundy inequality is used multiple times in this work,
and it is particularly useful to prove concentration inequalities for interacting particle systems.
For the reader's convenience,
and since we want to have dimension-independent convergence rates, we include it here.

\begin{theorem}
    [Burkholder--Davis--Gundy inequality, see Theorem 7.3 in~\cite{mao2007sde}]
    \label{thm:BDG_ineq}
    Let $(W_t)_{t \geq 0}$ denote a standard Brownian motion in $\real^m$ and let $(\mathcal F_t)_{t \geq 0}$ be the induced filtration.
    Let $(g_t)_{t \geq 0}$ be a $\real^{n \times m}$-valued $\mathcal F_t$-adapted process
    such that for every time~$T \geq 0$,
    it holds that $\int_{0}^{T} \bigl\lVert g(t) \bigr\rVert_{\rm F}^2 \, \d t < +\infty$ almost surely.
    Denote
    \[
        X_t := \int_0^t g(s) \, \d W_s
        \qquad \text{ and } \qquad
        \left\langle X\right\rangle_t := \int_0^t \bigl\lVert g(s) \bigr\rVert_{\rm F}^2 \, \d s.
    \]
    Then for all $p  > 0$, there exist positive constants $\csmallbdg{p}, \cbdg{p} < + \infty$ such that
    \begin{align}
        \forall t \geq 0, \qquad
        \csmallbdg{p} \E\left[ \left\langle X\right\rangle_t^{\frac{p}{2}} \right]
        \leq \E\left[ \sup_{0 \leq s \leq t} |X_s|^p \right]
        \leq \cbdg{p} \E\left[ \left\langle X\right\rangle_t^{\frac{p}{2}} \right].
    \end{align}
    The constants $\csmallbdg{p}$, $\cbdg{p}$ do not depend on any other parameters besides $p;$
    \begin{alignat}{3}
        \csmallbdg{p} &= \left(\frac{p}{2}\right)^p\,,
        \qquad &&\cbdg{p} = \left(\frac{32}{p}\right)^{\frac{p}{2}}
        \qquad && \text{ if } 0 < p < 2\,,\\
        \csmallbdg{p} &= 1\,,
        \qquad &&\cbdg{p} = 4
        \qquad && \text{ if } p = 2\,,\\
        \csmallbdg{p} &= (2p)^{-\frac{p}{2}}   \,,
        \qquad &&\cbdg{p} = \left( \frac{p^{p+1}}{2(p-1)^{p-1}} \right)^{\frac{p}{2}}
        \qquad && \text{ if } p > 2\,.
    \end{alignat}
\end{theorem}

\section{Constants used in this work}
\label{sec:constants}
\label{appendix:constants-table}

\cref{table:summary_constants} summarizes the constants that appear in the key inequalities used in this work, as well as their dependence on different parameters such as the method parameters $\bigl(\alpha, \sigma, \tau(S)\bigr)$ and the problem parameters~$(L_f, \overline f, \underline f)$.

\begin{table}[h!]
    \footnotesize
    \centering
    \begin{tabular}{c|c|c|p{8cm}}
        \toprule
        \textbf{Constant}
        & \textbf{Related result}
        & \textbf{Depends on}
        & \textbf{Mathematical expression} \\
        \midrule
          \phantom{$\Bigg($} $\tau(S)$ & & $S$ & see \eqref{eq:noise-prefactor} \\
        \phantom{$\Bigg($} $\cbdg{p}$
        & \cref{thm:BDG_ineq}
        & $p$
        & See exact expression in~\cref{thm:BDG_ineq}. \\
        \phantom{$\Bigg($} $\edecay{p}$
        & \cref{lem:decay-centered-moments}
        & $p$, $\sigma$, $\tau(S)$, $\alpha$, $\overline f$, $\underline f$
        & 
        \(
        \displaystyle
        p \left(1-  \frac{1}{2}\bigl(p-2+\tau(S) \bigr) \sigma^2 \Bigl(1 + \e^{\frac{\alpha}{p}(\overline f - \underline f)} \Bigr)^2\right)
        \)
        \\
        \phantom{$\Bigg($}
            $C_{\mathcal{M}}$
          & \cref{lem:stab_wmean}
          & $\alpha, L_f, \overline f, \underline f$
          &
          \(
          \displaystyle 2 \alpha L_f \e^{2 \alpha (\overline f - \underline f)}
          \)
          \\
          \phantom{$\Bigg($} $\craw{p}$
          & \cref{lem:uit-raw-moments-2norm}
          & $p, \sigma, \tau(S), \alpha, \overline f, \underline f$  
          &
          \(
          \displaystyle
          1 + \left( \frac{p}{\edecay{p}} \right)^{\frac{1}{p}} \Bigl(1 + \sigma \sqrt{\tau(S)} \cbdg{p}^{\frac{1}{p}}\Bigr) \left( 1 + \e^{\frac{\alpha}{p} (\overline f - \underline f)} \right) 
          \)
          \\
          \phantom{$\Bigg($} $\icbad{q}{\kappa}$
          & \cref{prop:bound-on-bad-set-general-noise}
          & $q, \kappa,\sigma, \alpha, \overline f, \underline f, \edecay{2q}, \cmz{2q}, \cbdg{q}$
          &$2^{3q-1} \cmz{2q}
        + 2^{4 q + 1} \cbdg{q} \sigma{^{q}}
        \left(\frac{q-2}{\edecay{2q } - q \kappa} \right)^{\frac{q}{2} - 1}
        \frac{
            \bra*{1 + \e^{\alpha  (\overline f - \underline f) }}^{\frac{1}{2}}
           }{\edecay{2q } - q \kappa}$
        \\
          \phantom{$\Bigg($} $\cbad{q}{\kappa}$
          & \cref{prop:bound-on-bad-set-synchronous-mf}
          & $q, \kappa,\sigma, \tau(S), \alpha, \overline f, \underline f, \edecay{2q}, \icbad{q}{\kappa}, C_{\rm WM, p}$
          &See expression from~\cref{prop:bound-on-bad-set-synchronous-mf}
    \\
          \phantom{$\Bigg($} $\cwm{p}$
          & \cref{lem:convergence_weighted_mean_iid-2}
          & $p, \alpha, \overline f, \underline f, \cmz{p}$
          &
          \(
          \displaystyle \cmz{p}
          \e^{p\alpha (\overline f - \underline f)}
          \Bigl(1 + \e^{\frac{\alpha}{p} (\overline f - \underline f)} \Bigr)^p
          \)
  \\
        \phantom{$\Bigg($} $C_{\rm Q}$ & Proof of \cref{eq:technical-inequality} & 
        $q, \kappa, \cbad{q}{\kappa}, \craw{8}, \rawmomentp{8}{\mfldis_0}, \momentp{2}{\mfldis_0}$
        & $2^{10} \cbad{4}{\kappa}^{\frac{1}{2}}
        \craw{8}^{2}
        \rawmomentp{8}{\mfldis_0} + 2 \Bigl( \momentp{2}{\mfldis_0} + 1 \Bigr)$
        \\
        \phantom{$\Bigg($}  $c_1$ & \cref{thm:uit-mfl} & $\kappa, \sigma, \tau(S), C_{\rm Q}, C_{\mathcal M}$  & $\cone$  \\
        \phantom{$\Bigg($}  $c_2$ &\cref{thm:uit-mfl} & $\kappa, \sigma, \tau(S), C_{\rm Q}, C_{\mathcal M}, \cwm{2}, \momentp{2}{\mfldis_0}  $   & 
        $\ctwo $ \\
        \phantom{$\Bigg($}  $ \tilde{c}_1$ & \cref{thm:stab_particle} & $ \sigma, \tau(S), \tilde{C}_{\rm Q}, \C_{\mathcal M}$  & $\ctildeone$  \\
        \phantom{$\Bigg($}  $\tilde{c}_2$ &\cref{thm:stab_particle} & $ \sigma, \tau(S), \tilde{C}_{\rm Q}, \C_{\mathcal M}$  & $\ctildetwo $ \\
        \phantom{$\Big($}
        $C_{\rm MFL}$
          & \cref{thm:uit-mfl,thm:microscopic-concentration-around-consensus-point}
          & $c_1, c_2 $
          & $ e^{ 2c_1} \cdot 2c_2 $ \\
        \phantom{$\Bigg($} $C_{\rm Stab, 1}$ &
        \cref{thm:stab_particle} &
        $c_1, \edecay{8}$ &
        $\exp \left( \frac{16 c_1}{\edecay{8} } \right) $
        \\
        \phantom{$\Bigg($} $C_{\rm Stab, 2}$ &
        \cref{thm:stab_particle} &
        $c_1, c_2, \edecay{8}$ &
        $\frac{16 c_2}{\edecay{8} } \exp \left( \frac{16 c_1}{\edecay{8} } \right) $\\
      \bottomrule
    \end{tabular}
    \caption{
        This table summarizes the constants that appear in key inequalities used in this work.
        Many of these depend on a number of variables, 
        including some of the method parameters $\bigl(\alpha, \sigma, \tau(S)\bigr)$ and problem parameters~$(L_f, \overline f, \underline f)$.
        The value these parameters are usually obvious from the context,
        so we usually suppress them in the notation.
        For example, 
        we write~$\cbad{q}{\kappa}$
        instead of $\cbad{q}{\kappa}(\sigma, \alpha, \overline f, \underline f)$
        and $\edecay{p}$  instead of~$\edecay{p}(\sigma, \tau(S), \alpha, \overline f, \underline f)$.
    }  
    \label{table:summary_constants}
\end{table}


\paragraph{Acknowledgments}
We thank the Lorentz Center (Leiden, the Netherlands) for hosting the workshop ``Purpose-driven particle systems'' in March 2023,
which initiated discussions about this work. In particular, we are grateful to Seung-Yeal Ha for suggesting this research direction, and for initial discussions.
FH is supported by start-up funds at the California Institute of Technology and by NSF CAREER Award 2340762.
UV is partially supported by the European Research Council (ERC) under the EU Horizon 2020 programme (grant agreement No 810367),
and by the Agence Nationale de la Recherche under grants ANR-21-CE40-0006 (SINEQ) and ANR-23-CE40-0027 (IPSO).

\printbibliography
\end{document}

%% file: preamble.tex
\newif\ifsiamart
\makeatletter%
\@ifclassloaded{siamart250211}%
{\siamarttrue}%
{\siamartfalse}%
\makeatother%

\ifsiamart
  \usepackage[margin=1.0in]{geometry}
\else
  \usepackage{authblk}
  \usepackage[margin=.6in]{geometry}
\fi
\usepackage{silence}
\usepackage[T1]{fontenc}
\usepackage[utf8]{inputenc}
\usepackage{csquotes}
\usepackage[UKenglish]{babel}
\usepackage{paralist}
\usepackage[shortlabels]{enumitem}
\usepackage{xcolor}
\usepackage{graphicx}
\usepackage{bbm}
\usepackage{amsmath}
\usepackage{amssymb}
\usepackage{amsfonts}
\usepackage{stmaryrd}
\usepackage{mathrsfs}
\usepackage{mathtools}
\usepackage{epstopdf}
\usepackage{comment}
\usepackage{verbatim}
\usepackage[color=yellow]{todonotes}
\usepackage{booktabs}
\usepackage{pifont}
\usepackage{makecell}
\usepackage{upgreek}

\usepackage{xspace}
\usepackage{float}
\usepackage{stackengine}

\ifsiamart
  \newcommand{\qedhere}{\hfill \square}
  \usepackage{thmtools}
  \usepackage{cleveref}
  \usepackage{autonum}
\else
  \usepackage{amsthm}
  \usepackage{hyperref}
  \hypersetup{%
    linktocpage=true,
    colorlinks=true,
    citecolor=darkgreen,
    linkcolor=blue,
    urlcolor=blue,
  }
  \usepackage[nameinlink,capitalise]{cleveref}
  \usepackage{autonum}
  \newcommand{\email}[1]{\href{mailto:#1}{#1}}
  \newenvironment{keywords}{\small \noindent\begin{quote}\textbf{Keywords.}}{\end{quote}}
  \newenvironment{MSCcodes}{\vspace{.2cm}\small \noindent \begin{quote}\textbf{MSC codes.}}{\end{quote}}
\fi

\ifsiamart
  \newsiamremark{remark}{Remark}
  \newsiamthm{assumption}{Assumption}
\else
  \theoremstyle{plain}
  
  \newtheorem{theorem}{Theorem}
  \newtheorem{lemma}[theorem]{Lemma}

  \newtheorem{remark}[theorem]{Remark}

  \newtheorem{assumption}{Assumption}
  
  \numberwithin{equation}{section}
  \numberwithin{theorem}{section}
  \crefname{lemma}{Lemma}{Lemmas}
  \crefname{remark}{Remark}{Remarks}
  \crefname{proposition}{Proposition}{Propositions}
  \crefname{section}{Section}{Sections}
  \crefname{subsection}{Subsection}{Subsections}
  \crefname{equation}{}{}
  \Crefname{equation}{Equation}{Equations}
  \Crefname{figure}{Figure}{Figures}
  \usepackage{setspace}
\fi
\crefname{remark}{Remark}{Remarks}
\crefname{assumption}{Assumption}{Assumptions}

\crefname{enumi}{}{}
\newcommand{\equalcontrib}{\textsuperscript{\textdagger}}

\input{widebar.tex}

\DeclarePairedDelimiter{\abs}{\lvert}{\rvert}
\DeclarePairedDelimiter{\norm}{\lVert}{\rVert}
\DeclarePairedDelimiter{\bra}{(}{)}
\DeclarePairedDelimiter{\pra}{[}{]}

\providecommand\st{}
\DeclarePairedDelimiterXPP\set[1]{}\{\}{}{
\renewcommand\st{\nonscript\:\delimsize\vert\nonscript\:\mathopen{}}
#1}
\DeclarePairedDelimiter{\scp}{\langle}{\rangle}

\DeclarePairedDelimiterXPP\condexpect[1]{\mathbb{E}}[]{}{
  
  #1}
\DeclarePairedDelimiterXPP\condexpectwithinitial[2]{\mathbb{E}_{#1}}[]{}{
  
  #2}

\DeclareMathOperator{\diag}{diag}

\DeclareMathOperator{\e}{e}

\DeclareMathOperator*{\argmin}{arg\,min}

\DeclareMathOperator{\supp}{supp}

\newcommand{\placeholder}{\mathord{\color{black!33}\bullet}}%
\newcommand{\range}[2]{\llbracket #1, #2 \rrbracket}

\renewcommand{\d}{\mathrm d}

\newcommand{\nm}{\mathcal M}
\newcommand{\wm}{\mathcal M_{\alpha}}

\newcommand{\C}{\mathbf{C}}

\newcommand{\N}{\mathbf{N}}
\newcommand{\R}{\mathbf R}
\newcommand{\real}{\R}
\newcommand{\nat}{\N}
\renewcommand{\t}{\mathsf T}

\renewcommand{\leq}{\leqslant}
\renewcommand{\geq}{\geqslant}
\renewcommand{\le}{\leqslant}
\renewcommand{\ge}{\geqslant}

\newcommand{\wasserstein}{\mathcal W}

\newcommand{\allx}[1]{\mathcal{X}^{#1}}
\newcommand{\allxl}[1]{\widebar {\mathcal{X}}^{#1}}
\newcommand{\allxt}[1]{\widetilde {\mathcal{X}}^{#1}}
\newcommand{\xnt}[1]{\widetilde X^{#1}}
\renewcommand{\d}{\mathrm d}
\newcommand{\wmx}[1]{\wm\bigl(\mu_{\mathcal X^J_{#1}}\bigr)}
\newcommand{\wmxl}[1]{\wm\bigl(\mu_{\widebar{\mathcal X}^J_{#1}}\bigr)}
\newcommand{\wmmu}{\wm\bigl(\mu\bigr)}
\newcommand{\wmnu}{\wm\bigl(\nu\bigr)}

\newcommand{\wmt}[1]{\wm\bigl(\mu_{\widetilde {\mathcal X}^J_{#1}}\bigr)}
\newcommand{\nmx}[1]{\mathcal M\bigl(\mu_{\mathcal X^J_{#1}}\bigr)}
\newcommand{\nmxl}[1]{\mathcal M\bigl(\mu_{\widebar{\mathcal X}^J_{#1}}\bigr)}
\newcommand{\nmt}[1]{\mathcal M\bigl(\mu_{\widetilde {\mathcal X}^J_{#1}}\bigr)}
\newcommand{\nml}[1]{\mathcal M\bigl(\mu_{\widebar{\mathcal X}^J_{#1}}\bigr)}

\newcommand{\emp}[1]{\mu_{\mathcal X^J_{#1}}}

\newcommand{\empl}[1]{\mu_{\widebar{\mathcal X}^J_{#1}}}

\newcommand{\empt}[1]{\mu_{\widetilde{\mathcal X}^J_{#1}}}

\newcommand{\proba}{\mathbf P}
\newcommand{\expect}{\mathbf{E}}
\newcommand{\E}{\mathbf{E}}

\newcommand{\var}{\mathbb V}
\newcommand{\tr}{{\rm trace}\,}

\newcommand{\meanx}[1]{\mathcal M\bigl(\mu_{\allx{J}_{#1}}\bigr)}
\newcommand{\wmeanx}[1]{\wm\bigl(\mu_{\allx{J}_{#1}}\bigr)}
\newcommand{\calE}{\mathcal E}
\newcommand{\wmeanmfl}[1]{\wm\bigl(\mfldis_{#1}\bigr)}
\newcommand{\meanmfl}[1]{\mathcal M\bigl(\mfldis_{#1}\bigr)}

\newcommand{\momentp}[2]{\mathfrak M_{#1}(#2)}

\newcommand{\empmomentp}[2]{\mathfrak M_{#1}\bigl(\emp{#2} \bigr)}
\newcommand{\emplmomentp}[2]{\mathfrak M_{#1}\bigl(\empl{#2}\bigr)}
\newcommand{\emptmomentp}[2]{\mathfrak M_{#1}\bigl(\empt{#2} \bigr)}

\newcommand{\rawmomentp}[2]{\mathfrak M^\circ_{#1}\bigl({#2} \bigr)}
\newcommand{\rawempmomentp}[2]{\mathfrak M^\circ_{#1}\bigl(\emp{#2} \bigr)}
\newcommand{\rawemplmomentp}[2]{\mathfrak M^\circ_{#1}\bigl(\empl{#2}\bigr)}
\newcommand{\rawemptmomentp}[2]{\mathfrak M^\circ_{#1}\bigl(\empt{#2}\bigr)}

\newcommand{\edecay}[1]{\lambda_{#1}}

\newcommand{\particleNumber}[1]{#1}
\newcommand{\xn}[1]{X^{\particleNumber{#1}}}
\newcommand{\zn}[1]{Z^{\particleNumber{#1}}}
\newcommand{\xnn}[2]{X^{\particleNumber{#1,#2}}}
\newcommand{\wn}[1]{W^{\particleNumber{#1}}}

\newcommand{\xl}{\widebar{X}}
\newcommand{\xnl}[1]{\widebar{X}^{\particleNumber{#1}}}
\newcommand{\mfldis}{\widebar{\rho}}

\usepackage{tikz}
\usetikzlibrary{shapes.misc}
\usetikzlibrary{positioning}

\cslet{blx@noerroretextools}\empty 
\usepackage[url=true,backend=biber,style=trad-abbrv,url=false,doi=false,isbn=false]{biblatex}
\DeclareFieldFormat{volume}{volume \textbf{#1}}
\DeclareFieldFormat[article]{volume}{\textbf{#1}}

\ifsiamart\else
  \let\oldparagraph=\paragraph
  \renewcommand\paragraph[1]{\oldparagraph{#1.}}
\fi

\newcommand{\eps}{\varepsilon}

\newcommand{\cptmfl}{x^\alpha_*}
\newcommand{\cpt}{\mathcal M^{\alpha,X_*}}

\definecolor{darkred}{rgb}{.7,0,0}
\definecolor{darkgreen}{rgb}{.1,.7,0}
\definecolor{SkyBlue}{HTML}{46C5DD}

\newcommand{\icbad}[2]{\widetilde C_{{\rm Bad},#1,#2}}
\newcommand{\cbad}[2]{C_{{\rm Bad},#1,#2}}

\newcommand{\craw}[1]{C_{{\rm Raw},#1}}
\newcommand{\cmz}[1]{C_{{\rm MZ},#1}}
\newcommand{\cwm}[1]{C_{{\rm WM},#1}}
\newcommand{\cbdg}[1]{C_{{\rm BDG},#1}}
\newcommand{\csmallbdg}[1]{c_{{\rm BDG},#1}}
\newcommand{\cone}{ \kappa^{-1} \bra*{2C_{\mathcal M}^2 C_{\rm Q} \Bigl(1 + 2 \tau(S) \sigma^2\Bigr)  + 2}}
\newcommand{\ctwo}{ \kappa^{-1}\Bigl( 2C_{\mathcal M}^2  C_{\rm Q} + \cwm{2} \momentp{2}{\mfldis_0}\Bigr) \Bigl(1 + 2 \tau(S) \sigma^2\Bigr)}
\newcommand{\ctildeone}{1 + 2 C_{\mathcal M}^2 \tilde{C}_{\rm Q} \left( 1 + \tau(S) \sigma^2 \right)}
\newcommand{\ctildetwo}{ 2 C_{\mathcal M}^2 \widetilde{C}_{\rm Q} \left( 1 + \tau(S)\sigma^2
  \right)}

\renewcommand{\var}{\operatorname{Var}}

\ifsiamart\else
  \AddToHook{env/lemma/begin}{\crefalias{theorem}{lemma}}
  \AddToHook{env/proposition/begin}{\crefalias{theorem}{proposition}}
  \AddToHook{env/assumption/begin}{\crefalias{theorem}{Assumption}}
  \AddToHook{env/remark/begin}{\crefalias{theorem}{remark}}
  \AddToHook{env/corollary/begin}{\crefalias{theorem}{corollary}}
\fi
\makeatletter
\AddToHook{cmd/appendix/before}{\def\cref@section@alias{appendix}}
\makeatother

\title{Uniform-in-time propagation of chaos for Consensus-Based Optimization}

\ifsiamart
  \author{%
    Nicolai Gerber\equalcontrib\thanks{%
      Institute of Applied Analysis, Ulm University, Germany
      (\email{nicolai.gerber@uni-bonn.de}).
    } \and
    Franca Hoffmann\equalcontrib\thanks{%
      Department of Computing and Mathematical Sciences, Caltech, USA (\email{franca.hoffmann@caltech.edu}).
    } \and
    Dohyeon Kim\equalcontrib\thanks{%
      Department of Computing and Mathematical Sciences, Caltech, USA (\email{dohyeon@caltech.edu}).
    } \and
    Urbain Vaes\equalcontrib\thanks{%
      MATHERIALS, Inria Paris \& CERMICS, \'Ecole des Ponts, France (\email{urbain.vaes@inria.fr})
    }
  }
\else
  \author[1]{Nicolai Gerber$^{a,}$\equalcontrib}
  \author[2]{Franca Hoffmann$^{b,}$\equalcontrib}
  \author[2]{Dohyeon Kim$^{c,}$\equalcontrib}
  \author[3,4]{Urbain Vaes$^{d,}$\equalcontrib}
  \affil[ ]{\footnotesize%
    $^a$\email{nicolai.gerber@uni-ulm.de},
    $^b$\email{franca.hoffmann@caltech.edu},
    $^c$\email{dohyeon@caltech.edu},
    $^d$\email{urbain.vaes@inria.fr}.
  }
  \affil[1]{\footnotesize Institute of Applied Analysis, Ulm University, Germany}
  \affil[2]{\footnotesize Department of Computing and Mathematical Sciences, Caltech, USA}
  \affil[3]{\footnotesize MATHERIALS team, Inria Paris, France}
  \affil[4]{\footnotesize CERMICS, \'Ecole des Ponts, France}
  \date{}
\fi